\documentclass[letterpaper,11pt]{amsart}

\usepackage[T1]{fontenc}
\usepackage[utf8]{inputenc}
\usepackage{lmodern}
\usepackage[english]{babel}
\usepackage[autostyle]{csquotes}

\usepackage[pdfusetitle,linktocpage,colorlinks=false]{hyperref}
\usepackage{doi}
\usepackage{filecontents}
\usepackage{enumitem}

\usepackage{amsmath,amssymb,amsthm,amscd}
\usepackage{pictexwd,dcpic}
\newtheorem{thm}{Theorem}[section]
\newtheorem{prop}[thm]{Proposition}
\newtheorem{lem}[thm]{Lemma}
\newtheorem{rem}[thm]{Remark}
\newtheorem{cor}[thm]{Corollary}
\newtheorem{ex}[thm]{Example}
\newtheorem{defn}[thm]{Definition}
\allowdisplaybreaks

\usepackage[colorinlistoftodos]{todonotes}

\newcommand{\diag}{\mathrm{diag}}

\AtEndDocument{\bigskip{\footnotesize%
  \textsc{Department of Mathematics, Texas A\&M University, College Station, TX 77843-3368, USA} \par  
  \textit{E-mail address}: \texttt{ycchung@math.tamu.edu} \par
}}

\begin{document}
\title{Quantitative $K$-Theory for Banach Algebras}
\author{Yeong Chyuan Chung}
\date{\today}
\maketitle

\begin{abstract}
Quantitative (or controlled) $K$-theory for $C^*$-algebras was used by Guoliang Yu in his work on the Novikov conjecture, and later developed more formally by Yu together with Herv\'{e} Oyono-Oyono. In this paper, we extend their work by developing a framework of quantitative $K$-theory for the class of algebras of bounded linear operators on subquotients (i.e., subspaces of quotients) of $L_p$ spaces. We also prove the existence of a controlled Mayer-Vietoris sequence in this framework. 
\end{abstract}

\tableofcontents

\section{Introduction}


Quantitative (or controlled) operator $K$-theory has its roots in \cite{Yu98}, where the idea was used by Yu in his work on the Novikov conjecture, which is a conjecture in topology on the homotopy invariance of certain higher signatures. In fact, by applying a certain controlled Mayer-Vietoris sequence to the Roe $C^*$-algebras associated to proper metric spaces with finite asymptotic dimension, Yu was able to prove the coarse Baum-Connes conjecture for these metric spaces, from which it follows that the Novikov conjecture holds for finitely generated groups with finite asymptotic dimension and whose classifying space has the homotopy type of a finite CW-complex. 

The underlying philosophy of quantitative $K$-theory is that the $K_0$ and $K_1$ groups of a complex Banach algebra can essentially be recovered by using ``quasi-idempotent'' and ``quasi-invertible'' elements respectively. A striking consequence of the flexibility gained by considering such elements instead of actual idempotent or invertible elements is that in place of closed (two-sided) ideals, which are often needed to apply the standard machinery in $K$-theory, we can sometimes use closed subalgebras that are almost ideals in an appropriate sense. This feature could already be seen in the controlled cutting and pasting technique used in \cite{Yu98}. 

In \cite{OY15}, Oyono-Oyono and Yu developed the theory for general filtered $C^*$-algebras, i.e., $C^*$-algebras equipped with a filtration, and they formulated a quantitative version of the Baum-Connes conjecture, proving it for a large class of groups. They also suggested that the theory can be extended to more general filtered Banach algebras. For most algebras of interest in noncommutative geometry, we can obtain a natural filtration from a length function defined on the algebra. This gives the algebra the necessary geometric structure in order to define quantitative $K$-theory, and we may regard these algebras as ``geometric'' algebras in the spirit of geometric group theory. It does seem, however, that one needs the algebras to be equipped with an appropriate matrix norm structure, such as a $p$-operator space structure, as one needs to incorporate norm control in the framework. This norm control is automatic in the $C^*$-algebra setting but not for arbitrary Banach algebras. Motivated by the successful application of quantitative $K$-theory in investigations of (variants of) the Baum-Connes conjecture \cite{Yu98,GWY2}, our goal in this paper is to develop a framework of quantitative $K$-theory that can be applied to filtered $L_p$ operator algebras, i.e., closed subalgebras of $B(L_p(X,\mu))$ for some measure space $(X,\mu)$, where $p\in[1,\infty)$. This will then give us a tool to investigate an $L_p$ version of the Baum-Connes conjecture, which in turn gives us information about the $K$-theory of certain classes of $L_p$ operator algebras.


In trying to extend techniques and results for $C^*$-algebras to more general Banach algebras, the algebras of bounded linear operators on $L_p$ spaces seem to be a natural class to begin with.	
Moreover, $L_p$ operator algebras have a natural $p$-operator space structure. From the point of view of noncommutative geometry, in particular the Baum-Connes conjecture \cite{BCH}, there are also reasons to study $L_p$ operator algebras. Indeed, in Lafforgue's work on the Baum-Connes conjecture \cite{Laf02,Skan}, he considered certain generalized Schwartz spaces whose elements act on (weighted) $L_p$ spaces. There is also the Bost conjecture, which is the Banach algebra analog of the Baum-Connes conjecture obtained by replacing the group $C^*$-algebra by the $L_1$ group convolution algebra. We refer the reader to \cite{LH} for a survey on the Baum-Connes conjecture and similar isomorphism conjectures. Also, in yet unpublished work \cite{KY}, Kasparov and Yu have been investigating an $L_p$ version of the Baum-Connes conjecture. 
One possible application of our framework of quantitative $K$-theory will be to the $L_p$ Baum-Connes conjecture with coefficients in $C(X)$ when we have a group acting with finite dynamic asymptotic dimension on $X$ as defined in \cite{GWY1}. This has been worked out in the $C^*$-algebraic setting in \cite{GWY2}, and we discuss the $L_p$ setting in \cite{Chung2}.

While we are mainly interested in $L_p$ operator algebras, the class of algebras of bounded linear operators on subquotients (i.e., subspaces of quotients) of $L_p$ spaces, which we will refer to as $SQ_p$ algebras, is a more natural class to work with, as suggested by the theory of $p$-operator spaces. Indeed, an abstract $p$-operator space (as defined in \cite{Daws10}) can be $p$-completely isometrically embedded in $B(E,F)$, where $E$ and $F$ are subquotients of $L_p$ spaces \cite{LM96}. Moreover, the class of $L_p$ operator algebras is not closed under taking quotients by closed ideals when $p\neq 2$ \cite{GardThiel16} while the class of $SQ_p$ algebras is \cite{LM96i}, which is relevant to us when we are considering short exact sequences of algebras. 

For general Banach algebras, one can still consider quantitative $K$-theory in a similar way as we do, provided one has an appropriate matrix norm structure on the Banach algebra. However, since notation becomes more cumbersome in this generality, and the application we have in mind involves only $L_p$ operator algebras, we develop our framework in the setting of $SQ_p$ algebras.


The outline of this paper is as follows: In section 2, we introduce the basic ingredients for quantitative $K$-theory. In section 3, we define the quantitative $K$-theory groups and describe some basic properties analogous to those in ordinary $K$-theory. We also explain how quantitative $K$-theory is related to ordinary $K$-theory. In section 4, we establish a controlled long exact sequence analogous to the long exact sequence in $K$-theory. In section 5, we prove the existence of a controlled Mayer-Vietoris sequence under appropriate decomposability and approximability assumptions. Finally, in section 6, we briefly explain how our framework essentially reduces to that in \cite{OY15} when dealing with $C^*$-algebras.

\emph{Acknowledgments.} We would like to thank Guoliang Yu for his patient guidance and support throughout this project, and also Herv\'{e} Oyono-Oyono for useful discussions. We are also grateful to an anonymous referee for carefully reading our manuscript and providing many valuable comments.


\section[Filtered Banach Algebras]{Filtered Banach Algebras, Quasi-Idempotents, and Quasi-Invertibles}

In this section, we introduce quasi-idempotents and quasi-invertibles in filtered Banach algebras, and we consider homotopy relations on these elements. These are the basic ingredients for our framework of quantitative $K$-theory. Throughout this paper, we will only consider complex Banach algebras.

\subsection{Filtered Banach algebras and $SQ_p$ algebras}

\begin{defn}
A filtered Banach algebra is a Banach algebra $A$ with a family $(A_r)_{r>0}$ of closed linear subspaces indexed by positive real numbers $r\in(0,\infty)$ such that
\begin{itemize}
\item $A_r\subseteq A_{r'}$ if $r\leq r'$;
\item $A_r A_{r'}\subseteq A_{r+r'}$ for all $r,r'>0$;
\item the subalgebra $\bigcup_{r> 0}A_r$ is dense in $A$.
\end{itemize}
If $A$ is unital with unit $1_A$, we require $1_A\in A_r$ for all $r> 0$. In this case, we set $A_0=\mathbb{C}1_A$. Elements of $A_r$ are said to have propagation $r$. The family $(A_r)_{r> 0}$ is called a filtration of $A$.
\end{defn}

When $A$ is a $C^*$-algebra, we also want $A_r^*=A_r$ for all $r>0$.

\begin{rem} \leavevmode
\begin{enumerate}
\item When $A=\mathbb{C}$, we will usually set $A_r=\mathbb{C}$ for all $r> 0$.
\item If $A$ is a filtered Banach algebra with filtration $(A_r)_{r>0}$, and $J$ is a closed ideal in $A$, then $A/J$ is a Banach algebra under the quotient norm, and has a filtration $((A_r+J)/J)_{r> 0}$.
\end{enumerate}
\end{rem}

\begin{ex}
Let $G$ be a countable discrete group equipped with a proper length function $l$, i.e., a function $l:G\rightarrow\mathbb{N}$ satisfying
\begin{itemize}
\item $l(g)=0$ if and only if $g=e$;
\item $l(gh)\leq l(g)+l(h)$ for all $g,h\in G$;
\item $l(g^{-1})=l(g)$ for all $g\in G$;
\item $\{g\in G:l(g)\leq r\}$ is finite for all $r\geq 0$.
\end{itemize}
Then 
\begin{enumerate}
\item The reduced group $C^*$-algebra, $C_\lambda^*(G)$, is a filtered $C^*$-algebra with a filtration given by \[(C_\lambda^*(G))_r=\{\sum a_g g:a_g\in\mathbb{C},l(g)\leq r\}.\]
\item Suppose that $G$ acts on a $C^*$-algebra $A$ by automorphisms. Then the reduced crossed product, $A\rtimes_\lambda G$, is a filtered $C^*$-algebra with a filtration given by \[(A\rtimes_\lambda G)_r=\{\sum a_g g:a_g\in A,l(g)\leq r\}.\]
\end{enumerate}
\end{ex}

Other examples of filtered $C^*$-algebras include finitely generated $C^*$-algebras and Roe algebras. One may also consider the $L_p$ analogs of the group $C^*$-algebra and crossed product, and these are examples of filtered Banach algebras. In fact, in each of these examples, one can define a length function $l:A\rightarrow[0,\infty]$ on the algebra $A$, satisfying the following conditions:
\begin{itemize}
\item $l(0)=0$ (or $l(1_A)=0$ if $A$ is unital);
\item $l(a+b)\leq\max(l(a),l(b))$ and $l(ab)\leq l(a)+l(b)$ for all $a,b\in A$;
\item $l(ca)\leq l(a)$ for any $a\in A$ and $c\in\mathbb{C}$;
\item the set $\{a\in A:l(a)<\infty\}$ is dense in $A$, and $\{a\in A:l(a)\leq r\}$ is a closed subset of $A$ for each $r\geq 0$.
\end{itemize}
This length function then gives rise to a natural filtration with \[A_r=\{a\in A:l(a)\leq r\}\] for each $r\geq 0$. Thus filtered Banach algebras may also be regarded as ``geometric'' Banach algebras in the spirit of geometric group theory.

If $A$ is a non-unital Banach algebra, let $A^+=\{(a,z):a\in A,z\in\mathbb{C}\}$ with multiplication given by $(a,z)(b,w)=(ab+z b+w a,zw)$. We call $A^+$ the unitization of $A$. We will use the notation \[\tilde{A}=\begin{cases} A & \text{if $A$ is unital,} \\ A^+ & \text{if $A$ is nonunital.} \end{cases}\]

Note that if $A$ is a unital Banach algebra, then we can always give it an equivalent Banach algebra norm such that $||1_A||=1$, namely the operator norm from the left regular representation of $A$ on itself. Thus we will always assume that $||1_A||=1$ when dealing with unital Banach algebras.


For our framework of quantitative $K$-theory, since we will consider matrices of all sizes simultaneously and we want to have norm control, we need our Banach algebras to have some matrix norm structure. 

\begin{defn}\cite{Daws10}
For $p\in[1,\infty)$, an abstract $p$-operator space is a Banach space $X$ together with a family of norms $||\cdot||_n$ on $M_n(X)$ satisfying:
\begin{enumerate}
\item[$\mathcal{D}_\infty$:] For $u\in M_n(X)$ and $v\in M_m(X)$, we have \[\biggl\Vert\begin{pmatrix} u & 0 \\ 0 & v \end{pmatrix}\biggr\Vert_{n+m}=\max(||u||_n,||v||_m);\]
\item[$\mathcal{M}_p$:] For $u\in M_m(X)$, $\alpha\in M_{n,m}(\mathbb{C})$, and $\beta\in M_{m,n}(\mathbb{C})$, we have \[||\alpha u\beta||_n\leq ||\alpha||_{B(\ell_p^m,\ell_p^n)} ||u||_m ||\beta||_{B(\ell_p^n,\ell_p^m)}.\]
\end{enumerate}
\end{defn}

In the general theory of $p$-operator spaces, one typically considers only $p\in(1,\infty)$ but the definition still makes sense when $p=1$, and the properties of $p$-operator spaces that we will use still hold in this case. We also mention Le Merdy's result \cite[Theorem 4.1]{LM96} that for $p\in(1,\infty)$, an abstract $p$-operator space $X$ can be $p$-completely isometrically embedded in $B(E,F)$ for some $E,F\in SQ_p$, where $SQ_p$ denotes the collection of subspaces of quotients of $L_p$ spaces. The class $SQ_2$ is precisely the class of Hilbert spaces while the class $SQ_1$ contains all Banach spaces since every Banach space is a quotient of some $L_1$ space.

Note that a $p$-operator space structure does not necessarily respect multiplicative structure, so given a filtered Banach algebra $A$ with $p$-operator space structure $\{||\cdot||_n\}_{n\in\mathbb{N}}$, we also require $(M_n(A),||\cdot||_n)$ to be a Banach algebra for each $n$. Here we note the following result of Le Merdy:

\begin{thm}\cite[Theorem 3.3]{LM96i}
If $A$ is a unital Banach algebra with a $p$-operator space structure $\{||\cdot||_n\}_{n\in\mathbb{N}}$ and $p\in(1,\infty)$, then the following are equivalent:
\begin{enumerate}
\item Each $(M_n(A),||\cdot||_n)$ is a Banach algebra.
\item $A$ is $p$-completely isometrically isomorphic to a subalgebra of $B(E)$ for some $E\in SQ_p$.
\end{enumerate}
\end{thm}

\begin{rem}
This theorem generalizes the result by Blecher, Ruan, and Sinclair in the $p=2$ case \cite[Theorem 3.1]{BRS90} (also see \cite[Theorem 2.1]{Ble95}). 
The $p=1$ case is omitted in Le Merdy's theorem but we note that at least (ii) $\Rightarrow$ (i) is valid, i.e., algebras of bounded linear operators on $SQ_1$ spaces have a canonical 1-operator space structure such that the matrix algebras are Banach algebras. This goes back to Kwapien \cite[Theorem 4.2']{Kwa72} (also see \cite[Theorem 3.2]{LM96}), and is sufficient for our purposes.
\end{rem}

\begin{defn}
If $E$ is a subquotient (i.e., subspace of a quotient) of an $L_p$ space for some $p\in[1,\infty)$, $B(E)$ is the algebra of bounded linear operators on $E$, and $A\subset B(E)$ is a norm-closed subalgebra, then we call $A$ an $SQ_p$ algebra.
\end{defn}

If $A\subset B(E)$ is a non-unital $SQ_p$ algebra, we view $A^+$ as $A+\mathbb{C}I_E$ so that $A^+$ is a unital $SQ_p$ algebra.

\begin{rem}
If $A$ is a filtered $SQ_p$ algebra, then
\begin{itemize}
\item $||I_n||_n=1$ for all $n\in\mathbb{N}$, where $I_n$ is the identity in $M_n(\tilde{A})$;
\item $||a_{kl}||_1\leq||(a_{ij})||_n\leq\sum_{i,j}||a_{ij}||_1$ for all $(a_{ij})\in M_n(\tilde{A})$, all $k,l\in\{1,\ldots,n\}$, and all $n\in\mathbb{N}$;
\item $M_n(A)$ is closed in $M_n(\tilde{A})$ for each $n$;
\item if $A$ is unital and has a filtration $(A_r)_{r>0}$, then $M_n(A)$ has a filtration $(M_n(A_r))_{r>0}$ for each $n$;
\item if $A$ is non-unital and has a filtration $(A_r)_{r>0}$, then $M_n(A^+)$ has a filtration $(M_n(A_r+\mathbb{C}))_{r>0}$ for each $n$, and $M_n(A)$ has a filtration $(M_n(A_r))_{r>0}$ for each $n$.
\end{itemize}
\end{rem}

Given a filtered $SQ_p$ algebra, other Banach algebras constructed from $A$ will be given matrix norm structures naturally induced by the matrix norm structure on $A$. We now consider some of these constructions.

In some situations, we may want to adjoin a unit to $A$ even if $A$ is already unital. In this case, $A^+$ is isomorphic to $A\oplus\mathbb{C}$ via the homomorphism $A^+\ni(a,z)\mapsto(a+z,z)\in A\oplus\mathbb{C}$, and we will identify $M_n(A^+)$ isometrically with $M_n(A)\oplus M_n(\mathbb{C})$, where $M_n(\mathbb{C})$ is viewed as $B(\ell_p^n)$ and the norm on the direct sum is given by $||(x,y)||=\max(||x||,||y||)$. 


\begin{rem}
Matrix algebras of the form $M_k(M_n(A)^+)$ will be viewed as subalgebras of $M_{kn}(A^+)$ by identifying $(a,z)\in M_n(A)^+$ with $(a,z I_n)\in M_n(A^+)$ for $a\in M_n(A)$ and $z\in\mathbb{C}$.
\end{rem}

If $A$ is a Banach algebra, and $J$ is a closed ideal in $A$, then $J$ is a closed ideal in $\tilde{A}$. We will view $\widetilde{A/J}$ as $\tilde{A}/J$, and we will identify $M_n(\tilde{A}/J)$ isometrically with $M_n(\tilde{A})/M_n(J)$ equipped with the quotient norm.


Given a Banach algebra $A$, recall that the cone of $A$ is defined as $CA=C_0((0,1],A)$, and the suspension of $A$ is defined as $SA=C_0((0,1),A)$. These are closed subalgebras of $C([0,1],A)$ with the supremum norm. We will sometimes denote $C([0,1],A)$ by $A[0,1]$.

\begin{lem}
If $A$ is a filtered Banach algebra, then $C([0,1],A)$ is a filtered Banach algebra with filtration $(C([0,1],A_r))_{r> 0}$. This induces filtrations on $CA$ and $SA$. 
\end{lem}

\begin{proof}
It is clear that $(C([0,1],A_r))_{r>0}$ satisfies the first two conditions in the definition so we just need to prove the density condition. Given $\varepsilon>0$ and $f\in C([0,1],A)$, let $0=t_0<t_1<\cdots<t_k=1$ be such that whenever $s,t\in[t_{i-1},t_i]$, we have $||f(s)-f(t)||<\frac{\varepsilon}{6}$. For each $i$, let $a_i\in A_{r_i}$ be such that $||a_i-f(t_i)||<\frac{\varepsilon}{6}$. Let $r=\max_{0\leq i\leq k}r_i$. Define $g(s)=\frac{s-t_{i-1}}{t_i-t_{i-1}}a_i+\frac{t_i-s}{t_i-t_{i-1}}a_{i-1}$ for $s\in[t_{i-1},t_i]$. Then $g(s)\in A_r$ for all $s\in[0,1]$ so $g\in C([0,1],A_r)$. Moreover, for each $i$ and for all $s\in[t_{i-1},t_i]$, we have 
\begin{align*}||g(s)-f(s)|| &\leq ||g(s)-a_i||+||a_i-f(t_i)||+||f(t_i)-f(s)|| \\ &\leq ||a_i-a_{i-1}||+||a_i-f(t_i)||+||f(t_i)-f(s)|| \\ &\leq ||a_i-f(t_i)||+||f(t_i)-f(t_{i-1})||+||f(t_{i-1})-a_{i-1}|| \\ &\quad\quad +||a_i-f(t_i)||+||f(t_i)-f(s)|| \\ &< \varepsilon.
\end{align*}
If $f\in C([0,1],A)$ and $f(0)=0$, then we may take $a_0=0$ so that $g(0)=0$. Likewise, if $f(1)=0$, then we may take $a_k=0$ so that $g(1)=0$. It follows that $CA$ and $SA$ are filtered Banach algebras. 
\end{proof}


%

We will view $M_n(\widetilde{A[0,1]})$ as a subalgebra of $M_n(\tilde{A})[0,1]$ for each $n$, and similarly for $M_n(\widetilde{SA})$ and $M_n(\widetilde{CA})$. 

%
%

\begin{defn}\cite{Pi90}
Let $X$ and $Y$ be $p$-operator spaces, and let $\phi:X\rightarrow Y$ be a bounded linear map. For each $n\in\mathbb{N}$, let $\phi_n:M_n(X)\rightarrow M_n(Y)$ be the induced map given by $\phi_n([x_{ij}])=[\phi(x_{ij})]$. We say that $\phi$ is $p$-completely bounded if $\sup_n||\phi_n||<\infty$. In this case, we let $||\phi||_{pcb}=\sup_n||\phi_n||$.

We say that $\phi$ is $p$-completely contractive if $||\phi||_{pcb}\leq 1$, and $\phi$ is $p$-completely isometric if $||\phi||_{pcb}=1$.
\end{defn}


The following lemma was proved for $p\in(1,\infty)$ but the proof remains valid when $p=1$.

\begin{lem}\cite[Lemma 4.2]{Daws10}
Let $X$ be a $p$-operator space, and let $\mu$ be a bounded linear functional on $X$. Then $\mu$ is $p$-completely bounded as a map to $\mathbb{C}$, and $||\mu||_{pcb}=||\mu||$.
\end{lem}

Since all characters on Banach algebras are contractive, we get the following

\begin{cor}
If $A$ is a non-unital $SQ_p$ algebra, then the canonical homomorphism $\pi:A^+\rightarrow\mathbb{C}$ is $p$-completely contractive.
\end{cor}

\begin{defn}
Let $A$ and $B$ be filtered $SQ_p$ algebras with filtrations $(A_r)_{r> 0}$ and $(B_r)_{r> 0}$ respectively. A filtered homomorphism $\phi:A\rightarrow B$ is an algebra homomorphism such that
\begin{itemize}
\item $\phi$ is $p$-completely bounded;
\item $\phi(A_r)\subset B_r$ for all $r> 0$.
\end{itemize}
\end{defn}

Recall that any bounded homomorphism $\phi:A\rightarrow B$ between Banach algebras induces a bounded homomorphism $\phi^+:A^+\rightarrow B^+$ given by $\phi^+(a,z)=(\phi(a),z)$. If $\phi$ is a filtered homomorphism, then so is $\phi^+$.
We will see later that filtered homomorphisms induce homomorphisms between quantitative $K$-theory groups.


\subsection{Quasi-idempotent elements and quasi-invertible elements}

Elements of $K$-theory groups of Banach algebras are equivalence classes of idempotents or invertibles. For quantitative $K$-theory, we will consider quasi-idempotents and quasi-invertibles. We now define these elements, and give a list of simple results based on norm estimates that we will often use to conclude that two such elements represent the same class in some quantitative $K$-theory group.

\begin{defn}
Let $A$ be a filtered Banach algebra. For $0< \varepsilon<\frac{1}{20}$, $r> 0$, and $N\geq 1$, 
\begin{itemize}
\item an element $e\in A$ is called an $(\varepsilon,r,N)$-idempotent if $||e^2-e||<\varepsilon$, $e\in A_r$, and $\max(||e||,||1_{\tilde{A}}-e||)\leq N$;
\item if $A$ is unital, an element $u\in A$ is called an $(\varepsilon,r,N)$-invertible if $u\in A_r$, $||u||\leq N$, and there exists $v\in A_r$ with $||v||\leq N$ such that $\max(||uv-1||,||vu-1||)<\varepsilon$.
\end{itemize}
We call $v$ an $(\varepsilon,r,N)$-inverse for $u$, and we call $(u,v)$ an $(\varepsilon,r,N)$-inverse pair. 

We will use the terms quasi-idempotent, quasi-invertible, quasi-inverse, and quasi-inverse pair when the precise parameters are not crucial.
\end{defn}

Observe that if $A$ is unital and $a\in A_r$ satisfies $||a||\leq N$ and $||a-1||<\varepsilon$, then $(a,1)$ is an $(\varepsilon,r,N)$-inverse pair. Also, if $u\in A$ is an $(\varepsilon,r,N)$-invertible, then $u$ is invertible and $||u^{-1}||<\frac{N}{1-\varepsilon}$. Indeed, if $v$ is an $(\varepsilon,r,N)$-inverse for $u$, then $uv$ and $vu$ are invertible so $u$ is invertible and \[||u^{-1}||=||v(uv)^{-1}||<\frac{N}{1-\varepsilon}.\]

\begin{lem}\label{normestlem3}
Let $A$ be a filtered Banach algebra.
\begin{enumerate}
\item Let $e\in A$ be an idempotent, and let $N=||e||+1$. If $a\in A_r$ and $||a-e||<\frac{\varepsilon}{3N}$, then $a$ is an $(\varepsilon,r,N)$-idempotent in $A$.
\item Suppose that $A$ is unital. Let $u_0\in A$ be invertible, and let $N=||u_0||+||u_0^{-1}||+1$. If $u,v\in A_r$ and $\max(||u-u_0||,||v-u_0^{-1}||)<\frac{\varepsilon}{N}$, then $(u,v)$ is an $(\varepsilon,r,N)$-inverse pair in $A$.
\end{enumerate}
\end{lem}

\begin{proof}\leavevmode
\begin{enumerate}
\item $||a^2-a||\leq||a(a-e)||+||(a-e)e||+||e-a||<(2N+1)\frac{\varepsilon}{3N}\leq\varepsilon$.
\item $||uv-1||\leq||(u-u_0)v||+||u_0(v-u_0^{-1})||<\varepsilon$; similarly $||vu-1||<\varepsilon$.
\end{enumerate}
\end{proof}

\begin{defn}
Let $A$ be a filtered Banach algebra.
\begin{itemize}
\item Two $(\varepsilon,r,N)$-idempotents $e_0$ and $e_1$ in $A$ are $(\varepsilon',r',N')$-homotopic for some $\varepsilon'\geq\varepsilon,r'\geq r$, and $N'\geq N$ if there exists a norm-continuous path $(e_t)_{t\in[0,1]}$ of $(\varepsilon',r',N')$-idempotents in $A$ from $e_0$ to $e_1$. 
Equivalently, there exists an $(\varepsilon',r',N')$-idempotent $e$ in $A[0,1]$ such that $e(0)=e_0$ and $e(1)=e_1$. 
\item If $A$ is unital, two $(\varepsilon,r,N)$-invertibles $u_0$ and $u_1$ in $A$ are $(\varepsilon',r',N')$-homotopic for some $\varepsilon'\geq\varepsilon,r'\geq r$, and $N'\geq N$ if there exists an $(\varepsilon',r',N')$-invertible $u$ in $A[0,1]$ with $u(0)=u_0$ and $u(1)=u_1$. In this case, we get a norm-continuous path $(u_t)_{t\in[0,1]}$ of $(\varepsilon',r',N')$-invertibles in $A$ from $u_0$ to $u_1$ by setting $u_t=u(t)$.
\end{itemize}
We also call the paths $(e_t)$ and $(u_t)$ homotopies of $(\varepsilon',r',N')$-idempotents and $(\varepsilon',r',N')$-invertibles respectively, and we write $e_0\stackrel{\varepsilon',r',N'}{\sim}e_1$ and $u_0\stackrel{\varepsilon',r',N'}{\sim}u_1$.
\end{defn}

\begin{rem}
If $u_0$ and $u_1$ are homotopic as $(\varepsilon,r,N)$-invertibles, then the definition yields \emph{some} $(\varepsilon,r,N)$-inverse of $u_0$ that is $(\varepsilon,r,N)$-homotopic to \emph{some} $(\varepsilon,r,N)$-inverse of $u_1$. If $v_0$ and $v_1$ are arbitrarily chosen $(\varepsilon,r,N)$-inverses of $u_0$ and $u_1$ respectively, then $v_0$ and $v_1$ are only $(2\varepsilon,r,N)$-homotopic in general. Indeed, let $(u_t)$ be a homotopy of $(\varepsilon,r,N)$-invertibles between $u_0$ and $u_1$, and let $v_t$ be an $(\varepsilon,r,N)$-inverse for $u_t$ for each $t\in[0,1]$. Let $0=t_0<t_1<\cdots<t_k=1$ be such that $||u_t-u_s||<\frac{\varepsilon}{N}$ for all $s,t\in[t_{i-1},t_i]$ and $i=1,\ldots,k$. Let \[w_t=\frac{t_i-t}{t_i-t_{i-1}}v_{t_{i-1}}+\frac{t-t_{i-1}}{t_i-t_{i-1}}v_{t_i}\] if $t\in[t_{i-1},t_i]$. Then $\max(||u_tw_t-1||,||w_tu_t-1||)<2\varepsilon$ for all $t\in[0,1]$ so $(w_t)$ is a homotopy of $(2\varepsilon,r,N)$-invertibles between $v_0$ and $v_1$. Moreover, setting $w(t)=w_t$ for $t\in[0,1]$ defines an element $w\in A[0,1]$ such that $w(0)=v_0$, $w(1)=v_1$, and $w$ is a $(2\varepsilon,r,N)$-inverse of $u=(u_t)_{t\in[0,1]}$ in $A[0,1]$. 
\end{rem}

\begin{lem}\label{normestlem1}
Let $A$ be a filtered Banach algebra.
\begin{enumerate}
\item If $e$ is an $(\varepsilon,r,N)$-idempotent in $A$, and $f\in A_r$ satisfies $||f||\leq N$ and $||e-f||<\frac{\varepsilon-||e^2-e||}{2N+1}$, then $f$ is a quasi-idempotent that is $(\varepsilon,r,N)$-homotopic to $e$. More generally, if $||e-f||<\delta$ for some $\delta>0$, then $e$ and $f$ are $((2N+1)\delta+\varepsilon,r,N)$-homotopic.
\item Suppose $A$ is unital. If $(u,v)$ is an $(\varepsilon,r,N)$-inverse pair in $A$, and $a\in A_r$ is such that $||a||\leq N$ and $||u-a||<\frac{\varepsilon-\max(||uv-1||,||vu-1||)}{N}$, then $a$ and $u$ are $(\varepsilon,r,N)$-homotopic, with $v$ being an $(\varepsilon,r,N)$-inverse of $a$. More generally, if $||u-a||<\delta$ for some $\delta>0$, then $a$ and $u$ are $(N\delta+\varepsilon,r,N)$-homotopic, with $v$ being an $(N\delta+\varepsilon,r,N)$-inverse of $a$.
\end{enumerate}
\end{lem}

\begin{proof}\leavevmode
\begin{enumerate}
	\item The following estimate shows that $f$ is a quasi-idempotent. \begin{align*} ||f^2-f||&\leq||f(f-e)||+||(f-e)e||+||e^2-e||+||e-f||\\ &<(||f||+||e||+1)||e-f||+||e^2-e||\\ &\leq(2N+1)||e-f||+||e^2-e||.\end{align*}
    
    The statement about homotopy follows from the observation that $||(1-t)e+tf||\leq N$ and $||((1-t)e+tf)-e||\leq||e-f||$ for $t\in[0,1]$.
    
    \item $||av-1||\leq ||a-u|| ||v||+||uv-1||\leq N||a-u||+||uv-1||$. Similarly, $||va-1||\leq N||a-u||+||vu-1||$.
    
    The statement about homotopy follows from the observation that $||(1-t)a+tu||\leq N$ and $||((1-t)a+tu)-u||\leq||u-a||$ for $t\in[0,1]$.
\end{enumerate}
\end{proof}

\begin{lem}\label{normestlem2}
Let $A$ be a filtered Banach algebra. If $e$ and $f$ are $(\varepsilon,r,N)$-idempotents in $A$, then $e$ and $f$ are $(\varepsilon',r,N)$-homotopic, where $\varepsilon'=\varepsilon+\frac{1}{4}||e-f||^2$.
\end{lem}

\begin{proof}
If $e$ and $f$ are $(\varepsilon,r,N)$-idempotents, then for $t\in[0,1]$, we have
\begin{align*}
&\quad ((1-t)e+tf)^2-((1-t)e+tf) \\ &= (1-t)(e^2-e)+(t^2-t)e^2+t(f^2-f)+(t^2-t)f^2+(t-t^2)(ef+fe) \\ &= (1-t)(e^2-e)+t(f^2-f)+(t^2-t)(e-f)^2
\end{align*}
so $||((1-t)e+tf)^2-((1-t)e+tf)||<\varepsilon+\frac{1}{4}||e-f||^2$.
\end{proof}



\begin{lem}\label{orthidem}
Let $A$ be a filtered $SQ_p$ algebra, and suppose that $e,f\in A$ are orthogonal $(\varepsilon,r,N)$-idempotents (i.e., $ef=fe=0$). Then $e+f$ is a $(2\varepsilon,r,2N)$-idempotent. Moreover, $\begin{pmatrix} e+f & 0 \\ 0 & 0 \end{pmatrix}$ and $\begin{pmatrix} e & 0 \\ 0 & f \end{pmatrix}$ are $(\frac{5}{2}\varepsilon,r,\frac{5}{2}N)$-homotopic in $M_2(A)$.
\end{lem}

\begin{proof}
Since $(e+f)^2=e^2+f^2$, it follows that $e+f$ is a $(2\varepsilon,r,2N)$-idempotent. For $t\in[0,1]$, let $c_t=\cos\frac{\pi t}{2}$ and $s_t=\sin\frac{\pi t}{2}$. Then \[E_t=\begin{pmatrix} e & 0 \\ 0 & 0 \end{pmatrix}+\begin{pmatrix} c_t & -s_t \\ s_t & c_t \end{pmatrix}\begin{pmatrix} f & 0 \\ 0 & 0 \end{pmatrix}\begin{pmatrix} c_t & s_t \\ -s_t & c_t \end{pmatrix}\] is a homotopy of $(\frac{5}{2}\varepsilon,r,\frac{5}{2}N)$-idempotents in $M_2(A)$ between $\begin{pmatrix} e+f & 0 \\ 0 & 0 \end{pmatrix}$ and $\begin{pmatrix} e & 0 \\ 0 & f \end{pmatrix}$.
\end{proof}

Let $A$ be a unital filtered $SQ_p$ algebra. It is straightforward to see that if $(u,v)$ is an $(\varepsilon,r,N)$-inverse pair in $A$, then 
\begin{enumerate}
\item $\left(\begin{pmatrix} u & 0 \\ 0 & v \end{pmatrix},\begin{pmatrix} v & 0 \\ 0 & u \end{pmatrix}\right)$ is an $(\varepsilon,r,N)$-inverse pair in $M_2(A)$;
\item $\left(\begin{pmatrix} uv & 0 \\ 0 & 1 \end{pmatrix},\begin{pmatrix} 1 & 0 \\ 0 & vu \end{pmatrix}\right)$ is an $(\varepsilon,2r,(1+\varepsilon))$-inverse pair in $M_2(A)$;
\item $\left(\begin{pmatrix} u & 0 \\ 0 & 1 \end{pmatrix},\begin{pmatrix} v & 0 \\ 0 & 1 \end{pmatrix}\right)$ is an $(\varepsilon,r,N)$-inverse pair in $M_2(A)$.
\end{enumerate}

\begin{lem} \label{inversepairhomotopy}
Let $A$ be a unital filtered $SQ_p$ algebra. If $(u,v)$ is an $(\varepsilon,r,N)$-inverse pair in $A$, then $\begin{pmatrix} u & 0 \\ 0 & v \end{pmatrix}$ and $I_2$ are $(2\varepsilon,2r,2(N+\varepsilon))$-homotopic in $M_2(A)$.
\end{lem}

\begin{proof} 
For $t\in[0,1]$, let $c_t=\cos\frac{\pi t}{2}$ and $s_t=\sin\frac{\pi t}{2}$. Define 
\begin{align*} U_t&=\begin{pmatrix} u & 0 \\ 0 & 1 \end{pmatrix}\begin{pmatrix} c_t & -s_t \\ s_t & c_t \end{pmatrix}\begin{pmatrix} 1 & 0 \\ 0 & v \end{pmatrix}\begin{pmatrix} c_t & s_t \\ -s_t & c_t \end{pmatrix}\\ &=c_t^2\begin{pmatrix} u & 0 \\ 0 & v \end{pmatrix}+s_t^2\begin{pmatrix} uv & 0 \\ 0 & 1 \end{pmatrix}+c_ts_t\begin{pmatrix} 0 & u-uv \\ 1-v & 0 \end{pmatrix}
,\end{align*}
\begin{align*} V_t&=\begin{pmatrix} c_t & -s_t \\ s_t & c_t \end{pmatrix}\begin{pmatrix} v & 0 \\ 0 & 1 \end{pmatrix}\begin{pmatrix} c_t & s_t \\ -s_t & c_t \end{pmatrix}\begin{pmatrix} 1 & 0 \\ 0 & u \end{pmatrix}\\ &=c_t^2\begin{pmatrix} v & 0 \\ 0 & u \end{pmatrix}+s_t^2\begin{pmatrix} 1 & 0 \\ 0 & vu \end{pmatrix}+c_ts_t\begin{pmatrix} 0 & vu-u \\ v-1 & 0 \end{pmatrix}
.\end{align*} 

Note that for each  $t\in[0,1]$, we have 
\begin{align*}
U_t V_t &= \begin{pmatrix} uv & 0 \\ 0 & vu \end{pmatrix}, \\
V_tU_t &= \begin{pmatrix} c_t & -s_t \\ s_t & c_t \end{pmatrix}\begin{pmatrix} vu & 0 \\ 0 & uv \end{pmatrix}\begin{pmatrix} c_t & s_t \\ -s_t & c_t \end{pmatrix}, \\
\max(||U_t||,||V_t||) &\leq c_t^2N+s_t^2(1+\varepsilon)+c_ts_t(N+1+\varepsilon) \\ &<(c_t+s_t)^2(N+\varepsilon) \\ &\leq 2(N+\varepsilon)
\end{align*} 
so $(U_t,V_t)$ is a $(2\varepsilon,2r,2(N+\varepsilon))$-inverse pair in $M_2(A)$ for each $t\in[0,1]$. In particular, $(U_t)_{t\in[0,1]}$ is a homotopy of $(2\varepsilon,2r,2(N+\varepsilon))$-invertibles between $\begin{pmatrix} u & 0 \\ 0 & v \end{pmatrix}$ and $\begin{pmatrix} uv & 0 \\ 0 & 1 \end{pmatrix}$, while $(V_t)_{t\in[0,1]}$ is a homotopy of $(2\varepsilon,2r,2(N+\varepsilon))$-invertibles between $\begin{pmatrix} v & 0 \\ 0 & u \end{pmatrix}$ and $\begin{pmatrix} 1 & 0 \\ 0 & vu \end{pmatrix}$. Then $\begin{pmatrix} (1-t)uv+t & 0 \\ 0 & 1 \end{pmatrix}$ is a homotopy of $(\varepsilon,2r,(1+\varepsilon))$-invertibles between $\begin{pmatrix} uv & 0 \\ 0 & 1 \end{pmatrix}$ and $I_2$.
\end{proof}


\begin{rem}
From the proof we see that if $u^{-1}\in A_r$ and $||u^{-1}||\leq N$, then we may take $v=u^{-1}$ so that $(U_t)_{t\in[0,1]}$ is a homotopy between $\begin{pmatrix} u & 0 \\ 0 & u^{-1} \end{pmatrix}$ and $I_2$ with $U_t V_t=V_t U_t=I_2$ for each $t$. Moreover, in this case, we have $||U_t||\leq 2N$ and $||V_t||\leq 2N$.
\end{rem}

\begin{lem} \label{invcomm}
Let $A$ be a unital filtered $SQ_p$ algebra. If $u,v\in A$ are $(\varepsilon,r,N)$-invertibles, then $\begin{pmatrix} u & 0 \\ 0 & v \end{pmatrix}$ and $\begin{pmatrix} v & 0 \\ 0 & u \end{pmatrix}$ are $(2\varepsilon,r,2N)$-homotopic in $M_2(A)$.
\end{lem}

\begin{proof}
Suppose that $(u,u')$ and $(v,v')$ are $(\varepsilon,r,N)$-inverse pairs. For $t\in[0,1]$, let $c_t=\cos\frac{\pi t}{2}$ and $s_t=\sin\frac{\pi t}{2}$. Define 
\begin{align*} U_t&=\begin{pmatrix} c_t & -s_t \\ s_t & c_t \end{pmatrix}\begin{pmatrix} u & 0 \\ 0 & v \end{pmatrix}\begin{pmatrix} c_t & s_t \\ -s_t & c_t \end{pmatrix}\\ &=c_t^2\begin{pmatrix} u & 0 \\ 0 & v \end{pmatrix}+s_t^2\begin{pmatrix} v & 0 \\ 0 & u \end{pmatrix}+c_ts_t\begin{pmatrix} 0 & u-v \\ u-v & 0 \end{pmatrix}
,\end{align*}
\begin{align*} V_t&=\begin{pmatrix} c_t & -s_t \\ s_t & c_t \end{pmatrix}\begin{pmatrix} u' & 0 \\ 0 & v' \end{pmatrix}\begin{pmatrix} c_t & s_t \\ -s_t & c_t \end{pmatrix}\\ &=c_t^2\begin{pmatrix} u' & 0 \\ 0 & v' \end{pmatrix}+s_t^2\begin{pmatrix} v' & 0 \\ 0 & u' \end{pmatrix}+c_ts_t\begin{pmatrix} 0 & u'-v' \\ u'-v' & 0 \end{pmatrix}
.\end{align*}
Then $||U_t||\leq 2N$ and $||V_t||\leq 2N$.
Also, \[U_tV_t-I=\begin{pmatrix} c_t & -s_t \\ s_t & c_t \end{pmatrix}\begin{pmatrix} uu'-1 & 0 \\ 0 & vv'-1 \end{pmatrix}\begin{pmatrix} c_t & s_t \\ -s_t & c_t \end{pmatrix}\] and \[V_tU_t-I=\begin{pmatrix} c_t & -s_t \\ s_t & c_t \end{pmatrix}\begin{pmatrix} u'u-1 & 0 \\ 0 & v'v-1 \end{pmatrix}\begin{pmatrix} c_t & s_t \\ -s_t & c_t \end{pmatrix}\] so $\max(||U_tV_t-I||,||V_tU_t-I||)<2\varepsilon$.
\end{proof}



It is a standard fact in $K$-theory for Banach algebras that if we consider matrices of all sizes simultaneously, then the homotopy relation and the similarity relation give us the same equivalence classes of idempotents \cite[Section 4]{Bl}. In the remainder of this section, we will examine the relationship between these two equivalence relations in our context.

\begin{lem}\label{simtohom1}
Let $A$ be a unital filtered $SQ_p$ algebra. Suppose that $e\in A$ is an $(\varepsilon,r,N)$-idempotent and $(u,v)$ is an $(\varepsilon',r',N')$-inverse pair in $A$. Then $uev$ is an $((NN')^2\varepsilon'+N'^2\varepsilon,r+2r',NN'^2)$-idempotent. In particular, if $e\in M_n(A)$ is an $(\varepsilon,r,N)$-idempotent and $u\in M_n(\mathbb{C})$ is invertible with $\max(||u||,||u^{-1}||)\leq 1$, then $ueu^{-1}$ is an $(\varepsilon,r,N)$-idempotent in $M_n(A)$.
\end{lem}

\begin{proof}
The first statement holds since \[||uevuev-uev||\leq||ue(vu-1)ev||+||u(e^2-e)v||<(NN')^2\varepsilon'+N'^2\varepsilon.\] The second statement follows from the first by setting $\varepsilon'=0$ and $r'=0$.
\end{proof}

\begin{lem}\label{simtohom2}
Let $A$ be a unital filtered $SQ_p$ algebra. If $e\in A$ is an $(\varepsilon,r,N)$-idempotent, and $(u,v)$ is an $(\varepsilon',r',N')$-inverse pair in $A$, then $\diag(uev,0)$ and $\diag(e,0)$ are $(\varepsilon'',r'',N'')$-homotopic in $M_2(A)$, where
\begin{align*}
\varepsilon''&=4(N'+\varepsilon')^2(2N^2\varepsilon'+\varepsilon),\\
r''&=r+4r', \text{and} \\
N''&=4N(N'+\varepsilon')^2.
\end{align*} 
In particular, if $e\in M_n(A)$ is an $(\varepsilon,r,N)$-idempotent and $u\in M_n(\mathbb{C})$ is invertible with $\max(||u||,||u^{-1}||)\leq 1$, then $\diag(ueu^{-1},0)$ and $\diag(e,0)$ are $(4\varepsilon,r,4N)$-homotopic in $M_{2n}(A)$.
\end{lem}

\begin{proof}
Applying Lemma \ref{inversepairhomotopy}, let $U_t$ be a homotopy of $(2\varepsilon',2r',2(N'+\varepsilon'))$-invertibles between $\diag(u,v)$ and $I_2$, and let $V_t$ be a homotopy of $(2\varepsilon',2r',2(N'+\varepsilon'))$-invertibles between $\diag(v,u)$ and $I_2$ such that $(U_t,V_t)$ is a $(2\varepsilon',2r',2(N'+\varepsilon'))$-inverse pair in $M_2(A)$ for each $t\in[0,1]$. Then $U_t\diag(e,0)V_t$ is a homotopy of $(\varepsilon'',r'',N'')$-idempotents between $\diag(uev,0)$ and $\diag(e,0)$ by Lemma \ref{simtohom1}, where $\varepsilon'',r'',N''$ are given by the expressions in the statement. The second statement follows from the first by setting $\varepsilon'=0$ and $r'=0$.
\end{proof}


\begin{lem} \label{closeimpliessimilar}
Let $A$ be a unital filtered Banach algebra. If $e$ and $f$ are $(\varepsilon,r,N)$-idempotents in $A$ such that $||e-f||<\frac{\varepsilon}{2N+1}$, then there exists an $(\varepsilon,n_\varepsilon r,\frac{1}{1-3\varepsilon})$-inverse pair $(u,v)$ in $A$ such that $||uev-f||<\frac{(5N+2)\varepsilon}{1-3\varepsilon}$, where $n_\varepsilon\geq 1$ and $\varepsilon\mapsto n_\varepsilon$ is non-increasing.
\end{lem}

\begin{proof}

Let $v=ef+(1-e)(1-f)$. Then $v-1=(2e-1)(f-e)+2(e^2-e)$. If $||e-f||<\frac{\varepsilon}{2N+1}$, then $||v-1||<3\varepsilon$ so $v$ is invertible. We also have $||v||<1+3\varepsilon$ and $||v^{-1}||\leq\frac{1}{1-||1-v||}<\frac{1}{1-3\varepsilon}$.

Now $ev=2e^2 f-e^2-ef+e$ and $vf=2ef^2-ef-f^2+f$ so \[ev-vf=2(e^2-e)f+2e(f-f^2)-(e^2-e)+(f^2-f)\] and $||ev-vf||\leq(4N+2)\varepsilon$. Since $||v^{-1}||<\frac{1}{1-3\varepsilon}$, we have \[||v^{-1}ev-f||<\frac{(4N+2)\varepsilon}{1-3\varepsilon}.\]

Let $m_\varepsilon$ be the smallest positive integer such that $\left\Vert\sum_{k=m_\varepsilon+1}^\infty (1-v)^k\right\Vert<\frac{\varepsilon}{2}$, and let $u=\sum_{k=0}^{m_\varepsilon} (1-v)^k$. Then $||u||<\frac{1}{1-3\varepsilon}$ and \[||uev-f||\leq ||(u-v^{-1})ev||+||v^{-1}ev-f||<N\varepsilon+\frac{(4N+2)\varepsilon}{1-3\varepsilon}<\frac{(5N+2)\varepsilon}{1-3\varepsilon}.\]

Also note that $v\in A_{2r}$, $u\in A_{2m_\varepsilon r}$, \[||uv-1||\leq\left\Vert\sum_{k=m_\varepsilon+1}^\infty (1-v)^k\right\Vert ||v||<\varepsilon,\] and similarly $||vu-1||<\varepsilon$ so $(u,v)$ is an $(\varepsilon,2m_\varepsilon r,\frac{1}{1-3\varepsilon})$-inverse pair in $A$.
\end{proof}


%
%
%

\begin{prop}\label{homtosim2}
Let $A$ be a unital filtered Banach algebra. If there is an $M$-Lipschitz homotopy of $(\varepsilon,r,N)$-idempotents in $A$ between $e$ and $f$ with $\frac{1}{M}\leq\varepsilon$, then there exists an $(\varepsilon',r',N')$-inverse pair $(u,v)$ in $A$ such that $||uev-f||<3(\frac{9}{4})^{M(2N+1)+1}(5N+2)\varepsilon$, where $\varepsilon'=2(\frac{9}{4})^{M(2N+1)+1}\varepsilon$, $r'=(M(2N+1)+1)n_\varepsilon r$ with $n_\varepsilon\geq 1$ and $\varepsilon\mapsto n_\varepsilon$ non-increasing, and $N'=(\frac{3}{2})^{M(2N+1)+1}$.
\end{prop}

\begin{proof}
Let $(e_t)$ be an $M$-Lipschitz homotopy of $(\varepsilon,r,N)$-idempotents in $A$ between $e$ and $f$ with $\frac{1}{M}\leq\varepsilon$. Let $0=t_0<t_1<\cdots<t_k=1$ be such that $\frac{1}{M(2N+1)+1}<|t_i-t_{i-1}|<\frac{1}{M(2N+1)}$. Note that $k<M(2N+1)+1$. By Lemma \ref{closeimpliessimilar}, there exists an $(\varepsilon,n_\varepsilon r,\frac{1}{1-3\varepsilon})$-inverse pair $(u_i,v_i)$ in $A$ such that \[||u_i e_{t_{i-1}}v_i-e_{t_i}||<\frac{(5N+2)\varepsilon}{1-3\varepsilon}.\] Set $(u,v)=(u_k\cdots u_1,v_1\cdots v_k)$. Then \[||uv-1||<2(\frac{9}{4})^{k}\varepsilon\] and similarly for $||vu-1||$. Thus $(u,v)$ is a $(2(\frac{9}{4})^{k}\varepsilon,kn_\varepsilon r,(\frac{3}{2})^k)$-inverse pair, and $||uev-f||<3(\frac{9}{4})^{k}(5N+2)\varepsilon$.
\end{proof}

If a homotopy of quasi-idempotents is not Lipschitz, the following lemma enables us to replace the homotopy with a Lipschitz homotopy by enlarging matrices, after which Proposition \ref{homtosim2} becomes applicable. The Lipschitz constant depends only on the parameter $N$.

\begin{lem}\label{idemliphom}
Let $A$ be a unital filtered $SQ_p$ algebra. If $e$ and $f$ are homotopic as $(\varepsilon,r,N)$-idempotents in $A$, then there exist $\alpha_N>0$, $k\in\mathbb{N}$, and an $\alpha_N$-Lipschitz homotopy of $(2\varepsilon,r,\frac{5}{2}N)$-idempotents between $\diag(e,I_k,0_k)$ and $\diag(f,I_k,0_k)$.
\end{lem}

\begin{proof}
Let $(e_t)_{t\in[0,1]}$ be a homotopy of $(\varepsilon,r,N)$-idempotents between $e$ and $f$, and let $0=t_0<t_1<\cdots<t_k=1$ be such that \[||e_{t_i}-e_{t_{i-1}}||<\inf_{t\in[0,1]}\frac{\varepsilon-||e_t^2-e_t||}{2N+1}.\] For each $t$, we have a Lipschitz homotopy of $(\varepsilon,r,\frac{5}{2}N)$-idempotents between $\diag(e_t,1-e_t)$ and $\diag(1,0)$ given by \[\begin{pmatrix} e_t & 0 \\ 0 & 0 \end{pmatrix}+\begin{pmatrix} c_l & -s_l \\ s_l & c_l \end{pmatrix}\begin{pmatrix} 1-e_t & 0 \\ 0 & 0 \end{pmatrix}\begin{pmatrix} c_l & s_l \\ -s_l & c_l \end{pmatrix},\] where $c_l=\cos\frac{\pi l}{2}$ and $s_l=\sin\frac{\pi l}{2}$ for $l\in[0,1]$. Also, there is a Lipschitz (in fact linear) homotopy of $(\varepsilon,r,N)$-idempotents between $e_{t_{i-1}}$ and $e_{t_i}$ for each $i$. Then we have the following sequence of Lipschitz homotopies of $(2\varepsilon,r,\frac{5}{2}N)$-idempotents in which the first and last homotopies are obtained from Lemma \ref{invcomm}:
\begin{footnotesize}
\begin{align*}
&\quad\quad\quad \begin{pmatrix} e_{t_0} & & \\ & I_k & \\ & & 0_k \end{pmatrix}\\ &\stackrel{2\varepsilon,r,2N}{\sim} \begin{pmatrix} e_{t_0} & & & & & \\ & 1 & & & & \\ & & 0 & & & \\ & & & \ddots & & \\ & & & & 1 & \\ & & & & & 0 \end{pmatrix} \stackrel{\varepsilon,r,\frac{5}{2}N}{\sim} \begin{pmatrix} e_{t_0} & & & & & \\ & 1-e_{t_1} & & & & \\ & & e_{t_1} & & & \\ & & & \ddots & & \\ & & & & 1-e_{t_k} & \\ & & & & & e_{t_k} \end{pmatrix} \\ 
&\stackrel{\varepsilon,r,N}{\sim} \begin{pmatrix} e_{t_0} & & & & & \\ & 1-e_{t_0} & & & & \\ & & e_{t_1} & & & \\ & & & \ddots & & \\ & & & & 1-e_{t_{k-1}} & \\ & & & & & e_{t_k} \end{pmatrix} 
\stackrel{\varepsilon,r,\frac{5}{2}N}{\sim} \begin{pmatrix} 1 & & & & & \\ & 0 & & & & \\ & & 1 & & & \\ & & & \ddots & & \\ & & & & 0 & \\ & & & & & e_{t_k} \end{pmatrix} \\ &\stackrel{2\varepsilon,r,2N}{\sim} \begin{pmatrix} e_{t_k} & & \\ & I_k & \\ & & 0_k \end{pmatrix}.
\end{align*}
\end{footnotesize}
\end{proof}

\begin{rem}
From the proofs above, we see that in Proposition \ref{homtosim2}, if $A$ is a non-unital filtered Banach algebra, and $e,f\in A^+$ are such that $e-1\in A$ and $f-1\in A$, then we can find $(u,v)$ satisfying the conclusion of Proposition \ref{homtosim2} such that $u-1\in A$ and $v-1\in A$.
\end{rem}

\section{Quantitative $K$-Theory}

In this section, we define the quantitative $K$-theory groups for a filtered $SQ_p$ algebra $A$. Then we establish some basic properties of these groups, and we examine the relation between these groups and the ordinary $K$-theory groups. Our setup is in parallel with the theory developed in \cite{OY15} for filtered $C^*$-algebras.

\subsection{Definitions of quantitative $K$-theory groups}

Given a filtered $SQ_p$ algebra $A$, we denote by $Idem^{\varepsilon,r,N}(A)$ the set of $(\varepsilon,r,N)$-idempotents in $A$. For each positive integer $n$, we set $Idem^{\varepsilon,r,N}_n(A)=Idem^{\varepsilon,r,N}(M_n(A))$. Then we have inclusions $Idem^{\varepsilon,r,N}_n(A)\hookrightarrow Idem^{\varepsilon,r,N}_{n+1}(A)$ given by $e\mapsto\begin{pmatrix} e & 0 \\ 0 & 0 \end{pmatrix}$, and we set \[Idem^{\varepsilon,r,N}_\infty(A)=\bigcup_{n\in\mathbb{N}}Idem^{\varepsilon,r,N}_n(A).\]

Consider the equivalence relation $\sim$ on $Idem^{\varepsilon,r,N}_\infty(A)$ defined by $e\sim f$ if $e$ and $f$ are $(4\varepsilon,r,4N)$-homotopic in $M_\infty(A)$. We will denote the equivalence class of $e\in Idem^{\varepsilon,r,N}_\infty(A)$ by $[e]$. We will sometimes write $[e]_{\varepsilon,r,N}$ if we wish to emphasize the parameters. 



We define addition on $Idem_\infty^{\varepsilon,r,N}(A)/{\sim}$ by $[e]+[f]=[\diag(e,f)]$, where $\diag(e,f)=\begin{pmatrix} e & 0 \\ 0 & f \end{pmatrix}$. 

\begin{prop}
For any filtered $SQ_p$ algebra $A$, $Idem_\infty^{\varepsilon,r,N}(A)/{\sim}$ is an abelian semigroup with identity $[0]$. If $B$ is another filtered $SQ_p$ algebra and $\phi:A\rightarrow B$ is a filtered homomorphism, then there is an induced homomorphism \[\phi_*:Idem_\infty^{\varepsilon,r,N}(A)/{\sim}\rightarrow Idem_\infty^{||\phi||_{pcb}\varepsilon,r,||\phi||_{pcb}N}(B)/{\sim}.\]
\end{prop}

\begin{proof}
To show commutativity, we need to show that if $e,f\in Idem^{\varepsilon,r,N}_\infty(A)$, then $[\diag(e,f)]=[\diag(f,e)]$. We may assume that $e,f\in M_n(A)$ for some $n\in\mathbb{N}$. Letting \[R_t=\begin{pmatrix} (\cos\frac{\pi t}{2})I_n & (\sin\frac{\pi t}{2})I_n \\ -(\sin\frac{\pi t}{2})I_n & (\cos\frac{\pi t}{2})I_n \end{pmatrix}\] for $t\in[0,1]$, one sees that $R_t\diag(e,f)R_t^{-1}$ is a homotopy of $(2\varepsilon,r,2N)$-idempotents between $\diag(e,f)$ and $\diag(f,e)$.

Using the same notation $\phi$ for the induced homomorphism $M_\infty(A)\rightarrow M_\infty(B)$, note that $\phi(e)\in Idem^{||\phi||_{pcb}\varepsilon,r,||\phi||_{pcb}N}_\infty(B)$ whenever $e\in Idem^{\varepsilon,r,N}_\infty(A)$. Moreover, if $e\sim e'$, then $\phi(e)\sim\phi(e')$. Thus \[\phi_*:Idem_\infty^{\varepsilon,r,N}(A)/{\sim}\rightarrow Idem_\infty^{||\phi||_{pcb}\varepsilon,r,||\phi||_{pcb}N}(B)/{\sim}\] given by $\phi_*([e])=[\phi(e)]$ is a well-defined homomorphism of semigroups.
\end{proof}

\begin{defn}
Let $A$ be a unital filtered $SQ_p$ algebra. For $0<\varepsilon<\frac{1}{20},r> 0$ and $N\geq 1$, define $K_0^{\varepsilon,r,N}(A)$ to be the Grothendieck group of $Idem_\infty^{\varepsilon,r,N}(A)/{\sim}$. 
\end{defn}

By the universal property of the Grothendieck group, when $A$ and $B$ are unital filtered $SQ_p$ algebras and $\phi:A\rightarrow B$ is a filtered homomorphism, the induced homomorphism \[\phi_*:Idem_\infty^{\varepsilon,r,N}(A)/{\sim}\rightarrow Idem_\infty^{||\phi||_{pcb}\varepsilon,r,||\phi||_{pcb}N}(B)/{\sim}\] extends to a group homomorphism \[\phi_*:K_0^{\varepsilon,r,N}(A)\rightarrow K_0^{||\phi||_{pcb}\varepsilon,r,||\phi||_{pcb}N}(B).\] Moreover, if $\psi:B\rightarrow C$ is another filtered homomorphism between unital filtered $SQ_p$ algebras, then \[(\psi\circ\phi)_*=\psi_*\circ\phi_*:K_0^{\varepsilon,r,N}(A)\rightarrow K_0^{||\psi||_{pcb}||\phi||_{pcb}\varepsilon,r,||\psi||_{pcb}||\phi||_{pcb}N}(C).\]

If $A$ is a non-unital filtered $SQ_p$ algebra, we have the usual quotient homomorphism $\pi:A^+\rightarrow\mathbb{C}$, which is $p$-completely contractive by \cite[Lemma 4.2]{Daws10} and the standard fact that all characters on Banach algebras are contractive. Thus $\pi$ induces a homomorphism $\pi_*:K_0^{\varepsilon,r,N}(A^+)\rightarrow K_0^{\varepsilon,r,N}(\mathbb{C})$.

\begin{defn}
Let $A$ be a non-unital filtered $SQ_p$ algebra. For $0<\varepsilon<\frac{1}{20},r> 0$ and $N\geq 1$, define \[K_0^{\varepsilon,r,N}(A)=\ker(\pi_*:K_0^{\varepsilon,r,N}(A^+)\rightarrow K_0^{\varepsilon,r,N}(\mathbb{C})).\]
\end{defn}

Note that for $0<\varepsilon\leq\varepsilon'<\frac{1}{20},0< r\leq r'$, and $1\leq N\leq N'$, we have a canonical group homomorphism \[\iota_0^{\varepsilon,\varepsilon',r,r',N,N'}:K_0^{\varepsilon,r,N}(A)\rightarrow K_0^{\varepsilon',r',N'}(A)\] given by $\iota_0^{\varepsilon,\varepsilon',r,r',N,N'}([e]_{\varepsilon,r,N})=[e]_{\varepsilon',r',N'}$.

We have already observed that if $A$ and $B$ are both unital filtered $SQ_p$ algebras, then a filtered homomorphism $\phi:A\rightarrow B$ induces a group homomorphism \[\phi_*:K_0^{\varepsilon,r,N}(A)\rightarrow K_0^{||\phi||_{pcb}\varepsilon,r,||\phi||_{pcb}N}(B).\] If $A$ and $B$ are both non-unital and $\phi^+:A^+\rightarrow B^+$ denotes the induced homomorphism between their unitizations, then we get a homomorphism \[K_0^{\varepsilon,r,N}(A^+)\rightarrow K_0^{||\phi^+||_{pcb}\varepsilon,r,||\phi^+||_{pcb}N}(B^+),\] which restricts to a homomorphism \[\phi_*:K_0^{\varepsilon,r,N}(A)\rightarrow K_0^{||\phi^+||_{pcb}\varepsilon,r,||\phi^+||_{pcb}N}(B).\]

Given a unital filtered $SQ_p$ algebra $A$, we denote by $GL^{\varepsilon,r,N}(A)$ the set of $(\varepsilon,r,N)$-invertibles in $A$. For each positive integer $n$, we set $GL_n^{\varepsilon,r,N}(A)=GL^{\varepsilon,r,N}(M_n(A))$. Then we have inclusions $GL_n^{\varepsilon,r,N}(A)\hookrightarrow GL_{n+1}^{\varepsilon,r,N}(A)$ given by $u\mapsto\begin{pmatrix} u & 0 \\ 0 & 1 \end{pmatrix}$, and we set \[GL_\infty^{\varepsilon,r,N}(A)=\bigcup_{n\in\mathbb{N}} GL_n^{\varepsilon,r,N}(A).\]

Consider the equivalence relation $\sim$ on $GL_\infty^{\varepsilon,r,N}(A)$ given by $u\sim v$ if $u$ and $v$ are $(4\varepsilon,2r,4N)$-homotopic in $M_\infty(A)$. We will denote the equivalence class of $u\in GL^{\varepsilon,r,N}_\infty(A)$ by $[u]$. We will sometimes write $[u]_{\varepsilon,r,N}$ if we wish to emphasize the parameters.  


We define addition on $GL_\infty^{\varepsilon,r,N}(A)/{\sim}$ by $[u]+[v]=[\diag(u,v)]$.

\begin{prop}
For any unital filtered $SQ_p$ algebra $A$, $GL_\infty^{\varepsilon,r,N}(A)/{\sim}$ is an abelian group. If $B$ is another unital filtered $SQ_p$ algebra and $\phi:A\rightarrow B$ is a unital filtered homomorphism, then there is an induced homomorphism \[\phi_*:GL_\infty^{\varepsilon,r,N}(A)/{\sim}\rightarrow GL_\infty^{||\phi||_{pcb}\varepsilon,r,||\phi||_{pcb}N}(B)/{\sim}.\]
\end{prop}

\begin{proof}
By Lemma \ref{inversepairhomotopy}, if $(u,v)$ is an $(\varepsilon,r,N)$-inverse pair in $M_n(A)$, then $[u]+[v]=[1]$. By Lemma \ref{invcomm}, we have $[u]+[v]=[v]+[u]$. Hence $GL_\infty^{\varepsilon,r,N}(A)/{\sim}$ is an abelian group.

If $\phi:A\rightarrow B$ is a unital filtered homomorphism, and $u\in GL_n^{\varepsilon,r,N}(A)$, then $\phi(u)\in GL_n^{||\phi||_{pcb}\varepsilon,r,||\phi||_{pcb}N}(B)$. Moreover, if $u\sim u'$, then $\phi(u)\sim\phi(u')$. Thus \[\phi_*:GL_\infty^{\varepsilon,r,N}(A)/{\sim}\rightarrow GL_\infty^{||\phi||_{pcb}\varepsilon,r,||\phi||_{pcb}N}(B)/{\sim}\] given by $\phi_*([u])=[\phi(u)]$ is a well-defined group homomorphism.
\end{proof}

\begin{defn}
Let $A$ be a unital filtered $SQ_p$ algebra. For $0<\varepsilon<\frac{1}{20}$, $r> 0$, and $N\geq 1$, define \[K_1^{\varepsilon,r,N}(A)=GL_\infty^{\varepsilon,r,N}(A)/{\sim}.\]
If $A$ is non-unital, define $K_1^{\varepsilon,r,N}(A)=\ker(\pi_*:K_1^{\varepsilon,r,N}(A^+)\rightarrow K_1^{\varepsilon,r,N}(\mathbb{C}))$.
\end{defn}

By the previous proposition, $K_1^{\varepsilon,r,N}(A)$ is an abelian group for any filtered $SQ_p$ algebra $A$. 

For $0<\varepsilon\leq\varepsilon'<\frac{1}{20}$, $0< r\leq r'$, and $1\leq N\leq N'$, we have a canonical group homomorphism \[\iota_1^{\varepsilon,\varepsilon',r,r',N,N'}:K_1^{\varepsilon,r,N}(A)\rightarrow K_1^{\varepsilon',r',N'}(A)\] given by $\iota_1^{\varepsilon,\varepsilon',r,r',N,N'}([u]_{\varepsilon,r,N})=[u]_{\varepsilon',r',N'}$. 

\begin{rem}
We will sometimes refer to the canonical homomorphisms \[\iota_*^{\varepsilon,\varepsilon',r,r',N,N'}:K_*^{\varepsilon,r,N}(A)\rightarrow K_*^{\varepsilon',r',N'}(A)\] as relaxation of control maps, and we will also omit the superscripts, writing just $\iota_*$, when they are clear from the context so as to reduce notational clutter.
\end{rem}

Just as in ordinary $K$-theory, we can give a unified treatment of the unital and non-unital cases. 
%
%
First, observe that if $A_1$ and $A_2$ are unital filtered $SQ_p$ algebras, then the coordinate projections induce isomorphisms \[K_i^{\varepsilon,r,N}(A_1\oplus A_2)\cong K_i^{\varepsilon,r,N}(A_1)\oplus K_i^{\varepsilon,r,N}(A_2)\] for $i=0,1$.

In ordinary $K$-theory, when $A$ is unital, we have \[K_0(A)=\ker(\pi_*:K_0(A^+)\rightarrow K_0(\mathbb{C}))\] and \[K_1(A)= K_1(A^+)=\ker(\pi_*:K_1(A^+)\rightarrow K_1(\mathbb{C}))\] since $K_1(\mathbb{C})=0$. Recall that when $A$ is unital, we identify $M_n(A^+)$ isometrically with $M_n(A)\oplus M_n(\mathbb{C})$ (equipped with the max-norm) via the canonical isomorphism $A^+\rightarrow A\oplus\mathbb{C}$. In our current setting, we have a homomorphism $K_0^{\varepsilon,r,N}(A)\rightarrow \ker\pi_*^{\varepsilon,r,N}$ given by $[e]\mapsto[(e,0)]$, and we also have a homomorphism $\ker\pi_*^{\varepsilon,r,N}\rightarrow K_0^{\varepsilon,r,N}(A)$ given by the composition
\[ \ker\pi_*^{\varepsilon,r,N}\hookrightarrow K_0^{\varepsilon,r,N}(A^+)\stackrel{\cong}{\rightarrow} K_0^{\varepsilon,r,N}(A\oplus\mathbb{C})\cong K_0^{\varepsilon,r,N}(A)\oplus K_0^{\varepsilon,r,N}(\mathbb{C})\rightarrow K_0^{\varepsilon,r,N}(A). \]
The composition $K_0^{\varepsilon,r,N}(A)\rightarrow \ker\pi_*^{\varepsilon,r,N}\rightarrow K_0^{\varepsilon,r,N}(A)$ is the identity map while the composition $\ker\pi_*^{\varepsilon,r,N}\rightarrow K_0^{\varepsilon,r,N}(A)\rightarrow \ker\pi_*^{\varepsilon,r,N}$ is given by $[(e,z)]_{\varepsilon,r,N}\mapsto[(e+z,0)]_{\varepsilon,r,N}$.
Let $\psi:A^+\rightarrow A\oplus\mathbb{C}$ be the canonical isomorphism. Since $[z]=0$ in $K_0^{\varepsilon,r,N}(\mathbb{C})$, we have \[ [\psi((e,z))]=[(e+z,z)]=[(e+z,0)]=[\psi((e+z,0))] \] in $K_0^{\varepsilon,r,N}(A\oplus\mathbb{C})$. It follows that $[(e,z)]=[(e+z,0)]$ in $K_0^{\varepsilon,r,N}(A^+)$, and we have $K_0^{\varepsilon,r,N}(A)\cong \ker\pi_*^{\varepsilon,r,N}$.

It is not clear that $K_1^{\varepsilon,r,N}(\mathbb{C})=0$ for all $\varepsilon,r,N$, and for all choices of norms on $M_n(\mathbb{C})$ that we are considering, but by a similar argument as in the even case above, we still have $K_1^{\varepsilon,r,N}(A)\cong\ker\pi_*^{\varepsilon,r,N}$.

Thus we have
\begin{prop}
For any filtered $SQ_p$ algebra $A$, \[K_*^{\varepsilon,r,N}(A)\cong \ker(\pi_*^{\varepsilon,r,N}:K_*^{\varepsilon,r,N}(A^+)\rightarrow K_*^{\varepsilon,r,N}(\mathbb{C})).\]
\end{prop}

\subsection{Some basic properties of quantitative $K$-theory groups}


If $A$ and $B$ are filtered $SQ_p$ algebras, and $\phi:A\rightarrow B$ is a filtered homomorphism, then we have the following commutative diagram:
\[
\begindc{\commdiag}[100]		
\obj(0,5)[2a]{$\ker\pi_{A*}^{\varepsilon,r,N}$}	\obj(14,5)[2b]{$K_*^{\varepsilon,r,N}(A^+)$}	\obj(28,5)[2c]{$K_*^{\varepsilon,r,N}(\mathbb{C})$}

\mor{2a}{2b}{}[1,\injectionarrow]	\mor{2b}{2c}{} 

\obj(0,0)[3a]{$\ker\pi_{B*}^{||\phi^+||_{pcb}\varepsilon,r,||\phi^+||_{pcb}N}$}	\obj(28,0)[3c]{$K_*^{||\phi^+||_{pcb}\varepsilon,r,||\phi^+||_{pcb}N}(\mathbb{C})$}	\obj(14,0)[3b]{$K_*^{||\phi^+||_{pcb}\varepsilon,r,||\phi^+||_{pcb}N}(B^+)$}

\mor{2a}{3a}{}	\mor{2c}{3c}{$\iota_*$}		\mor{3a}{3b}{}[1,\injectionarrow]	\mor{3b}{3c}{} \mor{2b}{3b}{$\phi^+_*$}
\enddc
\]
The homomorphism on the left is the restriction of $\phi_*^+$. By the discussion above, we then get a homomorphism \[\phi_*:K_*^{\varepsilon,r,N}(A)\rightarrow K_*^{||\phi^+||_{pcb}\varepsilon,r,||\phi^+||_{pcb}N}(B).\] Moreover, if $\psi:B\rightarrow C$ is another filtered homomorphism between filtered $SQ_p$ algebras, then $\psi_*\circ\phi_*=(\psi\circ\phi)_*$.


\begin{defn}
A homotopy between two filtered homomorphisms $\phi_0:A\rightarrow B$ and $\phi_1:A\rightarrow B$ of filtered $SQ_p$ algebras is a filtered homomorphism $\Phi:A\rightarrow C([0,1],B)$ such that $ev_0\circ\Phi=\phi_0$ and $ev_1\circ\Phi=\phi_1$, where $ev_t$ denotes the homomorphism $C([0,1],B)\rightarrow B$ given by evaluation at $t$.
\end{defn}

Recall that a filtered homomorphism $\Phi:A\rightarrow C([0,1],B)$ is a $p$-completely bounded homomorphism such that $\Phi(A_r)\subset C([0,1],B_r)$ for each $r>0$. Given such a filtered homomorphism $\Phi$, we may consider $\phi_t=ev_t\circ\Phi$ for $t\in[0,1]$. Then each $\phi_t:A\rightarrow B$ is a filtered homomorphism, the map $t\mapsto\phi_t(a)$ is continuous for each $a\in A$, and $\sup_{t\in[0,1]}||\phi_t||_{pcb}\leq||\Phi||_{pcb}$.


\begin{prop}
Let $\phi_0:A\rightarrow B$ and $\phi_1:A\rightarrow B$ be filtered homomorphisms between filtered $SQ_p$ algebras. Suppose that $\Phi:A\rightarrow C([0,1],B)$ is a homotopy between $\phi_0$ and $\phi_1$.
Then \[\phi_{0*}=\phi_{1*}:K_*^{\varepsilon,r,N}(A)\rightarrow K_*^{||\Phi^+||_{pcb}\varepsilon,r,||\Phi^+||_{pcb}N}(B). \]
\end{prop}

\begin{proof}
For $t\in[0,1]$, let $\phi_t=ev_t\circ\Phi$. For any $(\varepsilon,r,N)$-idempotent $e$ in $M_n(A^+)$, $\phi_t^+(e)$ is a homotopy of $(||\Phi^+||_{pcb}\varepsilon,r,||\Phi^+||_{pcb}N)$-idempotents in $M_n(B^+)$ between $\phi_0^+(e)$ and $\phi_1^+(e)$. Similarly, for any $(\varepsilon,r,N)$-invertible $u$ in $M_n(A^+)$, $\phi_t^+(u)$ is a homotopy of $(||\Phi^+||_{pcb}\varepsilon,r,||\Phi^+||_{pcb}N)$-invertibles in $M_n(B^+)$ between $\phi_0^+(u)$ and $\phi_1^+(u)$.
\end{proof}

\begin{prop}\label{stab}
If $A$ is a filtered $SQ_p$ algebra, then the canonical embedding $A\rightarrow M_n(A)$ induces an isomorphism \[K_*^{\varepsilon,r,N}(A)\cong K_*^{\varepsilon,r,N}(M_n(A))\] for each $n\in\mathbb{N}$.

\end{prop}


\begin{proof}
Let $A$ be a unital filtered $SQ_p$ algebra, and let $i:A\rightarrow M_n(A)$ be the canonical embedding $a\mapsto\diag(a,0)$. Then we get the induced maps $i_k:M_k(A)\rightarrow M_k(M_n(A))$ for $k\in\mathbb{N}$. Let \[\zeta:M_k(M_n(A))\rightarrow M_{kn}(A)\] be the isomorphism given by removing parentheses. Recall that we equip $M_k(M_n(A))$ with the norm induced by this isomorphism. If $e$ is an $(\varepsilon,r,N)$-idempotent in $M_k(A)$, then $i_k(e)$ is an $(\varepsilon,r,N)$-idempotent in $M_k(M_n(A))$, and $\zeta(i_k(e))$ is an $(\varepsilon,r,N)$-idempotent in $M_{kn}(A)$. Now \[\zeta(i_k(e))=u\diag(e,0)u^{-1}\] for some permutation matrix $u$. By Lemma \ref{simtohom2}, $\diag(\zeta(i_k(e)),0)$ and $\diag(e,0)$ are $(4\varepsilon,r,4N)$-homotopic in $M_{2kn}(A)$. Hence the composition $\zeta_*\circ i_*$ is the identity on $K_0^{\varepsilon,r,N}(A)$.

On the other hand, if $f$ is an $(\varepsilon,r,N)$-idempotent in $M_k(M_n(A))$, then $\zeta(f)$ is an $(\varepsilon,r,N)$-idempotent in $M_{kn}(A)$, and $i_{kn}(\zeta(f))$ is an $(\varepsilon,r,N)$-idempotent in $M_{kn}(M_n(A))$. As above, we see that $\diag(i_{kn}(\zeta(f)),0)$ and $\diag(f,0)$ are $(4\varepsilon,r,4N)$-homotopic in $M_{2kn}(M_n(A))$. Hence the composition $i_*\circ\zeta_*$ is the identity on $K_0^{\varepsilon,r,N}(M_n(A))$.

The odd case and the non-unital case are proved in a similar way. We just remark that in the non-unital case, we have \[\zeta(i_k(e))=u\diag(e,\pi(e),\ldots,\pi(e))u^{-1}\] for some permutation matrix $u$, where $\pi:A^+\rightarrow\mathbb{C}$ is the canonical quotient homomorphism.
\end{proof}


We also have a ``standard picture'' for the quantitative $K_0$ groups analogous to that for the usual $K_0$ group.

\begin{lem}
If $e,f\in M_\infty(A)$ are $(\varepsilon,r,N)$-idempotents and $[e]-[f]=0$ in $K_0^{\varepsilon,r,N}(A)$, then there exists $m\in\mathbb{N}$ such that $\diag(e,I_m,0_m)$ and $\diag(f,I_m,0_m)$ are homotopic as $(4\varepsilon,r,4N)$-idempotents in $M_\infty(\tilde{A})$.
\end{lem}


\begin{proof}
If $[e]-[f]=0$ in $K_0^{\varepsilon,r,N}(A)$, then $[e]+[g]=[f]+[g]$ in $Idem_\infty^{\varepsilon,r,N}(\tilde{A})/{\sim}$ for some $g\in Idem_m^{\varepsilon,r,N}(\tilde{A})$, so $\diag(e,g)$ and $\diag(f,g)$ are $(4\varepsilon,r,4N)$-homotopic in $M_\infty(\tilde{A})$. Now $I_m-g$ is $(\varepsilon,r,N)$-idempotent in $M_m(\tilde{A})$, and $\diag(e,g,I_m-g)\stackrel{4\varepsilon,r,4N}{\sim}\diag(f,g,I_m-g)$ in $M_\infty(\tilde{A})$. But we have a homotopy of $(\varepsilon,r,\frac{5}{2}N)$-idempotents between $\diag(I_m,0_m)$ and $\diag(g,I_m-g)$ given by \[\diag(g,0)+R_t\diag(I_m-g,0)R_t^{-1},\] where $R_t=\begin{pmatrix} (\cos\frac{\pi t}{2})I_m & (\sin\frac{\pi t}{2})I_m \\ -(\sin\frac{\pi t}{2})I_m & (\cos\frac{\pi t}{2})I_m \end{pmatrix}$, and hence a homotopy of $(4\varepsilon,r,4N)$-idempotents between $\diag(e,I_m,0_m)$ and $\diag(f,I_m,0_m)$.
\end{proof}

\begin{lem}
If $[e]-[f]\in K_0^{\varepsilon,r,N}(A)$, where $e,f\in M_k(\tilde{A})$, then $[e]-[f]=[e']-[I_k]$ in $K_0^{\varepsilon,r,N}(A)$ for some $e'\in M_{2k}(\tilde{A})$.
\end{lem}

\begin{proof}
Let $e'=\diag(e,I_k-f)$. Then $e'\in Idem_{2k}^{\varepsilon,r,N}(\tilde{A})$. We have a homotopy of $(\varepsilon,r,\frac{5}{2}N)$-idempotents between $\diag(e,I_k-f,f)$ and $\diag(e,I_k,0)$. If $A$ is non-unital, and $[\pi(e)]=[\pi(f)]$ in $K_0^{\varepsilon,r,N}(\mathbb{C})$, then $[\pi(e')]=[I_k]$.
Hence $[e]-[f]=[e']-[I_k]$ in $K_0^{\varepsilon,r,N}(A)$.
\end{proof}

If we allow ourselves to relax control, then we can write elements in $K_0^{\varepsilon,r,N}(A)$ in the form $[e]-[I_k]$ with $\pi(e)=\diag(I_k,0)$.

\begin{lem}
There exist a non-decreasing function $\lambda:[1,\infty)\rightarrow[1,\infty)$ and a function $h:(0,\frac{1}{20})\times[1,\infty)\rightarrow[1,\infty)$ with $h(\cdot,N)$ non-increasing for each fixed $N$ such that for any filtered $SQ_p$ algebra $A$, if $[e]-[f]\in K_0^{\varepsilon,r,N}(A)$, where $e,f\in M_k(\tilde{A})$, then $[e]-[f]=[e'']-[I_k]$ in $K_0^{\lambda_N\varepsilon,h_{\varepsilon,N} r,\lambda_N}(A)$ for some $e''\in M_{2k}(\tilde{A})$ with $\pi(e'')=\diag(I_k,0_k)$.
\end{lem}

\begin{proof}
By the previous proposition, $[e]-[f]=[e']-[I_k]$ in $K_0^{\varepsilon,r,N}(A)$ for some $e'\in M_{2k}(\tilde{A})$. Since $[\pi(e')]=[I_k]$, up to rescaling $\varepsilon$ and $N$, and up to stabilization, $\pi(e')$ and $\diag(I_k,0_k)$ are homotopic as $(\varepsilon,r,N)$-idempotents in $M_{2k}(\mathbb{C})$. By Lemma \ref{homtosim2} and Lemma \ref{idemliphom}, up to stabilization, there exist functions $\lambda$ and $h$ depending only on $\varepsilon$ and $N$, and there exists a $(\lambda_N\varepsilon,h_{\varepsilon,N} r,\lambda_N)$-inverse pair $(u,v)$ in $M_{2k}(\mathbb{C})$ such that $||u\pi(e')v-\diag(I_k,0_k)||<\lambda_N\varepsilon$. Then \[||\diag(u,v)\diag(\pi(e'),0_{2k})\diag(v,u)-\diag(I_k,0_{3k})||<\lambda_N\varepsilon.\]

Let $e''=ue'v-u\pi(e')v+\diag(I_k,0_k)$. Then $\pi(e'')=\diag(I_k,0_k)$ and $||e''-ue'v||<\lambda_N\varepsilon$. By Lemma \ref{simtohom2} and Lemma \ref{normestlem1}, $\diag(e'',0_{2k})$ is homotopic to $\diag(e',0_{2k})$ as $(\lambda_N'\varepsilon,h'_{\varepsilon,N}r,\lambda'_N)$-idempotents.
\end{proof}

In the odd case, if we allow ourselves to relax control, then we can write elements in $K_1^{\varepsilon,r,N}(A)$ in the form $[u]$ with $\pi(u)=I_k$.

\begin{lem}
There exist a non-decreasing function $\lambda:[1,\infty)\rightarrow[1,\infty)$ and a function $h:(0,\frac{1}{20})\times[1,\infty)\rightarrow[1,\infty)$ with $h(\cdot,N)$ non-increasing for each fixed $N$ such that for any filtered $SQ_p$ algebra $A$, if $[u]\in K_1^{\varepsilon,r,N}(A)$ with $u\in M_k(\tilde{A})$, then $[u]=[w]$ in $K_1^{\lambda_N\varepsilon,h_{\varepsilon,N}r,\lambda_N}(A)$ for some $w\in M_k(\tilde{A})$ with $\pi(w)=I_k$.

Moreover, if $u$ and $v$ are homotopic as $(\varepsilon,r,N)$-invertibles in $M_k(\tilde{A})$, and $\pi(u)=\pi(v)=I_k$, then there is a homotopy of $(\lambda_N\varepsilon,h_{\varepsilon,N}r,\lambda_N)$-invertibles $w_t$ between $u$ and $v$ such that $\pi(w_t)=I_k$ for each $t\in[0,1]$.
\end{lem}

\begin{proof}
Suppose that $\max(||uv-1||,||vu-1||)<\varepsilon$. Let $w=\pi(u^{-1})u$. Then $||w||\leq||u^{-1}|| ||u||<\frac{N^2}{1-\varepsilon}<\frac{20}{19}N^2$ and $\pi(w)=I_k$. Note that \[||u^{-1}v^{-1}-1||\leq||(vu)^{-1}|| ||1-vu||<\frac{\varepsilon}{1-\varepsilon}<\frac{20}{19}\varepsilon.\] Similarly $||v^{-1}u^{-1}-1||<\frac{20}{19}\varepsilon$. It follows that \[\max(||\pi(u^{-1})uv\pi(v^{-1})-1||,||v\pi(v^{-1})\pi(u^{-1})u-1||)<\biggl(\biggl(\frac{20}{19}N\biggr)^2+\frac{20}{19}\biggr)\varepsilon.\] Thus $w$ is a $(((\frac{20}{19}N)^2+\frac{20}{19})\varepsilon,r,\frac{20}{19}N^2)$-invertible in $M_k(\tilde{A})$. 

Up to stabilization, we may assume that $\pi(u)$ is $(4\varepsilon,2r,4N)$-homotopic to $I_k$, so $\pi(u^{-1})$ is $(4\varepsilon,2r,\frac{80}{19}N)$-homotopic to $I_k$. It follows that $w$ and $u$ are $(((\frac{80}{19}N)^2+1)\varepsilon,2r,\frac{80}{19}N^2)$-homotopic.

For the second statement, if $(u_t)$ is a homotopy of $(\varepsilon,r,N)$-invertibles between $u$ and $v$, then $w_t=\pi(u_t^{-1})u_t$ defines the desired homotopy.
\end{proof}

If $A$ is a filtered $SQ_p$ algebra with a filtration $(A_r)_{r>0}$, and $(B_k)_{k\in\mathbb{N}}$ is an increasing sequence of Banach subalgebras of $A$ such that $\bigcup_{k\in\mathbb{N}}B_k$ is dense in $A$ and $\bigcup_{r>0}(B_k\cap A_r)$ is dense in $B_k$ for each $k$, then each $B_k$ has a filtration $(B_k\cap A_r)_{r>0}$. In this case, we have a canonical homomorphism $\varinjlim_k K_*^{\varepsilon,r,N}(B_k)\rightarrow K_*^{\varepsilon,r,N}(A)$ but we can say more. 

\begin{prop}
Let $A$ be a unital $SQ_p$ algebra with filtration $(A_r)_{r>0}$ and let $(B_k)_{k\in\mathbb{N}}$ be an increasing sequence of unital Banach subalgebras of $A$ such that
\begin{itemize}
\item $\bigcup_{r>0}(B_k\cap A_r)$ is dense in $B_k$ for each $k\in\mathbb{N}$;
\item $\bigcup_{k\in\mathbb{N}}(B_k\cap A_r)$ is dense in $A_r$ for each $r>0$.
\end{itemize}
Then for each $0<\varepsilon<\frac{1}{20}$, $r>0$, and $N\geq 1$, there is a homomorphism \[K_*^{\varepsilon,r,N}(A)\rightarrow\varinjlim_k K_*^{\varepsilon,r,N+\varepsilon}(B_k)\] such that the compositions \[K_*^{\varepsilon,r,N}(A)\rightarrow\varinjlim_k K_*^{\varepsilon,r,N+\varepsilon}(B_k)\rightarrow K_*^{\varepsilon,r,N+\varepsilon}(A)\] and \[\varinjlim_k K_*^{\varepsilon,r,N}(B_k)\rightarrow K_*^{\varepsilon,r,N}(A)\rightarrow\varinjlim_k K_*^{\varepsilon,r,N+\varepsilon}(B_k)\] are (induced by) the relaxation of control maps $\iota_*$.
\end{prop}

\begin{proof}
Note that $\bigcup_{k\in\mathbb{N}}B_k$ is dense in $A$. Let $e$ be an $(\varepsilon,r,N)$-idempotent in $M_n(A)$, and let $\delta=\frac{\varepsilon-||e^2-e||}{6(N+1)}$. Since $\bigcup_{k\in\mathbb{N}}(B_k\cap A_r)$ is dense in $A_r$, there exist $k\in\mathbb{N}$ and $f\in M_n(B_k\cap A_r)$ such that $||e-f||<\delta$. By Lemma \ref{normestlem1}, $f$ is an $(\varepsilon,r,N+\varepsilon)$-idempotent. If $f'\in M_n(B_k\cap A_r)$ also satisfies $||e-f'||<\delta$, then $||f-f'||<2\delta$. Set $f_t=(1-t)f+tf'$. Then
\begin{align*}
||f_t^2-f_t|| &\leq ||f_t^2-f_tf||+||f_tf-f^2||+||f^2-f||+||f-f_t|| \\
&\leq ||f_t-f||(||f_t||+||f||+1)+||f^2-f|| \\
&< 2\delta(2(N+\delta)+1)+(2N+2)\delta+||e^2-e|| \\
&< \varepsilon.
\end{align*}
Thus $f$ and $f'$ are $(\varepsilon,r,N+\varepsilon)$-homotopic. This gives us a well-defined homomorphism $K_0^{\varepsilon,r,N}(A)\rightarrow\varinjlim_k K_0^{\varepsilon,r,N+\varepsilon}(B_k)$ given by \[[e]-[I_n]\mapsto[f]-[I_n].\] It is straightforward to see that the compositions \[K_0^{\varepsilon,r,N}(A)\rightarrow\varinjlim_k K_0^{\varepsilon,r,N+\varepsilon}(B_k)\rightarrow K_0^{\varepsilon,r,N+\varepsilon}(A)\] and \[\varinjlim_k K_0^{\varepsilon,r,N}(B_k)\rightarrow K_0^{\varepsilon,r,N}(A)\rightarrow\varinjlim_k K_0^{\varepsilon,r,N+\varepsilon}(B_k)\] are the canonical maps $\iota_0$.

The proof for the odd case is similar.
\end{proof}


\begin{rem}\leavevmode
\begin{enumerate}
\item In the preceding proposition, if each $a\in M_n(A_r)$ can be approximated arbitrarily closely by some $b\in\bigcup_{k\in\mathbb{N}}M_n(B_k\cap A_r)$ with $||b||\leq||a||$, then we have $K_*^{\varepsilon,r,N}(A)\cong\varinjlim_k K_*^{\varepsilon,r,N}(B_k)$.

\item Regarding $M_n(\mathbb{C})$ as $B(\ell_p^n)$, we may view $M_n(A)$ as $M_n(\mathbb{C})\otimes_p A$ when $A$ is an $L_p$ operator algebra and $\otimes_p$ denotes the spatial $L_p$ operator tensor product (see Remark 1.14 and Example 1.15 in \cite{Phil13}). Writing $\overline{M^p_\infty}$ for $\overline{\bigcup_{n\in\mathbb{N}}M_n(\mathbb{C})}$, we see that $\overline{M^p_\infty}$ is a closed subalgebra of $B(\ell_p)$. Let $P_n$ be the projection onto the first $n$ coordinates with respect to the standard basis in $\ell_p$. When $p\in(1,\infty)$, we have $\lim_{n\rightarrow\infty}||a-P_n aP_n||=0$ for any compact operator $a\in K(\ell_p)$. It follows that $\overline{M^p_\infty}=K(\ell_p)$ for $p\in(1,\infty)$. However, when $p=1$, we can only say that $\lim_{n\rightarrow\infty}||a-P_n a||=0$ for $a\in K(\ell_1)$. In fact, there is a rank one operator on $\ell_1$ that is not in $\overline{M^1_\infty}$. We refer the reader to Proposition 1.8 and Example 1.10 in \cite{Phil13} for details.
\end{enumerate}
\end{rem}

\begin{cor}
If $A$ is a filtered $L_p$ operator algebra for some $p\in[1,\infty)$, then \[K_*^{\varepsilon,r,N}(\overline{M^p_\infty}\otimes_p A)\cong K_*^{\varepsilon,r,N}(A).\] In particular, when $p\in(1,\infty)$, we have \[K_*^{\varepsilon,r,N}(K(\ell_p)\otimes_p A)\cong K_*^{\varepsilon,r,N}(A).\]
\end{cor}

\subsection{Relating quantitative $K$-theory to ordinary $K$-theory}


If $e$ is an $(\varepsilon,r,N)$-idempotent in a unital filtered Banach algebra $A$, then its spectrum $\sigma(e)$ is contained in ${B}_{\sqrt{\varepsilon}}(0)\cup B_{\sqrt{\varepsilon}}(1)\subset\mathbb{C}$, where $B_r(z)$ denotes the open ball of radius $r$ centered at $z\in\mathbb{C}$. In particular, if $\varepsilon<\frac{1}{4}$, then the two balls are disjoint. By choosing a function $\kappa_0$ that is holomorphic on a neighborhood of $\sigma(e)$, takes value 0 on $\bar{B}_{\sqrt{\varepsilon}}(0)$, and takes value 1 on $\bar{B}_{\sqrt{\varepsilon}}(1)$, we may apply the holomorphic functional calculus to get an idempotent \[\kappa_0(e)=\frac{1}{2\pi i}\int_\gamma \kappa_0(z)(z-e)^{-1}dz\in A,\] where $\gamma$ may be taken to be the contour \[\{z\in\mathbb{C}:|z|=\sqrt{\varepsilon}\}\cup\{z\in\mathbb{C}:|z-1|=\sqrt{\varepsilon}\}.\] This enables us to pass from the quantitative $K_0$ groups to the ordinary $K_0$ groups.

Note that if $e$ is an idempotent, then $\kappa_0(e)=e$. Indeed, for $z\in\gamma$, we have \begin{align*} \kappa_0(e)-e &=\frac{1}{2\pi i}\int_\gamma(\kappa_0(z)-z)(z-e)^{-1}dz \\ &=\frac{1}{2\pi i}\int_\gamma(\kappa_0(z)-z)\left(\frac{1-e}{z}+\frac{e}{z-1}\right)dz \\ &=0. \end{align*}


When $e$ is an $(\varepsilon,r,N)$-idempotent, we expect $\kappa_0(e)$ to be close to $e$. In fact, we can estimate $||\kappa_0(e)-e||$ and $||\kappa_0(e)||$ in terms of $\varepsilon$ and $N$. We can also estimate $||\kappa_0(e)-\kappa_0(f)||$ in terms of $||e-f||$.

\begin{prop}
Let $e$ be an $(\varepsilon,r,N)$-idempotent in a unital filtered Banach algebra $A$. Then \[||\kappa_0(e)-e||<\frac{2(N+1)\varepsilon}{(1-\sqrt{\varepsilon})(1-2\sqrt{\varepsilon})}\] and \[||\kappa_0(e)||<\frac{N+1}{1-2\sqrt{\varepsilon}}.\]
\end{prop}

\begin{proof}
Let $\gamma=\gamma_0\cup\gamma_1$, where $\gamma_j=\{z\in\mathbb{C}:|z-j|=\sqrt{\varepsilon}\}$ for $j=0,1$. Let $y=\frac{1}{z}(1-e)+\frac{1}{z-1}e$ for $z\in\gamma$. Then 
\begin{align*}
||\kappa_0(e)-e|| &= \frac{1}{2\pi}\left\Vert\int_\gamma (\kappa_0(z)-z)(z-e)^{-1}dz\right\Vert \\ &= \frac{1}{2\pi}\left\Vert\int_{\gamma_0}-z(z-e)^{-1}dz+\int_{\gamma_1}(1-z)(z-e)^{-1}dz\right\Vert \\
&\leq 
\frac{1}{2\pi} \left[\left\Vert\int_{\gamma_0}-z((z-e)^{-1}-y)dz\right\Vert+\left\Vert\int_{\gamma_1}(1-z)((z-e)^{-1}-y)dz\right\Vert\right] \\
&\leq \varepsilon\left[\max_{z\in\gamma_0}||(z-e)^{-1}-y||+\max_{z\in\gamma_1}||(z-e)^{-1}-y||\right].
\end{align*}
For $z\in\gamma$, we have 
\begin{align*} ||(z-e)y-1||&=\left\Vert\biggl(\frac{1}{z-1}-\frac{1}{z}\biggr)(e-e^2)\right\Vert <\frac{1}{|z(z-1)|}\varepsilon \\ &\leq\frac{1}{\sqrt{\varepsilon}(1-\sqrt{\varepsilon})}\varepsilon =\frac{\sqrt{\varepsilon}}{1-\sqrt{\varepsilon}}.\end{align*}
Thus $||((z-e)y)^{-1}-1||<\frac{\frac{\sqrt{\varepsilon}}{1-\sqrt{\varepsilon}}}{1-\frac{\sqrt{\varepsilon}}{1-\sqrt{\varepsilon}}}=\frac{\sqrt{\varepsilon}}{1-2\sqrt{\varepsilon}}$. Also, \[||y||=\left\Vert\frac{1}{z}+\biggl(\frac{1}{z-1}-\frac{1}{z}\biggr)e\right\Vert\leq\frac{1}{\sqrt{\varepsilon}}+\frac{N}{\sqrt{\varepsilon}(1-\sqrt{\varepsilon})}<\frac{N+1}{\sqrt{\varepsilon}(1-\sqrt{\varepsilon})}.\] Hence 
\begin{align*} ||(z-e)^{-1}-y||&\leq||y||\cdot||((z-e)y)^{-1}-1|| \\ &<\frac{N+1}{\sqrt{\varepsilon}(1-\sqrt{\varepsilon})}\frac{\sqrt{\varepsilon}}{1-2\sqrt{\varepsilon}} \\ &=\frac{N+1}{(1-\sqrt{\varepsilon})(1-2\sqrt{\varepsilon})}
\end{align*} 
for all $z\in\gamma$, and $||\kappa_0(e)-e||<\frac{2(N+1)\varepsilon}{(1-\sqrt{\varepsilon})(1-2\sqrt{\varepsilon})}$.

We also get 
\begin{align*}
||\kappa_0(e)|| &= \frac{1}{2\pi}\left\Vert\int_{\gamma_1} (z-e)^{-1}dz\right\Vert \leq \sqrt{\varepsilon}(\max_{z\in\gamma_1}||(z-e)^{-1}-y||+||y||) \\
&< \sqrt{\varepsilon}\left(\frac{N+1}{(1-\sqrt{\varepsilon})(1-2\sqrt{\varepsilon})}+\frac{N+1}{\sqrt{\varepsilon}(1-\sqrt{\varepsilon})}\right) \\
&= \frac{N+1}{1-2\sqrt{\varepsilon}}.
\end{align*}
\end{proof}

\begin{prop}
If $e$ and $f$ are $(\varepsilon,r,N)$-idempotents in a unital filtered Banach algebra $A$, then \[||\kappa_0(e)-\kappa_0(f)||<\frac{(N+1)^2}{\sqrt{\varepsilon}(1-2\sqrt{\varepsilon})^2}||e-f||.\]
In particular, if $||e-f||<\frac{\varepsilon}{9(2N+3)(N+1)^2}$ and $0<\varepsilon<\frac{1}{9}$, then $\kappa_0(e)$ and $\kappa_0(f)$ are homotopic idempotents.

\end{prop}

\begin{proof}
We have
\begin{align*} ||\kappa_0(e)-\kappa_0(f)||&=\left\Vert\frac{1}{2\pi i}\int_\gamma\kappa_0(z)[(z-e)^{-1}-(z-f)^{-1}]dz\right\Vert \\ &= \left\Vert\frac{1}{2\pi i}\int_{\gamma_1}(z-e)^{-1}(e-f)(z-f)^{-1}dz\right\Vert \\ &<\sqrt{\varepsilon}\max_{z\in\gamma_1}||(z-e)^{-1}(e-f)(z-f)^{-1}|| \\ &\leq\frac{(N+1)^2}{\sqrt{\varepsilon}(1-2\sqrt{\varepsilon})^2}||e-f|| \end{align*}.

If $||e-f||<\frac{\varepsilon}{9(2N+3)(N+1)^2}$ and $0<\varepsilon<\frac{1}{9}$, then since \[||2\kappa_0(e)-1||<\frac{2N+3-2\sqrt{\varepsilon}}{1-2\sqrt{\varepsilon}}<6N+7,\] we have \[||\kappa_0(e)-\kappa_0(f)||<\frac{\sqrt{\varepsilon}}{9(2N+3)(1-2\sqrt{\varepsilon})^2}<\frac{1}{6N+9}<\frac{1}{||2\kappa_0(e)-1||}\] so $\kappa_0(e)$ and $\kappa_0(f)$ are homotopic idempotents \cite[Proposition 4.3.2]{Bl}.
\end{proof}

If $A$ is a filtered $SQ_p$ algebra, and $e$ is an $(\varepsilon,r,N)$-idempotent in $M_n(\tilde{A})$, then we may apply the holomorphic functional calculus to get an idempotent $\kappa_0(e)$ in $M_n(\tilde{A})$. This gives us a group homomorphism $K_0^{\varepsilon,r,N}(A)\rightarrow K_0(A)$ given by $[e]\mapsto[\kappa_0(e)]$. Also, since every $(\varepsilon,r,N)$-invertible is actually invertible, we have a homomorphism $K_1^{\varepsilon,r,N}(A)\rightarrow K_1(A)$ given by $[u]_{\varepsilon,r,N}\mapsto[u]$. We will denote this homomorphism by $\kappa_1$. These homomorphisms allow us to represent elements in $K_0(A)$ and $K_1(A)$ in terms of quasi-idempotents and quasi-invertibles respectively.

\begin{prop}\leavevmode \label{qKtoKsurj}
\begin{enumerate}
\item Let $A$ be a filtered $SQ_p$ algebra. Let $f$ be an idempotent in $M_n(\tilde{A})$, and let  $0<\varepsilon<\frac{1}{20}$. Then there exist $r>0$, $N\geq 1$, and $[e]\in K_0^{\varepsilon,r,N}(A)$ with $e\in Idem_n^{\varepsilon,r,N}(\tilde{A})$ such that $[\kappa_0(e)]=[f]$ in $K_0(A)$. 
\item Let $A$ be a filtered $SQ_p$ algebra. Let $u$ be an invertible element in $M_n(\tilde{A})$, and let $0<\varepsilon<\frac{1}{20}$. Then there exist $r>0$, $N\geq 1$, and $[v]\in K_1^{\varepsilon,r,N}(A)$ with $v\in GL_n^{\varepsilon,r,N}(\tilde{A})$ such that $[v]=[u]$ in $K_1(A)$.
\end{enumerate}
\end{prop}

\begin{proof}\leavevmode
\begin{enumerate}
\item Let $N=||f||+1$. There exist $r>0$ and $e\in M_n(\tilde{A}_r)$ such that $||e-f||<\frac{\varepsilon}{9N(N+1)^2}$. Then $e$ is an $(\varepsilon,r,N)$-idempotent in $M_n(\tilde{A})$ by Lemma \ref{normestlem3}. Moreover, 
\begin{align*} ||\kappa_0(e)-f||&=||\kappa_0(e)-\kappa_0(f)|| \\ &<\frac{(N+1)^2}{\sqrt{\varepsilon}(1-2\sqrt{\varepsilon})^2}\frac{\varepsilon}{9N(N+1)^2} \\ &<\frac{1}{2N}<\frac{1}{||2f-1||} \end{align*}
so $\kappa_0(e)$ and $f$ are homotopic as idempotents \cite[Proposition 4.3.2]{Bl}. When $A$ is non-unital, by increasing $N$ if necessary, we get $[\pi(e)]=0$ in $K_0^{\varepsilon,r,N}(\mathbb{C})$ so that $[e]\in K_0^{\varepsilon,r,N}(A)$.
\item Let $N=||u||+||u^{-1}||+1$. There exist $r>0$ and $v,v'\in M_n(\tilde{A}_r)$ such that $||v-u||<\frac{\varepsilon}{N}$ and $||v'-u^{-1}||<\frac{\varepsilon}{N}$. Then $v\in GL_n^{\varepsilon,r,N}(\tilde{A})$ by Lemma \ref{normestlem3}. Moreover, since $||v-u||<\frac{1}{||u^{-1}||}$, we have $[v]=[u]$ in $K_1(A)$ \cite[Lemma 4.2.1]{WO}. When $A$ is non-unital, we may assume that $\pi(u)\sim I_n$. Then, by increasing $N$ if necessary, we get $[\pi(v)]=[I_n]$ in $K_1^{\varepsilon,r,N}(\mathbb{C})$ so that $[v]\in K_1^{\varepsilon,r,N}(A)$.
\end{enumerate}
\end{proof}

\begin{prop} \label{qKtoKinj} \leavevmode
\begin{enumerate}
\item There exists a (quadratic) polynomial $\rho$ with positive coefficients such that for any filtered $SQ_p$ algebra $A$, if $0<\varepsilon<\frac{1}{20\rho(N)}$, and $[e]_{\varepsilon,r,N},[f]_{\varepsilon,r,N}\in K_0^{\varepsilon,r,N}(A)$ are such that $[\kappa_0(e)]=[\kappa_0(f)]$ in $K_0(A)$, then there exist $r'\geq r$ and $N'\geq N$ such that $[e]_{\rho(N)\varepsilon,r',N'}=[f]_{\rho(N)\varepsilon,r',N'}$ in $K_0^{\rho(N)\varepsilon,r',N'}(A)$.
\item Let $A$ be a filtered $SQ_p$ algebra. If $0<\varepsilon<\frac{1}{20}$, and $[u]_{\varepsilon,r,N},[v]_{\varepsilon,r,N}\in K_1^{\varepsilon,r,N}(A)$ are such that $[u]=[v]$ in $K_1(A)$, then there exist $r'\geq r$ and $N'\geq N$ such that $[u]_{\varepsilon,r',N'}=[v]_{\varepsilon,r',N'}$ in $K_1^{\varepsilon,r',N'}(A)$.
\end{enumerate}
\end{prop}

\begin{proof}\leavevmode
\begin{enumerate}
\item Let $(p_t)_{t\in[0,1]}$ be a homotopy of idempotents in $M_n(\tilde{A})$ between $\kappa_0(e)$ and $\kappa_0(f)$. Then $P:=(p_t)$ is an idempotent in $C([0,1],M_n(\tilde{A}))$. There exist $r'\geq r$ and $E:=(e_t)\in C([0,1],M_n(\tilde{A}_{r'}))$ such that $||E-P||<\frac{\varepsilon}{4N'}$, where $N'=\max(N,||P||+1)$.
In particular, we have $||e_0-\kappa_0(e)||<\frac{\varepsilon}{4N'}$ and $||e_1-\kappa_0(f)||<\frac{\varepsilon}{4N'}$. By Lemma \ref{normestlem3}, $e_t$ is an $(\varepsilon,r',N')$-idempotent in $M_n(\tilde{A})$ for each $t\in[0,1]$. Also 
\begin{align*} ||e_0-e|| &\leq||e_0-\kappa_0(e)||+||\kappa_0(e)-e|| \\ &<\frac{\varepsilon}{4N'}+\frac{2(N+1)\varepsilon}{(1-\sqrt{\varepsilon})(1-2\sqrt{\varepsilon})} \\ &<(6N+7)\varepsilon \end{align*}
and similarly $||e_1-f||<(6N+7)\varepsilon$. By Lemma \ref{normestlem2}, $e_0$ and $e$ are $(\varepsilon',r',N')$-homotopic, where $\varepsilon'=\varepsilon+\frac{1}{4}(6N+7)^2\varepsilon^2$, and similarly for $e_1$ and $f$. Hence $[e]_{\varepsilon',r',N'}=[f]_{\varepsilon',r',N'}$.

\item Let $(u_t)_{t\in[0,1]}$ be a homotopy of invertibles in $M_n(\tilde{A})$ between $u$ and $v$. We may regard $U=(u_t)$ as an invertible element in $C([0,1],M_n(\tilde{A}))$. 
Let $N'=\max(N,||U||+||U^{-1}||+1)$. There exist $r'\geq r$ and $W\in C([0,1],M_n(\tilde{A}_{r'}))$ such that \[\quad\quad\quad ||W-U||<\frac{\varepsilon-\max(||uu'-1||,||u'u-1||,||vv'-1||,||v'v-1||)}{N'},\] where $u'$ is an $(\varepsilon,r,N)$-inverse for $u$, and $v'$ is an $(\varepsilon,r,N)$-inverse for $v$. Then $W$ is an $(\varepsilon,r',N')$-invertible in $C([0,1],M_n(\tilde{A}))$ by Lemma \ref{normestlem3}, and we have a homotopy of $(\varepsilon,r',N')$-invertibles $u\sim W_0\sim W_1\sim v$ by Lemma \ref{normestlem1}.
\end{enumerate}
\end{proof}

\section{Controlled Long Exact Sequence in Quantitative $K$-Theory}

In this section, we establish a controlled long exact sequence in quantitative $K$-theory analogous to the one in ordinary $K$-theory. The proofs are adaptations of those in ordinary $K$-theory (cf. \cite{Bl}) and also those in quantitative $K$-theory for filtered $C^*$-algebras (cf. \cite{OY15}). First, we introduce some terminology, adapted from \cite{OY15} and \cite{OY}, that will allow us to describe the functorial properties of quantitative $K$-theory.

\subsection{Controlled morphisms and controlled exact sequences}


The following notion of a control pair provides a convenient way to keep track of increases in the parameters associated with quantitative $K$-theory groups. It has already appeared in some of the earlier results, and we now give a formal definition.

\begin{defn}
A control pair is a pair $(\lambda,h)$ such that
\begin{itemize}
\item $\lambda:[1,\infty)\rightarrow[1,\infty)$ is a non-decreasing function;
\item $h:(0,\frac{1}{20})\times[1,\infty)\rightarrow[1,\infty)$ is a function such that $h(\cdot,N)$ is non-increasing for fixed $N$. \end{itemize}
We will write $\lambda_N$ for $\lambda(N)$, and $h_{\varepsilon,N}$ for $h(\varepsilon,N)$.

Given two control pairs $(\lambda,h)$ and $(\lambda',h')$, we write $(\lambda,h)\leq(\lambda',h')$ if $\lambda_N\leq \lambda'_N$ and $h_{\varepsilon,N}\leq h'_{\varepsilon,N}$ for all $\varepsilon\in(0,\frac{1}{20})$ and $N\geq 1$.
\end{defn}

\begin{rem}
These functions will appear as coefficients attached to the parameters $\varepsilon,r$, and $N$. One may choose to use three functions, one for each of the parameters, but we have chosen to use just two to reduce notational clutter, with $\lambda$ controlling both $\varepsilon$ and $N$.
\end{rem}

Given a filtered $SQ_p$ algebra $A$, we consider the families 
\begin{align*}
\mathcal{K}_0(A)&=(K_0^{\varepsilon,r,N}(A))_{0<\varepsilon<\frac{1}{20},r>0,N\geq 1}, \\
\mathcal{K}_1(A)&=(K_1^{\varepsilon,r,N}(A))_{0<\varepsilon<\frac{1}{20},r>0,N\geq 1}.
\end{align*}

\begin{defn}
Let $A$ and $B$ be filtered $SQ_p$ algebras, and let $(\lambda,h)$ be a control pair. A $(\lambda,h)$-controlled morphism $\mathcal{F}:\mathcal{K}_i(A)\rightarrow\mathcal{K}_j(B)$, where $i,j\in\{0,1\}$, is a family \[\mathcal{F}=(F^{\varepsilon,r,N})_{0<\varepsilon<\frac{1}{20\lambda_N},r>0,N\geq 1}\] of group homomorphisms \[F^{\varepsilon,r,N}:K_i^{\varepsilon,r,N}(A)\rightarrow K_j^{\lambda_N\varepsilon,h_{\varepsilon,N}r,\lambda_N}(B)\] such that whenever $0<\varepsilon\leq\varepsilon'<\frac{1}{20\lambda_{N'}}$, $h_{\varepsilon,N} r\leq h_{\varepsilon',N'}r'$, and $N\leq N'$, we have the following commutative diagram:
\[
\begindc{\commdiag}[100]		
\obj(10,5)[2a]{$K_i^{\varepsilon,r,N}(A)$}	\obj(25,5)[2b]{$K_i^{\varepsilon',r',N'}(A)$}	

\mor{2a}{2b}{$\iota_i$}	 

\obj(10,0)[3a]{$K_j^{\lambda_N\varepsilon,h_{\varepsilon,N} r,\lambda_N}(B)$}	\obj(25,0)[3b]{$K_j^{\lambda_{N'}\varepsilon',h_{\varepsilon',N'}r',\lambda_{N'}}(B)$}

\mor{2a}{3a}{$F^{\varepsilon,r,N}$}[\atright,\solidarrow]	\mor{3a}{3b}{$\iota_j$} 	\mor{2b}{3b}{$F^{\varepsilon',r',N'}$}
\enddc
\]
We say that $\mathcal{F}$ is a controlled morphism if it is a $(\lambda,h)$-controlled morphism for some control pair $(\lambda,h)$.
\end{defn}

In some cases, the family of group homomorphisms $F^{\varepsilon,r,N}$ may only be defined for values of $r$ within some finite interval rather than for all $r>0$, for example the boundary homomorphism in our controlled Mayer-Vietoris sequence in section \ref{sec:MV}. Thus we make the following refinement of the above definition.

\begin{defn}
Let $A$ and $B$ be filtered $SQ_p$ algebras, let $(\lambda,h)$ be a control pair, and let $R>0$. A $(\lambda,h)$-controlled morphism $\mathcal{F}:\mathcal{K}_i(A)\rightarrow\mathcal{K}_j(B)$ of order $R$, where $i,j\in\{0,1\}$, is a family \[\mathcal{F}=(F^{\varepsilon,r,N})_{0<\varepsilon<\frac{1}{20\lambda_N},0<r\leq\frac{R}{h_{\varepsilon,N}},N\geq 1}\] of group homomorphisms \[F^{\varepsilon,r,N}:K_i^{\varepsilon,r,N}(A)\rightarrow K_j^{\lambda_N\varepsilon,h_{\varepsilon,N}r,\lambda_N}(B)\] such that whenever $0<\varepsilon\leq\varepsilon'<\frac{1}{20\lambda_{N'}}$, $h_{\varepsilon,N}r\leq h_{\varepsilon',N'}r'\leq R$, and $N\leq N'$, we have the same commutative diagram as above.
\end{defn}

\begin{rem}
The definitions in the rest of this section will be stated in terms of controlled morphisms but they have obvious extensions to the setting of controlled morphisms of a given order.
\end{rem}

Given a filtered $SQ_p$ algebra $A$, we will denote by $\mathcal{I}d_{\mathcal{K}_i(A)}$ the $(1,1)$-controlled morphism given by the family $(Id_{K_i^{\varepsilon,r,N}(A)})_{0<\varepsilon<\frac{1}{20},r>0,N\geq 1}$, where $i\in\{0,1\}$.

A filtered homomorphism between filtered $SQ_p$ algebras $A$ and $B$ will induce a controlled morphism between $\mathcal{K}_*(A)$ and $\mathcal{K}_*(B)$. Moreover, such a controlled morphism will induce a homomorphism in $K$-theory.

\begin{prop}
Let $A$ and $B$ be filtered $SQ_p$ algebras. If $\mathcal{F}=(F^{\varepsilon,r,N}):\mathcal{K}_i(A)\rightarrow\mathcal{K}_j(B)$ is a controlled morphism, then there is a unique group homomorphism $F:K_i(A)\rightarrow K_j(B)$ satisfying \[F([\kappa_i(x)])=[\kappa_j(F^{\varepsilon,r,N}([x]))]\] for all $0<\varepsilon<\frac{1}{20}$, $r>0$, $N\geq 1$, and $[x]\in K_i^{\varepsilon,r,N}(A)$.

Moreover, if $\mathcal{F}:\mathcal{K}_i(A)\rightarrow\mathcal{K}_i(B)$ is induced by a filtered homomorphism $\phi:A\rightarrow B$, then $F=\phi_*:K_i(A)\rightarrow K_i(B)$.
\end{prop}

\begin{proof}
Given $[f]\in K_0(A)$ and $0<\varepsilon<\frac{1}{20}$, there exist $r>0$, $N\geq 1$, and $e\in Idem_n^{\varepsilon,r,N}(\tilde{A})$ such that $[\kappa_0(e)]=[f]$. Then define \[F([f])=[\kappa_j(F^{\varepsilon,r,N}([e]))]\in K_j(B).\] This is well-defined by Propositions \ref{qKtoKsurj} and \ref{qKtoKinj}. The proof in the odd case is similar, and the last statement follows from the definition of $F$.
\end{proof}

Let $A,B$, and $C$ be filtered $SQ_p$ algebras, and let $i,j,l\in\{0,1\}$. Suppose that $\mathcal{F}:\mathcal{K}_i(A)\rightarrow\mathcal{K}_j(B)$ is a $(\lambda,h)$-controlled morphism, and that $\mathcal{G}:\mathcal{K}_j(B)\rightarrow\mathcal{K}_l(C)$ is a $(\lambda',h')$-controlled morphism. Then we denote by $\mathcal{G}\circ\mathcal{F}:\mathcal{K}_i(A)\rightarrow\mathcal{K}_l(C)$ the family \[(G^{\lambda_N\varepsilon,h_{\varepsilon,N} r,\lambda_N}\circ F^{\varepsilon,r,N})_{0<\varepsilon<\frac{1}{20\lambda'_{\lambda_N}\lambda_N},r>0,N\geq 1}.\] Note that $\mathcal{G}\circ\mathcal{F}$ is a $(\lambda'',h'')$-controlled morphism, where $\lambda''_N=\lambda'_{\lambda_N}\lambda_N$ and $h''_{\varepsilon,N}=h'_{\lambda_N\varepsilon,\lambda_N}h_{\varepsilon,N}$.

Hereafter, given two control pairs $(\lambda,h)$ and $(\lambda',h')$, we will write $(\lambda'\cdot\lambda)_N$ for $\lambda'_{\lambda_N}\lambda_N$, and $(h'\cdot h)_{\varepsilon,N}$ for $h'_{\lambda_N\varepsilon,\lambda_N}h_{\varepsilon,N}$.

\begin{defn}
Let $A$ and $B$ be filtered $SQ_p$ algebras. Let $\mathcal{F}:\mathcal{K}_i(A)\rightarrow\mathcal{K}_j(B)$ and $\mathcal{G}:\mathcal{K}_i(A)\rightarrow\mathcal{K}_j(B)$ be $(\lambda^\mathcal{F},h^\mathcal{F})$-controlled and $(\lambda^\mathcal{G},h^\mathcal{G})$-controlled morphisms respectively. Let $(\lambda,h)$ be a control pair. We write $\mathcal{F}\stackrel{(\lambda,h)}{\sim}\mathcal{G}$ if $(\lambda^\mathcal{F},h^\mathcal{F})\leq(\lambda,h)$, $(\lambda^\mathcal{G},h^\mathcal{G})\leq(\lambda,h)$, and the following diagram commutes whenever $0<\varepsilon<\frac{1}{20\lambda_N}$, $r>0$, and $N\geq 1$:
\[
\begindc{\commdiag}[100]		
\obj(0,3)[2a]{$K_i^{\varepsilon,r,N}(A)$}	\obj(26,3)[2c]{$K_j^{\lambda_N\varepsilon,h_{\varepsilon,N}r,\lambda_N}(B)$}	

\obj(13,6)[1b]{$K_j^{\lambda^\mathcal{F}_N\varepsilon,h^\mathcal{F}_{\varepsilon,N} r,\lambda^\mathcal{F}_N}(B)$}	
\obj(13,0)[3b]{$K_j^{\lambda^\mathcal{G}_N\varepsilon,h^\mathcal{G}_{\varepsilon,N}r,\lambda^\mathcal{G}_N}(B)$}

\mor{2a}{1b}{$F^{\varepsilon,r,N}$}	 \mor{2a}{3b}{$G^{\varepsilon,r,N}$}

\mor{1b}{2c}{$\iota_j$}	\mor{3b}{2c}{$\iota_j$} 	
\enddc
\]
\end{defn}

Observe that if $\mathcal{F}\stackrel{(\lambda,h)}{\sim}\mathcal{G}$ for some control pair $(\lambda,h)$, then $\mathcal{F}$ and $\mathcal{G}$ induce the same homomorphism in $K$-theory.

\begin{defn}
Let $A,B,C$, and $D$ be filtered $SQ_p$ algebras. Let $\mathcal{F}:\mathcal{K}_i(A)\rightarrow\mathcal{K}_j(B)$, $\mathcal{F}':\mathcal{K}_i(A)\rightarrow\mathcal{K}_l(C)$, $\mathcal{G}:\mathcal{K}_j(B)\rightarrow\mathcal{K}_m(D)$, and $\mathcal{G}':\mathcal{K}_l(C)\rightarrow\mathcal{K}_m(D)$ be controlled morphisms, where $i,j,l,m\in\{0,1\}$, and let $(\lambda,h)$ be a control pair. We say that the diagram
\[
\begindc{\commdiag}[100]		
\obj(0,5)[1a]{$\mathcal{K}_i(A)$}		\obj(10,5)[1b]{$\mathcal{K}_j(B)$}	

\obj(0,0)[2a]{$\mathcal{K}_l(C)$}	
\obj(10,0)[2b]{$\mathcal{K}_m(D)$}

\mor{1a}{1b}{$\mathcal{F}$}	 \mor{1b}{2b}{$\mathcal{G}$}

\mor{1a}{2a}{$\mathcal{F}'$}[\atright,\solidarrow]	\mor{2a}{2b}{$\mathcal{G}'$} 	
\enddc
\]
is $(\lambda,h)$-commutative if $\mathcal{G}\circ\mathcal{F}\stackrel{(\lambda,h)}{\sim}\mathcal{G}'\circ\mathcal{F}'$.
\end{defn}

\begin{defn}
Let $A$ and $B$ be filtered $SQ_p$ algebras. Let $(\lambda,h)$ be a control pair, and let $\mathcal{F}:\mathcal{K}_i(A)\rightarrow\mathcal{K}_j(B)$ be a $(\lambda^\mathcal{F},h^\mathcal{F})$-controlled morphism with $(\lambda^\mathcal{F},h^\mathcal{F})\leq(\lambda,h)$.
\begin{itemize}
\item We say that $\mathcal{F}$ is left (resp. right) $(\lambda,h)$-invertible if there exists a controlled morphism $\mathcal{G}:\mathcal{K}_j(B)\rightarrow\mathcal{K}_i(A)$ such that $\mathcal{G}\circ\mathcal{F}\stackrel{(\lambda,h)}{\sim}\mathcal{I}d_{\mathcal{K}_i(A)}$ (resp. $\mathcal{F}\circ\mathcal{G}\stackrel{(\lambda,h)}{\sim}\mathcal{I}d_{\mathcal{K}_j(B)}$). In this case, we call $\mathcal{G}$ a left (resp. right) $(\lambda,h)$-inverse for $\mathcal{F}$.
\item We say that $\mathcal{F}$ is $(\lambda,h)$-invertible or a $(\lambda,h)$-isomorphism if there exists a controlled morphism $\mathcal{G}:\mathcal{K}_j(B)\rightarrow\mathcal{K}_i(A)$ that is both a left $(\lambda,h)$-inverse and a right $(\lambda,h)$-inverse for $\mathcal{F}$. In this case, we call $\mathcal{G}$ a $(\lambda,h)$-inverse for $\mathcal{F}$.

We say that $\mathcal{F}$ is a controlled isomorphism if it is a $(\lambda,h)$-isomorphism for some control pair $(\lambda,h)$.
\end{itemize}
\end{defn}

Note that if $\mathcal{F}$ is left $(\lambda,h)$-invertible and right $(\lambda,h)$-invertible, then there exists a control pair $(\lambda',h')\geq(\lambda,h)$, depending only on $(\lambda,h)$, such that $\mathcal{F}$ is $(\lambda',h')$-invertible. Also, a controlled isomorphism will induce an isomorphism in $K$-theory.

\begin{defn}
Let $A$ and $B$ be filtered $SQ_p$ algebras. Let $(\lambda,h)$ be a control pair, and let $\mathcal{F}:\mathcal{K}_i(A)\rightarrow\mathcal{K}_j(B)$ be a $(\lambda^\mathcal{F},h^\mathcal{F})$-controlled morphism.
\begin{itemize}
\item We say that $\mathcal{F}$ is $(\lambda,h)$-injective if $(\lambda^\mathcal{F},h^\mathcal{F})\leq(\lambda,h)$, and for any $0<\varepsilon<\frac{1}{20\lambda_N}$, $r>0$, $N\geq 1$, and $x\in K_i^{\varepsilon,r,N}(A)$, if $F^{\varepsilon,r,N}(x)=0$ in $K_j^{\lambda^\mathcal{F}_N\varepsilon,h^\mathcal{F}_{\varepsilon,N}r,\lambda^\mathcal{F}_N}(B)$, then $\iota_i(x)=0$ in $K_i^{\lambda_N\varepsilon,h_{\varepsilon,N}r,\lambda_N}(A)$.
\item We say that $\mathcal{F}$ is $(\lambda,h)$-surjective if for any $0<\varepsilon<\frac{1}{20(\lambda^\mathcal{F}\cdot\lambda)_N}$, $r>0$, $N\geq 1$, and $y\in K_j^{\varepsilon,r,N}(B)$, there exists $x\in K_i^{\lambda_N\varepsilon,h_{\varepsilon,N}r,\lambda_N}(A)$ such that $F^{\lambda_N\varepsilon,h_{\varepsilon,N}r,\lambda_N}(x)=\iota_j(y)$ in $K_j^{(\lambda^\mathcal{F}\cdot\lambda)_N\varepsilon,(h^\mathcal{F}\cdot h)_{\varepsilon,N}r,(\lambda^\mathcal{F}\cdot\lambda)_N}(B)$.
\end{itemize}
\end{defn}

It is clear from the definitions that if $\mathcal{F}$ is left $(\lambda,h)$-invertible, then $\mathcal{F}$ is $(\lambda,h)$-injective. If $\mathcal{F}$ is right $(\lambda,h)$-invertible, then there exists a control pair $(\lambda',h')\geq(\lambda,h)$, depending only on $(\lambda,h)$, such that $\mathcal{F}$ is $(\lambda',h')$-surjective. On the other hand, if $\mathcal{F}$ is both $(\lambda,h)$-injective and $(\lambda,h)$-surjective, then there exists a control pair $(\lambda',h')\geq(\lambda,h)$, depending only on $(\lambda,h)$, such that $\mathcal{F}$ is a $(\lambda',h')$-isomorphism.

\begin{defn}
Let $A,B$, and $C$ be filtered $SQ_p$ algebras, and let $(\lambda,h)$ be a control pair. Let $\mathcal{F}:\mathcal{K}_i(A)\rightarrow\mathcal{K}_j(B)$ be a $(\lambda^\mathcal{F},h^\mathcal{F})$-controlled morphism, and let $\mathcal{G}:\mathcal{K}_j(B)\rightarrow\mathcal{K}_l(C)$ be a $(\lambda^\mathcal{G},h^\mathcal{G})$-controlled morphism, where $i,j,l\in\{0,1\}$. Then the composition \[\mathcal{K}_i(A)\stackrel{\mathcal{F}}{\rightarrow}\mathcal{K}_j(B)\stackrel{\mathcal{G}}{\rightarrow}\mathcal{K}_l(C)\] is said to be $(\lambda,h)$-exact (at $\mathcal{K}_j(B)$) if 
\begin{itemize}
\item $\mathcal{G}\circ\mathcal{F}=0$;
\item for any $0<\varepsilon<\frac{1}{20\max((\lambda^\mathcal{F}\cdot\lambda)_N,\lambda^\mathcal{G}_N)}$, $r>0$, $N\geq 1$, and $y\in K_j^{\varepsilon,r,N}(B)$ such that $G^{\varepsilon,r,N}(y)=0$ in $K_l^{\lambda^\mathcal{G}_N\varepsilon,h^\mathcal{G}_{\varepsilon,N}r,\lambda^\mathcal{G}_N}(C)$, there exists $x\in K_i^{\lambda_N\varepsilon,h_{\varepsilon,N}r,\lambda_N}(A)$ such that $F^{\lambda_N\varepsilon,h_{\varepsilon,N}r,\lambda_N}(x)=\iota_j(y)$ in $K_j^{(\lambda^\mathcal{F}\cdot\lambda)_N\varepsilon,(h^\mathcal{F}\cdot h)_{\varepsilon,N}r,(\lambda^\mathcal{F}\cdot\lambda)_N}(B)$. 
\end{itemize}
A sequence of controlled morphisms
\[ \cdots\rightarrow\mathcal{K}_{i_{k-1}}(A_{k-1})\rightarrow\mathcal{K}_{i_k}(A_k)\rightarrow\mathcal{K}_{i_{k+1}}(A_{k+1})\rightarrow\mathcal{K}_{i_{k+2}}(A_{k+2})\rightarrow\cdots \]
is said to be $(\lambda,h)$-exact if the composition $\mathcal{K}_{i_{k-1}}(A_{k-1})\rightarrow\mathcal{K}_{i_k}(A_k)\rightarrow\mathcal{K}_{i_{k+1}}(A_{k+1})$ is $(\lambda,h)$-exact for every $k$.
\end{defn}

Note that controlled exact sequences in quantitative $K$-theory induce exact sequences in $K$-theory.


\subsection{Completely filtered extensions of Banach algebras}

Let $A$ be a filtered Banach algebra with filtration $(A_r)_{r> 0}$, and let $J$ be a closed ideal in $A$. Then $A/J$ has filtration \[(q(A_r))_{r>0}=((A_r+J)/J)_{r> 0},\] where $q:A\rightarrow A/J$ is the quotient homomorphism. We will consider extensions of filtered Banach algebras in which the ideal $J$ has the natural filtration inherited from the filtration for $A$, and for which we can perform controlled lifting from $A/J$ to $A$.

For $r>0$, let $J_r=J\cap A_r$. Suppose there exists $C\geq 1$ such that $\inf_{y\in J_r}||x+y||\leq C\inf_{y\in J}||x+y||$ for all $r>0$ and $x\in A_r$. For any $y\in J$ and $\varepsilon>0$, there exist $r>0$ and $a\in A_r$ such that $||a-y||<\frac{\varepsilon}{C+1}$. Then there exists $z\in J_r$ such that $||a-z||<\frac{C\varepsilon}{C+1}$, and $||y-z||<\varepsilon$. Hence $(J_r)_{r>0}$ is a filtration for $J$.

\begin{defn}
Let $A$ be a filtered $SQ_p$ algebra with filtration $(A_r)_{r> 0}$, and let $J$ be a closed ideal of $A$. The extension of Banach algebras \[0\rightarrow J\rightarrow A\rightarrow A/J\rightarrow 0\] is called a $C$-completely filtered extension of $SQ_p$ algebras if there exists $C\geq 1$ such that for any $n\in\mathbb{N}$, $r>0$, and $x\in M_n(A_r)$, we have \[\inf_{y\in M_n(J_r)}||x+y||\leq C\inf_{y\in M_n(J)}||x+y||.\]
\end{defn}

\begin{rem}\leavevmode
\begin{enumerate}
\item If $A$ is a non-unital filtered $SQ_p$ algebra, $J$ is a closed ideal of $A$, and the extension $0\rightarrow J\rightarrow A\rightarrow A/J\rightarrow 0$ is $C$-completely filtered, then the extension \[0\rightarrow J\rightarrow \tilde{A}\rightarrow\tilde{A}/J\rightarrow 0\] is $3C$-completely filtered.
\item If $0\rightarrow J\rightarrow A\rightarrow A/J\rightarrow 0$ is a $C$-completely filtered extension of filtered $SQ_p$ algebras, then the suspended extension \[0\rightarrow SJ\rightarrow SA\rightarrow S(A/J)\rightarrow 0\] is $3C$-completely filtered.
\end{enumerate}
\end{rem}

%
%

A particular class of completely filtered extensions are the extensions
\[0\rightarrow J\rightarrow A\stackrel{q}{\rightarrow} A/J\rightarrow 0\] 
that admit $p$-completely contractive sections $s:A/J\rightarrow A$ with $s(q(A_r))\subset A_r$ for all $r>0$.


\begin{ex}
If $A$ is a filtered $SQ_p$ algebra, then the cone $CA$ and the suspension $SA$ have filtrations induced by the filtration for $A$. The extension $0\rightarrow SA\rightarrow CA\rightarrow A\rightarrow 0$ admits a $p$-completely contractive section $s$ such that $s(q((CA)_r))\subset (CA)_r$ for all $r>0$.
\end{ex}

\begin{lem}
Extensions that admit $p$-completely contractive sections $s$ with $s(q(A_r))\subset A_r$ for all $r>0$ are $1$-completely filtered.
\end{lem}

\begin{proof}
Let $x\in M_n(A_r)$. Since $s(q(A_r))\subset A_r$, there exists $z\in M_n(J_r)$ such that $s(q(x))=x+z$. Then \[||x+z||=||s(q(x))||\leq||q(x)||=\inf_{y\in M_n(J)}||x+y||\] so $\inf_{y\in M_n(J_r)}||x+y||\leq\inf_{y\in M_n(J)}||x+y||$.
\end{proof}

\subsection{Controlled half-exactness of $\mathcal{K}_0$ and $\mathcal{K}_1$}


\begin{lem} \label{invlift1}
Let $A$ be a unital filtered $SQ_p$ algebra. For any $C\geq 1$ and any $C$-completely filtered extension $0\rightarrow J\rightarrow A\stackrel{q}{\rightarrow}A/J\rightarrow 0$, for any $0<\varepsilon<\frac{1}{20}$, $r>0$, and $N\geq 1$, if $(u,v)$ is an $(\varepsilon,r,N)$-inverse pair in $M_n(A/J)$, then there exists an invertible $w$ in $M_{2n}(A)$ such that
\begin{itemize}
\item $w,w^{-1}\in M_{2n}(A_{3r})$;
\item $\max(||w||,||w^{-1}||)\leq (CN+\varepsilon+1)^3$; 
\item $\max(||q(w)-\diag(u,v)||,||q(w^{-1})-\diag(v,u)||)<(N+1)\varepsilon$;
\item $w$ is homotopic to $I_{2n}$ via a homotopy of invertible elements in $M_{2n}(A_{3r})$ with norm at most $\sqrt{2}(CN+\varepsilon+1)^3$, and likewise for $w^{-1}$.
\end{itemize}
\end{lem}

\begin{proof}
Given an $(\varepsilon,r,N)$-inverse pair $(u,v)$ in $M_n(A/J)$, there exist $U,V\in M_n(A_r)$ such that $q(U)=u$, $q(V)=v$, and $\max(||U||,||V||)<CN+\varepsilon$. Consider $w=\begin{pmatrix} I & U \\ 0 & I \end{pmatrix}\begin{pmatrix} I & 0 \\ -V & I \end{pmatrix}\begin{pmatrix} I & U \\ 0 & I \end{pmatrix}\begin{pmatrix} 0 & -I \\ I & 0 \end{pmatrix}\in M_{2n}(A_{3r})$. Then $w^{-1}
\in M_{2n}(A_{3r})$, and $\max(||w||,||w^{-1}||)
<(CN+\varepsilon+1)^3$. 

Moreover, $q(w)-\diag(u,v)=\begin{pmatrix} u(1-vu) & uv-1 \\ 1-vu & 0 \end{pmatrix}$ so \[||q(w)-\diag(u,v)||<(N+1)\varepsilon,\] and similarly $||q(w^{-1})-\diag(v,u)||<(N+1)\varepsilon$. Finally, \[w_t=\begin{pmatrix} I & tU \\ 0 & I \end{pmatrix}\begin{pmatrix} I & 0 \\ -tV & I \end{pmatrix}\begin{pmatrix} I & tU \\ 0 & I \end{pmatrix}\begin{pmatrix} \cos\frac{\pi t}{2} & -\sin\frac{\pi t}{2} \\ \sin\frac{\pi t}{2} & \cos\frac{\pi t}{2} \end{pmatrix}\] is a homotopy of invertible elements in $M_{2n}(A_{3r})$ between $w$ and $I_{2n}$, $w_t^{-1}$ is a homotopy of invertible elements in $M_{2n}(A_{3r})$ between $w^{-1}$ and $I_{2n}$, and $\max(||w_t||,||w_t^{-1}||)<\sqrt{2}(CN+\varepsilon+1)^3$ for all $t\in[0,1]$.
\end{proof}

\begin{prop}
For any $C\geq 1$, there exists a control pair $(\lambda,h)$ such that if $0\rightarrow J\stackrel{j}{\rightarrow} A\stackrel{q}{\rightarrow} A/J\rightarrow 0$ is a $C$-completely filtered extension of $SQ_p$ algebras, then we have a $(\lambda,h)$-exact sequence \[\mathcal{K}_0(J)\stackrel{j_*}{\rightarrow}\mathcal{K}_0(A)\stackrel{q_*}{\rightarrow}\mathcal{K}_0(A/J).\]
\end{prop}

\begin{proof}
Clearly the composition $K_0^{\varepsilon,r,N}(J)\rightarrow K_0^{\varepsilon,r,N}(A)\rightarrow K_0^{\varepsilon,r,N}(A/J)$ is the zero map. Let $[e]-[I_n]\in K_0^{\varepsilon,r,N}(A)$ be such that $[q(e)]-[I_n]=0$ in $K_0^{\varepsilon,r,N}(A/J)$, where $e$ is an $(\varepsilon,r,N)$-idempotent in $M_n(\tilde{A})$. Up to stabilization and relaxing control, we may assume that $q(e)$ and $I_n$ are $(\varepsilon,r,N)$-homotopic in $M_n(\tilde{A}/J)$.
By Proposition \ref{homtosim2} and Lemma \ref{idemliphom}, there exists a control pair $(\lambda,h)$ such that, up to stabilization, $||uq(e)v-I_n||<\lambda_N\varepsilon$ for some $(\lambda_N\varepsilon,h_{\varepsilon,N}r,\lambda_N)$-inverse pair $(u,v)$ in $M_n(\tilde{A}/J)$.

By Lemma \ref{invlift1}, there exists an invertible $w \in M_{2k}(\tilde{A}_{3h_{\varepsilon,N}r})$ with $w^{-1}\in M_{2k}(\tilde{A}_{3h_{\varepsilon,N}r})$ such that $\max(||w||,||w^{-1}||)\leq (C\lambda_N+\varepsilon+1)^3$, and $\max(||q(w)-\diag(u,v)||,||q(w^{-1})-\diag(v,u)||)<(\lambda_N+1)\varepsilon$. Set \[e'=w\diag(e,0)w^{-1}.\] Since $||q(e')-\diag(I_n,0)||<3N(\lambda_N+1)(C\lambda_N+\varepsilon+1)^3\varepsilon$, there exists $f\in M_{2n}(\tilde{J}_{(6h_{\varepsilon,N}+1)r})$ such that \[||f-e'||<3CN(\lambda_N+1)(C\lambda_N+\varepsilon+1)^3\varepsilon.\] By further enlarging the control pair $(\lambda,h)$ if necessary, and applying Lemma \ref{simtohom2} and Lemma \ref{normestlem1}, we get $[f]-[I_n]=[e]-[I_n]$ in $K_0^{\lambda_N\varepsilon,h_{\varepsilon,N}r,\lambda_N}(A)$.
\end{proof}


%

\begin{lem} \label{invlift2}
For any $C\geq 1$, there exists a control pair $(\lambda,h)$ such that for any $C$-completely filtered extension $0\rightarrow J\stackrel{j}{\rightarrow} A\stackrel{q}{\rightarrow} A/J\rightarrow 0$ of $SQ_p$ algebras with $A$ unital, if $u\in GL_n^{\varepsilon,r,N}(A/J)$ is $(\varepsilon,r,N)$-homotopic to $I_n$, then there exist $k\in\mathbb{N}$ and $a\in GL_{(2k+2)n}^{\lambda_N\varepsilon,h_{\varepsilon,N}r,\lambda_N}(A)$ homotopic to $I_{(2k+2)n}$ such that $q(a)=\diag(u,I_{(2k+1)n})$.
\end{lem}

\begin{proof}
Let $(u_t)_{t\in[0,1]}$ be a homotopy of $(\varepsilon,r,N)$-invertibles with $u_0=u$ and $u_1=I_n$. Let $0=t_0<t_1<\cdots<t_k=1$ be such that $||u_{t_i}-u_{t_{i-1}}||<\frac{\varepsilon}{N}$ for $i=1,\ldots,k$. For each $t$, let $u_t'$ be an $(\varepsilon,r,N)$-inverse for $u_t$ with $u_1'=I_n$. Set \begin{align*} V&=\diag(u_{t_0},\ldots,u_{t_k},u_{t_0}',\ldots,u_{t_k}'),\\ W&=\diag(I_n,u_{t_0}',\ldots,u_{t_{k-1}}',u_{t_0},\ldots,u_{t_k}).\end{align*} By Lemma \ref{invlift1}, there exists an invertible $v\in M_{(2k+2)n}(A_{3r})$ homotopic to $I_{(2k+2)n}$ such that $||v||<(CN+\varepsilon+1)^3$ and $||q(v)-V||<(N+1)\varepsilon$. Since \[W=\diag(P,I)\diag(u_{t_0}',\ldots,u_{t_k}',u_{t_0},\ldots,u_{t_k})\diag(P^{-1},I)\] for some permutation matrix $P$, by Lemma \ref{invlift1} again, there exists an invertible $w\in M_{(2k+2)n}(A_{3r})$ homotopic to $I_{(2k+2)n}$ such that $||w||<(CN+\varepsilon+1)^3$ and $||q(w)-W||<(N+1)\varepsilon$. Then $vw$ is homotopic to $I_{(2k+2)n}$ and \begin{align*} ||q(vw)-VW|| &\leq||(q(v)-V)q(w)||+||V(q(w)-W)|| \\ &<(N+1)((CN+\varepsilon+1)^3+N)\varepsilon.\end{align*}


Let $b\in M_n(A_r)$ be a lift of $u$. Then \[||q(\diag(b,I_{(2k+1)n})-vw)||<((N+1)((CN+\varepsilon+1)^3+N)+2)\varepsilon\] so there exists $d\in M_{(2k+2)n}(J_r)$ such that \[||\diag(b,I_{(2k+1)n})-vw+d||<((N+1)((CN+\varepsilon+1)^3+N)+2)C\varepsilon.\] Now $a=\diag(b,I_{(2k+1)n})+d$ has the desired properties.
\end{proof}

\begin{prop}
For any $C\geq 1$, there exists a control pair $(\lambda,h)$ such that if $0\rightarrow J\stackrel{j}{\rightarrow} A\stackrel{q}{\rightarrow} A/J\rightarrow 0$ is a $C$-completely filtered extension of $SQ_p$ algebras, then we have a $(\lambda,h)$-exact sequence \[\mathcal{K}_1(J)\stackrel{j_*}{\rightarrow}\mathcal{K}_1(A)\stackrel{q_*}{\rightarrow}\mathcal{K}_1(A/J).\]
\end{prop}

\begin{proof}
It is clear that the composition 
is the zero map. Let $[u]\in K_1^{\varepsilon,r,N}(A)$ with $u\in M_n(\tilde{A})$ and $[q(u)]=[I]\in K_1^{\varepsilon,r,N}(A/J)$. By Lemma \ref{invlift2}, there exist a control pair $(\lambda,h)$, $k\in\mathbb{N}$, and $w\in GL_{(2k+2)n}^{\lambda_N\varepsilon,h_{\varepsilon,N}r,\lambda_N}(\tilde{A})$ homotopic to $I_{(2k+2)n}$ such that $q(w)=\diag(q(u),I_{(2k+1)n})$.

Let $w'$ be a $(\lambda_N\varepsilon,h_{\varepsilon,N}r,\lambda_N)$-inverse for $w$. Then, up to relaxing control, $w'$ is homotopic to $I_{(2k+2)n}$, and $w'\diag(u,I_{(2k+1)n})$ is homotopic to $\diag(u,I_{(2k+1)n})$ as $(\lambda_N\varepsilon,h_{\varepsilon,N}r,\lambda_N)$-invertibles.
Since \begin{align*} &\quad\; ||q(w'\diag(u,I_{(2k+1)n}))-I_{(2k+2)n}|| \\ &\leq||q(w')|| ||\diag(q(u),I_{(2k+1)n})-q(w)||+||q(w'w-I_{(2k+2)n})|| \\ &<(\lambda_N^2+\lambda_N)\varepsilon,\end{align*} there exists $U\in M_{(2k+2)n}(\tilde{J}_{h_{\varepsilon,N}r})$ such that \[||U-w'\diag(u,I_{(2k+1)n})||<C(\lambda_N^2+\lambda_N)\varepsilon.\] Thus by further enlarging the control pair $(\lambda,h)$, we get $j_*([U])=[u]$ in $K_1^{\lambda_N\varepsilon,h_{\varepsilon,N}r,\lambda_N}(A)$.
\end{proof}

\subsection{Controlled boundary map and controlled long exact sequence}

\begin{prop}
For any $C\geq 1$, there exists a control pair $(\lambda^{\mathcal{D}},h^{\mathcal{D}})$ such that if $0\rightarrow J\rightarrow A\stackrel{q}{\rightarrow} A/J\rightarrow 0$ is a $C$-completely filtered extension of $SQ_p$ algebras, then there is a $(\lambda^{\mathcal{D}},h^{\mathcal{D}})$-controlled morphism \[\mathcal{D}_1=(\partial_1^{\varepsilon,r,N}):\mathcal{K}_1(A/J)\rightarrow\mathcal{K}_0(J)\] which induces the usual boundary map $\partial_1:K_1(A/J)\rightarrow K_0(J)$ in $K$-theory.
\end{prop}

\begin{proof}
Let $u\in GL_n^{\varepsilon,r,N}(\tilde{A}/J)$ and choose $v\in GL_m^{\varepsilon,r,N}(\tilde{A}/J)$ such that $\diag(u,v)$ is $(2\varepsilon,2r,2(N+\varepsilon))$-homotopic to $I_{n+m}$. One choice is to take $v\in GL_n^{\varepsilon,r,N}(\tilde{A}/J)$ to be an $(\varepsilon,r,N)$-inverse for $u$ by Lemma \ref{inversepairhomotopy}. By Lemma \ref{invlift2}, up to stabilization, there exist a control pair $(\lambda,h)$ and a $(\lambda_N\varepsilon,h_{\varepsilon,N}r,\lambda_N)$-invertible $w\in M_{n+m}(\tilde{A})$ with $q(w)=\diag(u,v)$. Let $w'$ be a $(\lambda_N\varepsilon,h_{\varepsilon,N}r,\lambda_N)$-inverse for $w$, and set \[x=w\diag(I_n,0)w'\in M_{n+m}(\tilde{A}).\] There is a control pair $(\lambda',h')\geq(\lambda,h)$ such that 
\begin{itemize}
\item $(w,w')$ is a $(\lambda'_N\varepsilon,h'_{\varepsilon,N}r,\lambda'_N)$-inverse pair, 
\item $||q(w')-\diag(u',v')||<\lambda'_N\varepsilon$, where $u'$ and $v'$ are $(\varepsilon,r,N)$-inverses for $u$ and $v$ respectively, 
\item $x$ is homotopic to $\diag(I_n,0)$ as $(\lambda'_N\varepsilon,h'_{\varepsilon,N}r,\lambda'_N)$-idempotents in $M_{n+m}(\tilde{A})$,  
\item $||q(x)-\diag(I_n,0)||<\lambda'_N\varepsilon$. 
\end{itemize}
Then there exists $y\in M_{n+m}(\tilde{A}_{h'_{\varepsilon,N}r}\cap J)$ such that \[||x-\diag(I_n,0)-y||<3C\lambda'_N\varepsilon\] so there is a control pair $(\lambda'',h'')\geq(\lambda',h')$ such that $y+\diag(I_n,0)$ is a $(\lambda''_N\varepsilon,h''_{\varepsilon,N}r,\lambda''_N)$-idempotent in $M_{n+m}(\tilde{J})$. As $(\lambda''_N\varepsilon,h''_{\varepsilon,N}r,\lambda''_N)$-idempotents in $M_{n+m}(\tilde{A})$, $y+\diag(I_n,0)$ is homotopic to $x$. We will define \[ \partial_1([u])=[y+\diag(I_n,0)]-[\diag(I_n,0)].\] It remains to be shown that there is a control pair $(\lambda^{\mathcal{D}},h^{\mathcal{D}})\geq(\lambda'',h'')$ so that $\partial_1:K_1^{\varepsilon,r,N}(A/J)\rightarrow K_0^{\lambda^{\mathcal{D}}_N\varepsilon,h^{\mathcal{D}}_{\varepsilon,N}r,\lambda^{\mathcal{D}}_N}(J)$ is a well-defined homomorphism. We do so in several steps as follows:

\begin{enumerate}
\item Suppose that $y'\in M_{n+m}(\tilde{A}_{h'_{\varepsilon,N}r}\cap J)$ also satisfies \[||x-\diag(I_n,0)-y'||<3C\lambda'_N\varepsilon.\] Then $||y-y'||<6C\lambda'_N\varepsilon$ so $y+\diag(I_n,0)$ and $y'+\diag(I_n,0)$ are homotopic as $(\lambda^{\mathcal{D}}_N\varepsilon,h^{\mathcal{D}}_{\varepsilon,N}r,\lambda^{\mathcal{D}}_N)$-idempotents for an appropriate control pair $(\lambda^{\mathcal{D}},h^{\mathcal{D}})$.

\item Suppose that $z\in M_{n+m}(\tilde{A})$ is another $(\lambda_N\varepsilon,h_{\varepsilon,N}r,\lambda_N)$-invertible such that $||q(z)-\diag(u,v)||<\lambda_N\varepsilon$, and $z'\in M_{n+m}(\tilde{A})$ is a $(\lambda_N\varepsilon,h_{\varepsilon,N}r,\lambda_N)$-inverse for $z$. Setting $x':=z\diag(I_n,0)z'$, there exists $y'\in M_{2n}(\tilde{A}_{h'_{\varepsilon,N}r}\cap J)$ such that \[||x'-\diag(I_n,0)-y'||<3C\lambda'_N\varepsilon.\] Since $||q(x)-q(x')||<2\lambda'_N\varepsilon$, we have $y''\in M_{n+m}(\tilde{A}_{h'_{\varepsilon,N}r}\cap J)$ such that $||x-x'-y''||<6C\lambda'_N\varepsilon$. Then \[||x-\diag(I_n,0)-y'-y''||<9C\lambda'_N\varepsilon.\] This reduces to the case in (i).

\item Replacing $u$ by $\diag(u,I_k)$ and replacing $v$ by $\diag(v,I_j)$, we get $z\in M_{n+m+k+j}(\tilde{A})$ such that \[||q(z)-\diag(u,I_k,v,I_j)||<\lambda_N\varepsilon.\] In fact, if we write $w$ as a block matrix $\begin{pmatrix} w_{11} & w_{12} \\ w_{21} & w_{22} \end{pmatrix}$, where $w_{11}$ has size $n\times n$ and $w_{22}$ has size $m\times m$, then by (ii), we may take $z=\begin{pmatrix} w_{11} & 0 & w_{12} & 0 \\ 0 & I_k & 0 & 0 \\ w_{21} & 0 & w_{22} & 0 \\ 0 & 0 & 0 & I_j  \end{pmatrix}$, and similarly for $z'$ corresponding to $w'$. Now $x':=z\diag(I_{n+k},0)z'$ and $\diag(w,I_{k+j})\diag(I_n,0,I_k,0)\diag(w',I_{k+j})=\diag(x,I_k,0)$ are homotopic. There exists $y'\in M_{n+m+k+j}(\tilde{A}_{h'_{\varepsilon,N}r}\cap J)$ such that \[||x'-\diag(I_{n+k},0)-y'||<3C\lambda'_N\varepsilon.\] Conjugating by permutation matrices, we obtain $y''\in M_{n+m+k+j}(\tilde{A}_{h'_{\varepsilon,N}r}\cap J)$ such that \[||\diag(x,I_k,0)-\diag(I_n,0,I_k,0)-y''||<3C\lambda'_N\varepsilon.\] Then $||y''-\diag(y,0)||<6C\lambda'_N\varepsilon$. It follows that there is an appropriate control pair $(\lambda^{\mathcal{D}},h^{\mathcal{D}})$ such that in $K_0^{\lambda^{\mathcal{D}}_N\varepsilon,h^\mathcal{D}_{\varepsilon,N}r,\lambda^{\mathcal{D}}_N}(J)$ we have 
\begin{align*}
&\quad\;\partial_1([\diag(u,I_m)]) \\ &= [y'+\diag(I_{n+k},0_{m+j})]-[\diag(I_{n+k},0_{m+j})] \\ 
&= [y''+\diag(I_n,0_{m+k+j})]-[\diag(I_n,0_{m+k+j})] \\
&= [\diag(y,0_{k+j})+\diag(I_n,0_{m+k+j})]-[\diag(I_n,0_{m+k+j})] \\
&= [y+\diag(I_n,0_m)]-[\diag(I_n,0_m)]=\partial_1([u]).
\end{align*}

\item Suppose that $u_0$ and $u_1$ are homotopic as $(\varepsilon,r,N)$-invertibles in $M_n(\tilde{A}/J)$, and that $v_i\in GL_m^{\varepsilon,r,N}(\tilde{A}/J)$ is such that $\diag(u_i,v_i)$ is $(2\varepsilon,2r,2(N+\varepsilon))$-homotopic to $I_{n+m}$ for $i=0,1$. Let $u_0'$ and $v_0'$ be $(\varepsilon,r,N)$-inverses for $u_0$ and $v_0$ respectively. Then there exists a control pair $(\lambda^1,h^1)\geq(\lambda,h)$ such that as $(\lambda^1_N\varepsilon,h^1_{\varepsilon,N}r,\lambda^1_N)$-invertibles, $u_0'u_1\sim I_{n}$, and $\diag(v_0'v_1,I_n)\sim\diag(I_n,v_0'v_1)\sim\diag(u_0'u_1,v_0'v_1)\sim I_{n+m}$. Let $w_0\in M_{n+m}(\tilde{A})$ be such that $||q(w_0)-\diag(u_0,v_0)||<\lambda_N\varepsilon$. There is a control pair $(\lambda^2,h^2)\geq(\lambda^1,h^1)$ such that, up to stabilization, there exist $a\in M_n(\tilde{A})$ and $b\in M_{n+m}(\tilde{A})$ such that 
\begin{itemize}[leftmargin=*]
\item $||q(a)-u_0'u_1||<\lambda^2_N\varepsilon$, 
\item $||q(b)-\diag(v_0'v_1,I_n)||<\lambda^2_N\varepsilon$, 
\item $||q(\diag(w_0,I_n)\diag(a,b))-\diag(u_1,v_1,I_m)||<\lambda^2_N\varepsilon$, and 
\item $||\diag(w_0a,b)\diag(I_n,0)\diag(a'w_0',b')-\diag(I_n,0)-\diag(y,0)||<\lambda^2_N\varepsilon$,
\end{itemize}
where $a',b',w_0'$ are quasi-inverses for $a,b,w_0$ respectively. By the previous cases, there is a control pair $(\lambda^\mathcal{D},h^\mathcal{D})$ such that \begin{align*} \partial_1([u_1])&=[\diag(y,0)+\diag(I_n,0)]-[\diag(I_n,0)] \\ &=[y+\diag(I_n,0)]-[\diag(I_n,0)] \\ &=\partial_1([u_0]) \end{align*} in $K_0^{\lambda^\mathcal{D}_N\varepsilon,h^\mathcal{D}_{\varepsilon,N}r,\lambda^\mathcal{D}_N}(J)$.

\end{enumerate}

To see that this controlled boundary map induces the usual boundary map in $K$-theory, let $u_0\in M_n(\tilde{A}/J)$ be invertible, and let $u\in M_n(\tilde{A}/J)$ be sufficiently close to $u_0$ so that $u$ is $(\varepsilon,r,N)$-invertible for some $\varepsilon,r$, and $N$. Similarly, let $v$ correspond to $u_0^{-1}$. Up to relaxing control, we may assume that $(u,v)$ is an $(\varepsilon,r,N)$-inverse pair. Let $w_0\in M_{2n}(\tilde{A})$ be a lift of $\diag(u_0,u_0^{-1})$. Recall that the usual boundary map $\partial_1:K_1(A/J)\rightarrow K_0(J)$ is defined by \[\partial_1([u_0])=[w_0\diag(I_n,0)w_0^{-1}]-[\diag(I_n,0)].\] Let $N=||w_0||+||w_0^{-1}||+1$, and let $w,w'\in M_{2n}(\tilde{A}_r)$ be such that $||w-w_0||<\frac{\varepsilon}{N}$ and $||w'-w_0^{-1}||<\frac{\varepsilon}{N}$ so that $(w,w')$ is an $(\varepsilon,r,N)$-inverse pair for some $r>0$. Moreover, $||q(w)-\diag(u,v)||<\frac{2\varepsilon}{N}$. We use this $w$ in the definition of the controlled boundary map to obtain $y$ such that \[||w\diag(I_n,0)w'-\diag(I_n,0)-y||<3C\lambda'_N\varepsilon\] and \[\partial_1^{\varepsilon,r,N}([u])=[y+\diag(I_n,0)]-[\diag(I_n,0)].\] Now \begin{align*} &\quad\; ||y+\diag(I_n,0)-w_0\diag(I_n,0)w_0^{-1}|| \\ &\leq||y+\diag(I_n,0)-w\diag(I_n,0)w'||+||w\diag(I_n,0)w'-w_0\diag(I_n,0)w_0^{-1}||\end{align*} so by making $\varepsilon$ sufficiently small, $y+\diag(I_n,0)$ and $w_0\diag(I_n,0)w_0^{-1}$ will be sufficiently close in norm so that \[ [\kappa_0(y+\diag(I_n,0))]=[w_0\diag(I_n,0)w_0^{-1}] \] in $K_0(J)$.
\end{proof}

\begin{rem}\label{boundaryrmk} The following can be deduced from the definition of the controlled boundary map.
\begin{enumerate}
\item If $0\rightarrow I\rightarrow A\rightarrow A/I\rightarrow 0$ and $0\rightarrow J\rightarrow B\rightarrow B/J\rightarrow 0$ are $C$-completely filtered extensions of $SQ_p$ algebras, and $\phi:A\rightarrow B$ is a filtered homomorphism such that $\phi(I)\subset J$ (so $\phi$ induces a filtered homomorphism $\tilde{\phi}:A/I\rightarrow B/J$), then $\mathcal{D}_1^B\circ\tilde{\phi}_*=\phi_*\circ\mathcal{D}_1^A$,
\item If $0\rightarrow J\rightarrow A\rightarrow A/J\rightarrow 0$ is a completely filtered split extension of $SQ_p$ algebras, i.e., there is a filtered homomorphism $s:A/J\rightarrow A$ such that $q\circ s=Id_{A/J}$ and the induced homomorphism $\tilde{A}/J\rightarrow\tilde{A}$ is $p$-completely contractive, then $\mathcal{D}_1=0$.
\end{enumerate}
\end{rem}

\begin{prop}
For any $C\geq 1$, there exists a control pair $(\lambda,h)$ such that if $0\rightarrow J\rightarrow A\stackrel{q}{\rightarrow}A/J\rightarrow 0$ is a $C$-completely filtered extension of $SQ_p$ algebras, then we have a $(\lambda,h)$-exact sequence \[\mathcal{K}_1(A)\rightarrow\mathcal{K}_1(A/J)\stackrel{\mathcal{D}_1}{\rightarrow}\mathcal{K}_0(J).\]
\end{prop}

\begin{proof}
Let $[u]\in K_1^{\varepsilon,r,N}(A)$, where $u\in GL_n^{\varepsilon,r,N}(\tilde{A})$, and let $v$ be an $(\varepsilon,r,N)$-inverse for $u$. Then $q(v)$ is an $(\varepsilon,r,N)$-inverse for $q(u)$, and $\diag(u,v)$ is a lift of $\diag(q(u),q(v))$. In the definition of the controlled boundary map, we may take $w=\diag(u,v)$ and $w'=\diag(v,u)$. Then $x=w\diag(I_n,0)w'=\diag(uv,0)$, and $||x-\diag(I_n,0)||<\varepsilon$ so we may take $y=0$. Thus $\partial_1([q(u)])=[\diag(I_n,0)]-[\diag(I_n,0)]=0$, i.e., the composition of the two morphisms is the zero map.

Now suppose that $\partial_1([u])=[y+\diag(I_n,0)]-[\diag(I_n,0)]=0$ with $u\in GL_n^{\varepsilon,r,N}(\tilde{A}/J)$. We may assume that $y+\diag(I_n,0)$ and $\diag(I_n,0)$ are homotopic as $(\varepsilon,r,N)$-idempotents in $M_{2n}(\tilde{J})$ and thus in $M_{2n}(\tilde{A})$. By Proposition \ref{homtosim2} and Lemma \ref{idemliphom}, up to stabilization, there exist a control pair $(\lambda'',h'')$ and a $(\lambda''_N\varepsilon,h''_{\varepsilon,N}r,\lambda''_N)$-inverse pair $(z,z')$ in $M_{2n}(\tilde{J})$ such that \[ ||y+\diag(I_n,0)-z\diag(I_n,0)z'||<\lambda''_N\varepsilon.\] Moreover, with $x=w\diag(I_n,0)w'$ as in the definition of the controlled boundary map, we have $||x-\diag(I_n,0)-y||<3C\lambda'_N\varepsilon$ so \[ ||x-z\diag(I_n,0)z'||<(3C\lambda'_N+\lambda''_N)\varepsilon.\] Let $V=\pi(z)z'w$, where $\pi:\tilde{J}\rightarrow\mathbb{C}$ is the usual quotient homomorphism. Then $q(V)=\pi(z)\pi(z')q(w)$ so \[ ||q(V)-\diag(u,v)||<2\lambda_N\varepsilon.\] Also, $||\pi(z)\diag(I_n,0)\pi(z')-\pi(w)\diag(I_n,0)\pi(w')||<(3C\lambda'_N+\lambda''_N)\varepsilon$ so \[||\pi(z)\diag(I_n,0)\pi(z')-\diag(\pi(uu'),0)||<(2\lambda_N+3C\lambda'_N+\lambda''_N)\varepsilon\] and \[ ||\pi(z)\diag(I_n,0)\pi(z')-\diag(I_n,0)||<\lambda'''_N\varepsilon,\] where $\lambda'''_N=2\lambda_N+3C\lambda'_N+\lambda''_N+1$. Thus, up to relaxing control, we have \[ ||V\diag(I_n,0)-\diag(I_n,0)V||<\lambda'''_N\varepsilon.\] Let $X$ be the upper left $n\times n$ corner of $V$. Then there exists a control pair $(\lambda^1,h^1)$ such that $X$ is a $(\lambda^1_N\varepsilon,h^1_{\varepsilon,N}r,\lambda^1_N)$-invertible element in $M_n(\tilde{A})$ and $[q(X)]=[u]$ in $K_1^{\lambda^1_N\varepsilon,h^1_{\varepsilon,N}r,\lambda^1_N}(A/J)$.
\end{proof}

\begin{prop}
For any $C\geq 1$, there exists a control pair $(\lambda,h)$ such that if $0\rightarrow J\rightarrow A\stackrel{q}{\rightarrow}A/J\rightarrow 0$ is a $C$-completely filtered extension of $SQ_p$ algebras, then we have a $(\lambda,h)$-exact sequence \[\mathcal{K}_1(A/J)\stackrel{\mathcal{D}_1}{\rightarrow}\mathcal{K}_0(J)\rightarrow\mathcal{K}_0(A).\]
\end{prop}

\begin{proof}
Let $[u]\in K_1^{\varepsilon,r,N}(A/J)$ and \[\partial_1([u])=[y+\diag(I_n,0)]-[\diag(I_n,0)]\in K_0^{\lambda^{\mathcal{D}}_N\varepsilon,h^{\mathcal{D}}_{\varepsilon,N}r,\lambda^{\mathcal{D}}_N}(J).\] From the definition of the controlled boundary map, we have \[ ||x-\diag(I_n,0)-y||<\lambda'_N\varepsilon,\] and $x$ is homotopic to $\diag(I_n,0)$ as $(\lambda'_N\varepsilon,h'_{\varepsilon,N}r,\lambda'_N)$-idempotents in $M_{2n}(\tilde{A})$. Hence $[y+\diag(I_n,0)]=[\diag(I_n,0)]$ in $K_0^{\lambda^{\mathcal{D}}_N\varepsilon,h^{\mathcal{D}}_{\varepsilon,N}r,\lambda^{\mathcal{D}}_N}(A)$, i.e., the composition of the two morphisms is the zero map.

Suppose that $[e]-[I_k]\in K_0^{\varepsilon,r,N}(J)$ with $e\in M_n(\tilde{J})$, and $[e]-[I_k]=0$ in $K_0^{\varepsilon,r,N}(A)$. Up to stabilization and relaxing control, we may assume that $e$ is $(\varepsilon,r,N)$-homotopic to $\diag(I_k,0)$ in $M_n(\tilde{A})$, and there exists a $(\lambda_N\varepsilon,h_{\varepsilon,N}r,\lambda_N)$-inverse pair 
$(w,w')$ in $M_n(\tilde{A})$ such that \[||e-w\diag(I_k,0)w')||<\lambda_N\varepsilon.\] Then 
\[||\diag(I_k,0)-\diag(q(w),q(w'))\diag(I_k,0)\diag(q(w'),q(w))||<2\lambda_N\varepsilon\] and it follows that \[||\diag(I_k,0)\diag(q(w),q(w'))-\diag(q(w),q(w'))\diag(I_k,0)||<3\lambda_N^2\varepsilon.\] Let $V_1$ be the upper left $k\times k$ corner of $q(w)$ and let $V_2$ be the lower right $(2n-k)\times(2n-k)$ corner of $\diag(q(w),q(w'))$. Then \[||\diag(V_1,V_2)-\diag(q(w),q(w'))||<3\lambda_N^2\varepsilon\] so there exists a control pair $(\lambda',h')$ such that $\diag(V_1,V_2)$ is homotopic to $I_{2n}$ as $(\lambda'_N\varepsilon,h'_{\varepsilon,N}r,\lambda'_N)$-invertibles. Since \[||\diag(w,w')\diag(I_k,0)\diag(w',w)-\diag(I_k,0)-(\diag(e,0)-\diag(I_k,0))||<\lambda_N\varepsilon,\] it follows from the definition of the controlled boundary map that there is an appropriate control pair such that \[\partial_1([V_1])=[\diag(e,0)]-[\diag(I_k,0)]=[e]-[I_k].\]
\end{proof}

Combining the results of this section, we get the following

\begin{thm}
For any $C\geq 1$, there exists a control pair $(\lambda,h)$ such that if $0\rightarrow J\rightarrow A\rightarrow A/J\rightarrow 0$ is a $C$-completely filtered extension of $SQ_p$ algebras, then we have a $(\lambda,h)$-exact sequence 
\[ \mathcal{K}_1(J)\rightarrow\mathcal{K}_1(A)\rightarrow\mathcal{K}_1(A/J)\stackrel{\mathcal{D}_1}{\rightarrow}\mathcal{K}_0(J)\rightarrow\mathcal{K}_0(A)\rightarrow\mathcal{K}_0(A/J). \]
\end{thm}

Applying the theorem to the semisplit extension \[0\rightarrow SA\rightarrow CA\rightarrow A\rightarrow 0,\] and recalling that $\mathcal{K}_*(CA)=0$, we see that there is a controlled isomorphism between $\mathcal{K}_1(A)$ and $\mathcal{K}_0(SA)$.

\begin{cor}
Let $\mathcal{D}_{1S}:\mathcal{K}_1(A)\rightarrow\mathcal{K}_0(SA)$ be the controlled boundary map associated to the semisplit extension \[0\rightarrow SA\rightarrow CA\rightarrow A\rightarrow 0.\] Then there exists a control pair $(\lambda,h)$ such that $\mathcal{D}_{1S}$ is $(\lambda,h)$-invertible for any filtered $SQ_p$ algebra $A$. Moreover, for any filtered $SQ_p$ algebra $A$, there exists a $(\lambda,h)$-controlled morphism $\mathcal{B}_A:\mathcal{K}_0(SA)\rightarrow\mathcal{K}_1(A)$ that is a $(\lambda,h)$-inverse for $\mathcal{D}_{1S}$ and such that for any filtered homomorphism $f:A\rightarrow B$ of filtered $SQ_p$ algebras, we have $\mathcal{B}_B\circ (Sf)_*=f_*\circ\mathcal{B}_A$.
\end{cor}

\begin{cor}
There exists a control pair $(\lambda,h)$ such that if \[0\rightarrow J\stackrel{j}{\rightarrow}A\rightarrow A/J\rightarrow 0\] is a completely filtered split extension of $SQ_p$ algebras (so there exists a filtered homomorphism $s:A/J\rightarrow A$ such that $q\circ s=Id_{A/J}$ and the induced homomorphism $\tilde{A}/J\rightarrow\tilde{A}$ is $p$-completely contractive), then we have $(\lambda,h)$-exact sequences \[0\rightarrow\mathcal{K}_*(J)\rightarrow\mathcal{K}_*(A)\rightarrow\mathcal{K}_*(A/J)\rightarrow 0,\] and we have $(\lambda,h)$-isomorphisms \[\mathcal{K}_*(J)\oplus\mathcal{K}_*(A/J)\rightarrow\mathcal{K}_*(A)\] given by $(x,y)\mapsto j_*(x)+s_*(y)$.
\end{cor}

\begin{rem}
At this point, it seems appropriate to make a remark about Bott periodicity. Indeed, one can recover Bott periodicity in ordinary $K$-theory from our controlled long exact sequence by considering suspensions. In \cite{OY15}, a proof of controlled Bott periodicity in the $C^*$-algebraic setting was given using the power of $KK$-theory. However, we do not have a proof of controlled Bott periodicity in our setting. 
\end{rem}

\section{A Controlled Mayer-Vietoris Sequence} \label{sec:MV} 

In this section, we follow the approach in \cite{OY} to obtain a controlled Mayer-Vietoris sequence for filtered $SQ_p$ algebras. Throughout this section, whenever $A$ is a Banach subalgebra of a unital Banach algebra $B$, we will view $A^+$ as $A+\mathbb{C}1_B\subset B$.

\subsection{Preliminary definitions}

\begin{defn}
Let $A$ be a filtered $SQ_p$ algebra with filtration $(A_r)_{r> 0}$. Let $s> 0$ and let $(\Delta_1,\Delta_2)$ be a pair of closed linear subspaces of $A_s$. Then $(\Delta_1,\Delta_2)$ is called a completely coercive decomposition pair of degree $s$ for $A$ if there exists $c>0$ such that for every $r\leq s$, any positive integer $n$, and any $x\in M_n(A_r)$, there exist $x_1\in M_n(\Delta_1\cap A_r)$ and $x_2\in M_n(\Delta_2\cap A_r)$ such that $\max(||x_1||,||x_2||)\leq c||x||$ and $x=x_1+x_2$.

The number $c$ is called the coercity of the pair $(\Delta_1,\Delta_2)$.
\end{defn}


\begin{defn}
Let $A$ be a filtered $SQ_p$ algebra with filtration $(A_r)_{r> 0}$. Let $r$ and $R$ be positive numbers, and let $\Delta$ be a closed linear subspace of $A_r$. We define the $R$-neighborhood of $\Delta$, denoted by $\mathfrak{N}_\Delta^{(r,R)}$, to be $\Delta+A_R\Delta+\Delta A_R+A_R\Delta A_R$.

We will denote $\mathfrak{N}_\Delta^{(r,R)}\cap A_s$ by $\mathfrak{N}_{\Delta,s}^{(r,R)}$. For $s\leq r$, we also denote the $R$-neighborhood of $\Delta\cap A_s$ by $\mathfrak{N}_\Delta^{(s,R)}$.
\end{defn}


\begin{defn}
Let $A$ be a filtered $SQ_p$ algebra with filtration $(A_r)_{r> 0}$. Let $s> 0$ and let $\Delta$ be a closed linear subspace of $A_s$. Then a Banach subalgebra $B$ of $A$ is called an $s$-controlled $\Delta$-neighborhood in $A$ if $B$ has filtration $(B\cap A_r)_{r> 0}$, and $B$ contains $\mathfrak{N}_{\Delta}^{(s,5s)}$.
\end{defn}

The choice of the coefficient 5 in the preceding definition is determined by the proof of Lemma \ref{MVtechlem}.

\begin{defn}
Let $S_1$ and $S_2$ be subsets of an $SQ_p$ algebra $A$. The pair $(S_1,S_2)$ is said to have the complete intersection approximation property if there exists $c>0$ such that for any $\varepsilon>0$, any positive integer $n$, any $x\in M_n(S_1)$ and $y\in M_n(S_2)$ with $||x-y||<\varepsilon$, there exists $z\in M_n(S_1\cap S_2)$ such that $\max(||z-x||,||z-y||)<c\varepsilon$.

The number $c$ is called the coercity of the pair $(S_1,S_2)$.
\end{defn}

\begin{defn}
Let $A$ be a filtered $SQ_p$ algebra with filtration $(A_r)_{r> 0}$, let $s> 0$, and let $c>0$. An $(s,c)$-controlled Mayer-Vietoris pair for $A$ is a pair $(A_{\Delta_1},A_{\Delta_2})$ of Banach subalgebras of $A$ associated with a pair $(\Delta_1,\Delta_2)$ of closed linear subspaces of $A_s$ such that
\begin{itemize}
\item $(\Delta_1,\Delta_2)$ is a completely coercive decomposition pair for $A$ of degree $s$ with coercity $c$;
\item $A_{\Delta_i}$ is an $s$-controlled $\Delta_i$-neighborhood in $A$ for $i=1,2$;
\item $(A_{\Delta_1,r},A_{\Delta_2,r})$ has the complete intersection approximation property for every $r\leq s$ with coercity $c$.
\end{itemize}
\end{defn}

\begin{rem}\label{ciarem}
If $A$ is a non-unital filtered $SQ_p$ algebra, and $(A_{\Delta_1,r},A_{\Delta_2,r})$ has the complete intersection approximation property with coercity $c$ with respect to $A$, then $(\widetilde{A_{\Delta_1,r}},\widetilde{A_{\Delta_2,r}})$ has the complete intersection approximation property with coercity $2c+\frac{1}{2}$ with respect to $\tilde{A}$. Indeed, if $x\in M_n(\widetilde{A_{\Delta_1,r}})$ and $y\in M_n(\widetilde{A_{\Delta_2,r}})$ are such that $||x-y||<\varepsilon$, then $||(x-\pi(x))-(y-\pi(y))||<2\varepsilon$ so there exists $z_0\in M_n(A_{\Delta_1,r}\cap A_{\Delta_2,r})$ such that \[\max(||x-\pi(x)-z_0||,||y-\pi(y)-z_0||)<2c\varepsilon.\] Letting $z=z_0+\frac{1}{2}(\pi(x)+\pi(y))\in M_n(\widetilde{A_{\Delta_1,r}}\cap\widetilde{A_{\Delta_2,r}})$, we have \[\max(||x-z||,||y-z||)<(2c+\frac{1}{2})\varepsilon.\]
\end{rem}

Let $A$ be a filtered $SQ_p$ algebra with filtration $(A_r)_{r> 0}$, let $s> 0$, and let $c>0$. Then $C([0,1],A)$ 
has filtration $(C([0,1],A_r))_{r> 0}$. Suppose that $(A_{\Delta_1},A_{\Delta_2})$ is an $(s,c)$-controlled Mayer-Vietoris pair for $A$. Note that $C([0,1],\Delta_i)$ is a closed linear subspace of $C([0,1],A_s)$ for $i=1,2$. Also, $C([0,1],A_{\Delta_i})$ is a Banach subalgebra of $C([0,1],A)$ for $i=1,2$. For $n\in\mathbb{N}$, we view $M_n(C([0,1],A))$ as a subalgebra of $C([0,1],M_n(A))$ with the supremum norm. 

One can show that $(A_{\Delta_1}[0,1],A_{\Delta_2}[0,1])$ is an $(s,2c)$-controlled Mayer-Vietoris pair for $A[0,1]$. Moreover, if the decomposition $x=x_1+x_2$ in the definition of completely coercive decomposition pair is continuous in the sense that $||x(k)_i-x(l)_i||\leq c||x(k)-x(l)||$ for all $k,l\in[0,1]$, $i=1,2$, and $x\in M_n(A[0,1]_r)$, then $(A_{\Delta_1}[0,1],A_{\Delta_2}[0,1])$ is an $(s,c)$-controlled Mayer-Vietoris pair for $A[0,1]$. Similar statements hold for the respective suspensions.

Our main goal is to show that a Mayer-Vietoris pair gives rise to a controlled Mayer-Vietoris sequence. The following example is our motivating example. It illustrates the potential applicability of the notions introduced here, and thus the existence of a controlled Mayer-Vietoris sequence, in a situation where we have a group action with finite dynamic asymptotic dimension as defined in \cite{GWY1}. Indeed, this idea was implemented to investigate the $C^*$-algebraic Baum-Connes conjecture in \cite{GWY2}, and we will discuss the $L_p$ operator algebra version in \cite{Chung2}.

\begin{ex}
Let $X$ be a compact Hausdorff space, and let $G$ be a discrete group acting on $X$ by homeomorphisms. Then we get an isometric action of $G$ on $C(X)$, where we regard $C(X)$ as an $L_p$ operator algebra with functions in $C(X)$ acting as multiplication operators on $L_p(X,\mu)$ for some regular Borel measure $\mu$. Consider the reduced $L_p$ crossed product $A=C(X)\rtimes_{\lambda,p}G$, which is defined in an analogous manner to the reduced $C^*$ crossed product by completing the skew group ring in $B(L_p(G\times X))$ with the operator norm (cf. \cite{Phil13}). Equip $G$ with a proper length function $l$. Then $A$ is a filtered $L_p$ operator algebra (and thus a filtered $SQ_p$ algebra) with filtration given by $A_r=\{\sum f_g g:f_g\in C(X),l(g)\leq r\}$. 

Suppose that for any finite subset $E$ of $G$, there exist open subsets $U_0$ and $U_1$ such that $X=U_0\cup U_1$, and the set \[ \left\{g\in G:\;\parbox[c][4em][c]{0.6\textwidth}{there exist $x\in U_i$ and $g_n,\ldots,g_1\in E$ such that $g=g_n\cdots g_1$ and $g_k\cdots g_1x\in U_i$ for all $k\in\{1,\ldots,n\}$}\right\} \]  is finite for $i=0,1$. Let $E=\{g\in G:l(g)\leq 4R\}$ for some fixed $R>0$, let $V_0$ and $V_1$ be the open subsets of $X$ corresponding to the finite set $E^3$, and let $U_i=\bigcup_{g\in E}gV_i$ for $i=0,1$. Let $A_i$ be the Banach subalgebra of $A$ generated by $\{f_g g:supp(f_g)\subset U_i,g\in G\}$. By considering \[\Delta_i=\{\sum f_g g:supp(f_g)\subset V_i,l(g)\leq R\},\] one can show that $(A_0,A_1)$ is an $(R,1)$-controlled Mayer-Vietoris pair for $A$.
\end{ex}

\subsection{Key technical lemma}

In order to build a controlled Mayer-Vietoris sequence from a Mayer-Vietoris pair, we need to be able to factor a quasi-invertible as a product of two invertibles, one from each subalgebra in the Mayer-Vietoris pair. We will prove a lemma that allows us to do so.

Let $A$ be a unital filtered $SQ_p$ algebra. For $x,y\in A$, set
\[ X(x)=\begin{pmatrix} 1 & x \\ 0 & 1 \end{pmatrix}, Y(y)=\begin{pmatrix} 1 & 0 \\ y & 1 \end{pmatrix}. \]
Note that $X(x_1+x_2)=X(x_1)X(x_2)$, $Y(y_1+y_2)=Y(y_1)Y(y_2)$, $X(x)^{-1}=X(-x)$, and $Y(y)^{-1}=Y(-y)$. 
Also, if $(u,v)$ is an $(\varepsilon,r,N)$-inverse pair in $A$, then 
\begin{align*} \left\Vert\begin{pmatrix} u & 0 \\ 0 & v \end{pmatrix}-X(u)Y(-v)X(u)\begin{pmatrix} 0 & -1 \\ 1 & 0 \end{pmatrix}\right\Vert &=\left\Vert\begin{pmatrix} (uv-1)u & 1-uv \\ vu-1 & 0 \end{pmatrix}\right\Vert \\ &<(N+1)\varepsilon.\end{align*}

If $A$ is non-unital, and $(u,v)$ is an $(\varepsilon,r,N)$-inverse pair in $A^+$ with $\pi(u)=\pi(v)=1$, $u-1=u_1+u_2$, and $v-1=v_1+v_2$, then letting \[ P_1=X(u_1+1)X(u_2)Y(-v_1-1)X(u_1+1)\begin{pmatrix} 0 & -1 \\ 1 & 0 \end{pmatrix}X(-u_2) \] and \[ P_2=X(u_2)\begin{pmatrix} 0 & 1 \\ -1 & 0 \end{pmatrix}X(-u_1-1)Y(-v_2)X(u_1+1)X(u_2)\begin{pmatrix} 0 & -1 \\ 1 & 0 \end{pmatrix}, \]
we have $X(u)Y(-v)X(u)\begin{pmatrix} 0 & -1 \\ 1 & 0 \end{pmatrix}=P_1P_2$. Moreover, if $u_i,v_i\in\Delta_i\cap A_r$, then $P_1-I\in M_2(\mathfrak{N}_{\Delta_1,4r}^{(r,2r)})$ and $P_2-I\in M_2(\mathfrak{N}_{\Delta_2,4r}^{(r,2r)})$.

\begin{lem} \label{MVtechlem}
There exists a polynomial $Q$ with positive integer coefficients such that for any filtered $SQ_p$ algebra $A$ with filtration $(A_r)_{r> 0}$, any $(s,c)$-controlled Mayer-Vietoris pair $(A_{\Delta_1},A_{\Delta_2})$ for $A$, and any $(\varepsilon,r,N)$-inverse pair $(u,v)$ in $\tilde{A}$ such that $u$ is homotopic to 1 with $0<r\leq s$, $u-1\in A$, and $v-1\in A$, there exist an integer $k\geq 2$, $M=M(c,N)>1$, and $z_i\in M_k(\widetilde{A_{\Delta_i}}\cap\tilde{A}_{16r})$ such that
\begin{itemize}
\item $z_i$ is invertible in $M_k(\widetilde{A_{\Delta_i}})$ for $i=1,2$;
\item $\max(||z_i||,||z_i^{-1}||)\leq M$ for $i=1,2$;

\item $||\diag(u,I_{k-1})-z_1z_2||<Q(N)\varepsilon$; 
\item for $i=1,2$, there exists a homotopy $(z_{i,t})_{t\in[0,1]}$ of invertible elements in $M_k(\widetilde{A_{\Delta_i}}\cap\tilde{A}_{16r})$ between $I_k$ and $z_i$ such that \[\max(||z_{i,t}||,||z_{i,t}^{-1}||)\leq M\] for each $t\in[0,1]$. Moreover, $\pi_i(z_{i,t})$ and $\pi_i(z_{i,t}^{-1})$ are homotopic to $I_k$ in $M_k(\mathbb{C})$ via homotopies of invertible elements of norm less than 3, where $\pi_i:\widetilde{A_{\Delta_i}}\rightarrow\mathbb{C}$ is the quotient homomorphism. 

\end{itemize}
\end{lem}


\begin{proof}
If $||u-1||<\varepsilon$, then we may simply take $z_1=z_2=I_2$. In the general case, we proceed as follows.

Let $(u_t)_{t\in[0,1]}$ be a homotopy of $(\varepsilon,r,N)$-invertibles in $\tilde{A}$ with $u_0=u$ and $u_1=1$. Up to relaxing control, we may assume that $u_t-1\in A$ for each $t$. For each $t\in[0,1]$, let $u_t'$ be an $(\varepsilon,r,N)$-inverse for $u_t$. Up to relaxing control, we may also assume that $u_t'-1\in A$ for each $t$. Let $0=t_0<\cdots<t_m=1$ be such that $||u_{t_i}-u_{t_{i-1}}||<\frac{\varepsilon}{N}$ for $i=1,\ldots,m$. Set \begin{align*} V&=\diag(u_{t_0},\ldots,u_{t_m},u_{t_0}',\ldots,u_{t_m}'), \\ W&=\diag(1,u_{t_0}',\ldots,u_{t_{m-1}}',u_{t_0},\ldots,u_{t_{m-1}},1).\end{align*} Then \[||\diag(u,I_{2m+1})-VW||<2\varepsilon.\]

In the rest of this proof, we shall write $X^\wedge$ for $\diag(1,X,1)\in M_{2k+2}(\tilde{A})$ whenever $X\in M_{2k}(\tilde{A})$.

For each $i\in\{0,\ldots,m\}$, there exist $v_i,v_i'\in\Delta_1\cap A_r$ and $w_i,w_i'\in\Delta_2\cap A_r$ such that \begin{align*} u_{t_i}-1=v_i+w_i,&\quad u_{t_i}'-1=v_i'+w_i', \\ \max(||v_i||,||w_i||)\leq c||u_{t_i}-1||, &\quad \max(||v_i'||,||w_i'||)\leq c||u_{t_i}'-1||.\end{align*} Set 
\begin{align*}
x_1=\diag(v_0+1,\ldots,v_m+1) &, x_2=\diag(w_0,\ldots,w_m), \\
x_1'=\diag(v_0'+1,\ldots,v_m'+1) &, x_2'=\diag(w_0',\ldots,w_m'), \\
y_1=\diag(v_0+1,\ldots,v_{m-1}+1) &, y_2=\diag(w_0,\ldots,w_{m-1}), \\
y_1'=\diag(v_0'+1,\ldots,v_{m-1}'+1) &, y_2'=\diag(w_0',\ldots,w_{m-1}').
\end{align*}
Then
\[ \begin{pmatrix} x_1+x_2 & 0 \\ 0 & x_1'+x_2' \end{pmatrix}=V \] and
\[ \begin{pmatrix} y_1'+y_2' & 0 \\ 0 & y_1+y_2 \end{pmatrix}^\wedge=W.  \]

Set \[P_1(x,y)=X(x)X(y)Y(-x')X(x)\begin{pmatrix} 0 & -1 \\ 1 & 0 \end{pmatrix}X(-y)\]
and \[P_2(x,y)=X(y)\begin{pmatrix} 0 & 1 \\ -1 & 0 \end{pmatrix}X(-x)Y(-y')X(x)X(y)\begin{pmatrix} 0 & -1 \\ 1 & 0 \end{pmatrix}\]
with the convention that $(x')'=x$. 
We have
\[ \Vert V-P_1(x_1,x_2)P_2(x_1,x_2)\Vert<(N+1)\varepsilon \]
and
\[ \left\Vert \begin{pmatrix} y_1'+y_2' & 0 \\ 0 & y_1+y_2 \end{pmatrix}-P_1(y_1',y_2')P_2(y_1',y_2') \right\Vert<(N+1)\varepsilon \]
so
\[ \Vert W-P_1(y_1',y_2')^\wedge P_2(y_1',y_2')^\wedge\Vert<(N+1)\varepsilon \]
and
\[ \Vert VW-P_1(x_1,x_2)P_2(x_1,x_2)P_1(y_1',y_2')^\wedge P_2(y_1',y_2')^\wedge \Vert <3N(N+1)\varepsilon. \]
Hence
\[ \biggl\Vert \begin{pmatrix} u & 0 \\ 0 & I_{2m+1} \end{pmatrix}-P_1(x_1,x_2)P_2(x_1,x_2)P_1(y_1',y_2')^\wedge P_2(y_1',y_2')^\wedge \biggr\Vert <(3N(N+1)+2)\varepsilon. \]

We will show that 
\[P_1(x_1,x_2)P_2(x_1,x_2)P_1(y_1',y_2')^\wedge P_2(y_1',y_2')^\wedge\]
can be factored as a product $z_1z_2$ with $z_1$ and $z_2$ having the required properties.

Let
\[ z_1=P_1(x_1,x_2)P_1(y_1',y_2')^\wedge  \]
and
\[ z_2=\left[(P_1(y_1',y_2')^\wedge)^{-1}P_2(x_1,x_2) P_1(y_1',y_2')^\wedge\right] P_2(y_1',y_2')^\wedge. \] 
Then $z_1$ and $z_2$ are matrices in $M_{2m+2}(\widetilde{A_{\Delta_i}}\cap\tilde{A}_{16r})$ that are invertible in $M_{2m+2}(\widetilde{A_{\Delta_i}})$, and \[ P_1(x_1,x_2)P_2(x_1,x_2)P_1(y_1',y_2')^\wedge P_2(y_1',y_2')^\wedge=z_1z_2. \]

Let $\pi_i:\widetilde{A_{\Delta_i}}\rightarrow\mathbb{C}$ denote the quotient homomorphism. Then \[\pi_1(P_1(x_1,x_2))=\pi_1(P_1(y_1',y_2')^\wedge)=I_{2m+2},\]
and \[\pi_2(P_2(x_1,x_2))=\pi_2(P_2(y_1',y_2')^\wedge)=I_{2m+2}.\]
Thus $\pi_1(z_1)=\pi_1(z_1^{-1})=\pi_2(z_2)=\pi_2(z_2^{-1})=I_{2m+2}$.

Note that $||X(x)||\leq 1+||x||$ and $||Y(y)||\leq 1+||y||$. Since \[\max(||x_i||,||x_i'||,||y_i||,||y_i'||)\leq \begin{cases} c(N+1)+1, & i=1 \\ c(N+1), & i=2 \end{cases},\] we have for $i=1,2$, \[\max(||z_i||,||z_i^{-1}||)\leq (1+N)^4(1+c(N+1))^6(2+c(N+1))^6.\]
For $t\in[0,1]$, set \begin{align*} z_{1,t}&=P_1(tx_1,tx_2)P_1(ty_1',ty_2')^\wedge, \\  
z_{2,t}&=\left[ (P_1(y_1',y_2')^\wedge)^{-1}P_2(tx_1,tx_2)P_1(y_1',y_2')^\wedge\right] P_2(ty_1',ty_2')^\wedge.\end{align*}
Then $(z_{i,t})_{t\in[0,1]}$ is a homotopy of invertibles in $M_{2m+2}(\widetilde{A_{\Delta_i}}\cap\tilde{A}_{16r})$ between $I_{2m+2}$ and $z_i$ for $i=1,2$. Moreover, for each $t\in[0,1]$, $\pi_2(z_{2,t})=\pi_2(z_{2,t}^{-1})=I_{2m+2}$, while $\pi_1(z_{1,t})$ and $\pi_1(z_{1,t}^{-1})$ are homotopic to $I_{2m+2}$ in $M_{2m+2}(\mathbb{C})$ via homotopies of invertible elements of norm less than 3. Finally, \[\max(||z_{i,t}||,||z_{i,t}^{-1}||)\leq (1+N)^4(1+c(N+1))^6(2+c(N+1))^6\] for $i=1,2$.
\end{proof}

\subsection{Controlled Mayer-Vietoris sequence}

Given a filtered $SQ_p$ algebra $A$ and an $(s,c)$-controlled Mayer-Vietoris pair $(A_{\Delta_1},A_{\Delta_2})$ for $A$, we will consider the inclusion maps $j_1:A_{\Delta_1}\rightarrow A$, $j_2:A_{\Delta_2}\rightarrow A$, $j_{1,2}:A_{\Delta_1}\cap A_{\Delta_2}\rightarrow A_{\Delta_1}$, and $j_{2,1}:A_{\Delta_1}\cap A_{\Delta_2}\rightarrow A_{\Delta_2}$.

For any $0<\varepsilon<\frac{1}{20},r> 0$, and $N\geq 1$, it is clear that the composition 
\[ \begin{CD}
K_0^{\varepsilon,r,N}(A_{\Delta_1}\cap A_{\Delta_2}) @>(j_{1,2*},j_{2,1*})>> K_0^{\varepsilon,r,N}(A_{\Delta_1})\oplus K_0^{\varepsilon,r,N}(A_{\Delta_2}) @>j_{1*}-j_{2*}>> K_0^{\varepsilon,r,N}(A)
\end{CD} \]
is the zero map.

\begin{prop}
For every $c>0$, there exists a control pair $(\lambda,h)$ such that for any filtered $SQ_p$ algebra $A$, any $(s,c)$-controlled Mayer-Vietoris pair $(A_{\Delta_1},A_{\Delta_2})$ for $A$, any $N\geq 1$, $0<\varepsilon<\frac{1}{20\lambda_N}$, and $0< r\leq\frac{s}{h_{\varepsilon,N}}$, if $y_1\in K_0^{\varepsilon,r,N}(A_{\Delta_1})$ and $y_2\in K_0^{\varepsilon,r,N}(A_{\Delta_2})$ are such that $j_{1*}(y_1)=j_{2*}(y_2)$ in $K_0^{\varepsilon,r,N}(A)$, then there exists $z\in K_0^{\lambda_N\varepsilon,h_{\varepsilon,N}r,\lambda_N}(A_{\Delta_1}\cap A_{\Delta_2})$ such that $j_{1,2*}(z)=y_1$ in $K_0^{\lambda_N\varepsilon,h_{\varepsilon,N}r,\lambda_N}(A_{\Delta_1})$ and $j_{2,1*}(z)=y_2$ in $K_0^{\lambda_N\varepsilon,h_{\varepsilon,N}r,\lambda_N}(A_{\Delta_2})$.
\end{prop}

\begin{proof}
Up to rescaling $\varepsilon$, $r$, and $N$, and up to stabilization, we may write $y_i=[e_i]-[I_k]$, where $e_i$ is an $(\varepsilon,r,N)$-idempotent in $M_n(\widetilde{A_{\Delta_i}})$ for $i=1,2$ with $e_i-\diag(I_k,0)\in M_n(A_{\Delta_i})$ and with $e_1$ and $e_2$ homotopic as $(\varepsilon,r,N)$-idempotents in $M_n(\tilde{A})$. 
By Lemma \ref{idemliphom} and Proposition \ref{homtosim2}, and up to stabilization, there exists a control pair $(\lambda,h)$, and a $(\lambda_N\varepsilon,h_{\varepsilon,N}r,\lambda_N)$-inverse pair $(u,v)$ in $M_n(\tilde{A})$ such that $||ue_2v-e_1||<\lambda_N\varepsilon$. Then \[||\diag(u,v)\diag(e_2,0)\diag(v,u)-\diag(e_1,0)||<\lambda_N\varepsilon.\] We may also assume that $u-I_n\in M_n(A)$ and $v-I_n\in M_n(A)$.

By Lemma \ref{MVtechlem}, there exist $Q=Q(\lambda_N)$ and $M=M(c,\lambda_N)$, and up to stabilization, there exist invertible elements $w_i\in M_{2n}(\widetilde{A_{\Delta_i}}\cap\tilde{A}_{16h_{\varepsilon,N} r})$ such that $\max(||w_i||,||w_i^{-1}||)\leq M$ for $i=1,2$, $\pi_i(w_i)$ and $\pi_i(w_i^{-1})$ are homotopic to $I_{2n}$ via homotopies of invertible elements with norm less than 3, where $\pi_i:\widetilde{A_{\Delta_i}}\rightarrow\mathbb{C}$ is the quotient homomorphism, and $||\diag(u,v)-w_1w_2||<Q\varepsilon$. Then we also have 
\begin{align*}
&\quad\; ||\diag(v,u)-w_2^{-1}w_1^{-1}|| \\ &= ||(\diag(v,u)(w_1w_2-\diag(u,v))+\diag(vu-1,uv-1))w_2^{-1}w_1^{-1}|| \\
&\leq \lambda_N(Q+1)M^2\varepsilon.
\end{align*}
Thus 
\begin{align*}
&\quad\; ||w_1^{-1}\diag(e_1,0)w_1-w_2\diag(e_2,0)w_2^{-1}|| \\ &= ||w_1^{-1}(\diag(e_1,0)-w_1w_2\diag(e_2,0)w_2^{-1}w_1^{-1})w_1|| \\
&\leq M^2 ||(\diag(u,v)-w_1w_2)\diag(e_2,0)w_2^{-1}w_1^{-1} \\ &\quad\quad +\diag(u,v)\diag(e_2,0)(\diag(v,u)-w_2^{-1}w_1^{-1}) 
+\diag(e_1-ue_2v,0)|| \\
&\leq M^2(QNM^2 + \lambda_N^2 N(Q+1)M^2 + \lambda_N)\varepsilon.
\end{align*}

Note that $M^2(QNM^2 + \lambda_N^2 N(Q+1)M^2 + \lambda_N)$ depends only on $c$ and $N$. We will denote it by $M'=M'(c,N)$.



Thus there exists $e\in M_{2n}(\widetilde{A}_{\Delta_1,(32h_{\varepsilon,N}+1)r}\cap \widetilde{A}_{\Delta_2,(32h_{\varepsilon,N}+1)r})$ such that 
\[ ||e-w_1^{-1}\diag(e_1,0)w_1||<\biggl(2c+\frac{1}{2}\biggr)M'\varepsilon \]
and \[ ||e-w_2\diag(e_2,0)w_2^{-1}||<\biggl(2c+\frac{1}{2}\biggr)M'\varepsilon, \]
with the estimates being obtained based on Remark \ref{ciarem}.

By Lemma \ref{simtohom2} and Lemma \ref{normestlem1}, there exists a control pair $(\lambda'',h'')$ such that 
\begin{itemize}[leftmargin=*]
\item $w_1^{-1}\diag(e_1,0)w_1$ and $w_2\diag(e_2,0)w_2^{-1}$ are $(\lambda''_N\varepsilon,h''_{\varepsilon,N}r,\lambda''_N)$-idempotents in $M_{2n}(\tilde{A})$,
\item $\diag(w_1^{-1},0)\diag(e_1,0)\diag(w_1,0)$ and $\diag(w_2,0)\diag(e_2,0)\diag(w_2^{-1},0)$ are $(\lambda''_N\varepsilon,h''_{\varepsilon,N}r,\lambda''_N)$-homotopic to $\diag(e_1,0)$ and $\diag(e_2,0)$ respectively in $M_{4n}(\tilde{A})$,
\item $e$ is $(\lambda''_N\varepsilon,h''_{\varepsilon,N}r,\lambda''_N)$-homotopic to $w_1^{-1}\diag(e_1,0)w_1$ and $w_2\diag(e_2,0)w_2^{-1}$ in $M_{2n}(\tilde{A})$,
\item $\pi(e)$ is homotopic to $I$ as $(\lambda''_N\varepsilon,h''_{\varepsilon,N}r,\lambda''_N)$-idempotents in $M_{2n}(\mathbb{C})$, where $\pi:\widetilde{A_{\Delta_1}}\cap\widetilde{A_{\Delta_2}}\rightarrow\mathbb{C}$ is the quotient homomorphism. 
\end{itemize}
Hence $j_{1,2*}([e]-[I_k])=y_1$ in $K_0^{\lambda''_N\varepsilon,h''_{\varepsilon,N}r,\lambda''_N}(A_{\Delta_1})$ and $j_{2,1*}([e]-[I_k])=y_2$ in $K_0^{\lambda''_N\varepsilon,h''_{\varepsilon,N}r,\lambda''_N}(A_{\Delta_2})$.
\end{proof}

In the odd case, we also have that for any $0<\varepsilon<\frac{1}{20}$, $r> 0$, and $N\geq 1$, the composition
\[ \begin{CD}
K_1^{\varepsilon,r,N}(A_{\Delta_1}\cap A_{\Delta_2}) @>(j_{1,2*},j_{2,1*})>> K_1^{\varepsilon,r,N}(A_{\Delta_1})\oplus K_1^{\varepsilon,r,N}(A_{\Delta_2}) @>j_{1*}-j_{2*}>> K_1^{\varepsilon,r,N}(A)
\end{CD} \]
is the zero map.

\begin{prop}
For every $c>0$, there exists a control pair $(\lambda,h)$ such that for any filtered $SQ_p$ algebra $A$, any $(s,c)$-controlled Mayer-Vietoris pair $(A_{\Delta_1},A_{\Delta_2})$ for $A$, any $N\geq 1$, $0<\varepsilon<\frac{1}{20\lambda_N}$, and $0< r\leq\frac{s}{h_{\varepsilon,N}}$, if $y_1\in K_1^{\varepsilon,r,N}(A_{\Delta_1})$ and $y_2\in K_1^{\varepsilon,r,N}(A_{\Delta_2})$ are such that $j_{1*}(y_1)=j_{2*}(y_2)$ in $K_1^{\varepsilon,r,N}(A)$, then there exists $z\in K_1^{\lambda_N\varepsilon,h_{\varepsilon,N}r,\lambda_N}(A_{\Delta_1}\cap A_{\Delta_2})$ such that $j_{1,2*}(z)=y_1$ in $K_1^{\lambda_N\varepsilon,h_{\varepsilon,N}r,\lambda_N}(A_{\Delta_1})$ and $j_{2,1*}(z)=y_2$ in $K_1^{\lambda_N\varepsilon,h_{\varepsilon,N}r,\lambda_N}(A_{\Delta_2})$.
\end{prop}

\begin{proof}
By relaxing control, we may write $y_i=[u_i]$, where $u_i$ is an $(\varepsilon,r,N)$-invertible in $M_n(\widetilde{A_{\Delta_i}})$ for $i=1,2$ with $\pi_i(u_i)=I_n$. Then $[u_1]=[u_2]$ in $K_1^{\varepsilon,r,N}(A)$ so up to stabilization and relaxing control, we may assume that $u_1$ and $u_2$ are homotopic as $(\varepsilon,r,N)$-invertibles in $M_n(\tilde{A})$. Let $v_2\in M_n(\widetilde{A_{\Delta_2}})$ be an $(\varepsilon,r,N)$-inverse for $u_2$ with $\pi_2(v_2)=I_n$. If $(u_{t+1})_{t\in[0,1]}$ is a homotopy of $(\varepsilon,r,N)$-invertibles in $M_n(\tilde{A})$ between $u_1$ and $u_2$ with $\pi(u_{t+1})=I_n$ for all $t$, then $(u_{t+1}v_2)_{t\in[0,1]}$ is a homotopy of $((N^2+1)\varepsilon,2r,N^2)$-invertibles in $M_n(\tilde{A})$ between $u_1v_2$ and $u_2v_2$. Moreover, since $||u_2v_2-I_n||<\varepsilon$, we have that $u_2v_2$ and $I_n$ are homotopic as $(\varepsilon,2r,1+\varepsilon)$-invertibles in $M_n(\tilde{A})$. Hence $u_1v_2$ and $I_n$ are homotopic as $((N^2+1)\varepsilon,2r,N^2)$-invertibles in $M_n(\tilde{A})$.

By Lemma \ref{MVtechlem}, there exist $Q(N)$, $M=M(c,N)$, and up to stabilization, there exist invertibles $w_i$ in $M_n(\widetilde{A_{\Delta_i}}\cap\tilde{A}_{32 r})$ such that $\max(||w_i||,||w_i^{-1}||)\leq M$, $\pi(w_i)$ and $\pi(w_i^{-1})$ are homotopic to $I_n$ via homotopies of invertible elements in $M_n(\mathbb{C})$ with norm less than 3, and $||u_1v_2-w_1w_2||<Q(N)\varepsilon$. Now \begin{align*} ||w_1^{-1}u_1-w_2u_2|| &\leq||w_1^{-1}(u_1v_2-w_1w_2)u_2||+||w_1^{-1}u_1(v_2u_2-1)|| \\ &<MN(Q(N)+1)\varepsilon.\end{align*}
Thus there exists $z\in M_n(\widetilde{A_{\Delta_1,}}_{33r}\cap \widetilde{A_{\Delta_2,}}_{33r})$ such that \[||z-w_1^{-1}u_1||<\biggl(2c+\frac{1}{2}\biggr)MN(Q(N)+1)\varepsilon\] and \[||z-w_2u_2||<\biggl(2c+\frac{1}{2}\biggr)MN(Q(N)+1)\varepsilon.\]

It follows from Lemma \ref{normestlem1} that there exists a control pair $(\lambda,h)$ such that $z$ is homotopic to $u_i$ as $(\lambda_N\varepsilon,h_{\varepsilon,N}r,\lambda_N)$-invertibles in $M_n(\widetilde{A_{\Delta_i}})$ for $i=1,2$, and $\pi(z)$ is homotopic to $I_n$ as $(\lambda_N\varepsilon,h_{\varepsilon,N}r,\lambda_N)$-invertibles in $M_n(\mathbb{C})$, where $\pi:\widetilde{A_{\Delta_1}}\cap\widetilde{A_{\Delta_2}}\rightarrow\mathbb{C}$ is the quotient homomorphism. Hence 
\begin{itemize}
\item $[z]\in K_1^{\lambda_N\varepsilon,h_{\varepsilon,N}r,\lambda_N}(A_{\Delta_1}\cap A_{\Delta_2})$, 
\item $j_{1,2*}([z])=y_1$ in $K_1^{\lambda_N\varepsilon,h_{\varepsilon,N}r,\lambda_N}(A_{\Delta_1})$, and 
\item $j_{2,1*}([z])=y_2$ in $K_1^{\lambda_N\varepsilon,h_{\varepsilon,N}r,\lambda_N}(A_{\Delta_2})$.
\end{itemize}

%
%

\end{proof}

Next, we want to define a boundary map \[\partial^{\varepsilon,r,N}:K_1^{\varepsilon,s,N}(A)\rightarrow K_0^{\lambda^{\mathcal{D}}_N\varepsilon,h^{\mathcal{D}}_{\varepsilon,N}s,\lambda^{\mathcal{D}}_N}(A_{\Delta_1}\cap A_{\Delta_2})\] for an appropriate control pair $(\lambda^{\mathcal{D}},h^{\mathcal{D}})$ depending only on the coercity $c$.

\begin{lem} \label{MVboundlem}
For every $c>0$, there exists a control pair $(\lambda,h)$ such that for any filtered $SQ_p$ algebra $A$, any $(s,c)$-controlled Mayer-Vietoris pair $(A_{\Delta_1},A_{\Delta_2})$ for $A$, any $N\geq 1$, $0<\varepsilon<\frac{1}{20\lambda_N}$, and $0<r\leq\frac{s}{h_{\varepsilon,N}}$, if $u\in GL_n^{\varepsilon,r,N}(\tilde{A})$ and $v\in GL_m^{\varepsilon,r,N}(\tilde{A})$ are such that $u-I_n\in M_n(A)$ and $v-I_n\in M_n(A)$, and 
$w_i\in GL_{n+m}^{\varepsilon,r,N}(\widetilde{A_{\Delta_i}})$ for $i=1,2$ are such that 
$||\diag(u,v)-w_1w_2||<\varepsilon$, then letting $w_i'$ be an $(\varepsilon,r,N)$-inverse for $w_i$, there exists a $(\lambda_N\varepsilon,h_{\varepsilon,N}r,\lambda_N)$-idempotent $e\in M_{n+m}(\widetilde{A_{\Delta_1}}\cap \widetilde{A_{\Delta_2}})$ such that \[\max(||e-w_1'\diag(I_n,0)w_1||,||e-w_2\diag(I_n,0)w_2'||)<\lambda_N\varepsilon,\] and $\diag(\pi(e),0)$ is homotopic to $\diag(I_n,0)$ as $(\lambda_N\varepsilon,h_{\varepsilon,N}r,\lambda_N)$-idempotents in $M_{2(n+m)}(\mathbb{C})$, where $\pi:\widetilde{A_{\Delta_1}}\cap\widetilde{A_{\Delta_2}}\rightarrow\mathbb{C}$ is the quotient homomorphism.
\end{lem} 

\begin{proof}
Let $u'$ be an $(\varepsilon,r,N)$-inverse for $u$ with $u'-I_n\in M_n(A)$, and let $v'$ be an $(\varepsilon,r,N)$-inverse for $v$ with $v'-I_n\in M_n(A)$. Then 
\begin{align*} &\quad\; ||w_1'\diag(I_n,0)w_1-w_1'\diag(u,v)\diag(I_n,0)\diag(u',v')w_1|| \\ &=||w_1'\diag(I_n-uu',0)w_1|| \\ &<N^2\varepsilon, \end{align*} 
Since $||\diag(u,v)-w_1w_2||<\varepsilon$, we have \begin{align*} ||w_1'\diag(u,v)-w_2|| &\leq||w_1'(\diag(u,v)-w_1w_2)||+||(w_1'w_1-1)w_2||\\ &<2N\varepsilon. \end{align*} Thus 
\begin{align*}
&\quad\; ||w_1'\diag(I_n,0)w_1-w_2\diag(I_n,0)w_2'|| \\ 
&< N^2\varepsilon+||w_1'\diag(u,v)\diag(I_n,0)\diag(u',v')w_1-w_2\diag(I_n,0)w_2'|| \\
&\leq N^2\varepsilon+||(w_1'\diag(u,v)-w_2)\diag(I_n,0)\diag(u',v')w_1|| \\ &\quad\quad +||w_2\diag(I_n,0)(\diag(u',v')w_1-w_2')|| \\
&< N^2\varepsilon + 2N^3\varepsilon + N||\diag(u',v')w_1-w_2'|| \\
&\leq (N^2+2N^3)\varepsilon + N(||\diag(u',v')w_1(w_2-w_1'\diag(u,v))w_2'|| \\ &\quad\quad +||\diag(u',v')w_1(1-w_2w_2')|| +||\diag(u',v')(w_1w_1'-1)\diag(u,v)w_2'|| \\ &\quad\quad +||\diag(u'u-1,v'v-1)w_2'||) \\
&< (N^2+2N^3)\varepsilon + N(2N^4+N^2+N^3+N)\varepsilon \\
&= (2N^5+N^4+3N^3+2N^2)\varepsilon.
\end{align*}
There exists $e\in M_{n+m}(\widetilde{A_{\Delta_1}}\cap \widetilde{A_{\Delta_2}})$ such that \[||e-w_1'\diag(I_n,0)w_1||<\biggl(2c+\frac{1}{2}\biggr)(2N^5+N^4+3N^3+2N^2)\varepsilon\] and \[||e-w_2\diag(I_n,0)w_2'||)<\biggl(2c+\frac{1}{2}\biggr)(2N^5+N^4+3N^3+2N^2)\varepsilon.\]
By Lemma \ref{normestlem1} and Lemma \ref{simtohom2}, there exists a control pair $(\lambda,h)$ depending only on $c$ such that $e$ is a $(\lambda_N\varepsilon,h_{\varepsilon,N}r,\lambda_N)$-idempotent, 
and $\diag(\pi(e),0)$ is homotopic to $\diag(I_n,0)$ as $(\lambda_N\varepsilon,h_{\varepsilon,N}r,\lambda_N)$-idempotents in $M_{2(n+m)}(\mathbb{C})$.
\end{proof}

Given $u\in GL_n^{\varepsilon,r,N}(\tilde{A})$ with $u-I_n\in M_n(A)$, pick $v\in GL_m^{\varepsilon,r,N}(\tilde{A})$ with $v-I_n\in M_n(A)$ such that $\diag(u,v)$ is homotopic to $I_{n+m}$ as $(2\varepsilon,2r,2(N+\varepsilon))$-invertibles. For instance, by Lemma \ref{inversepairhomotopy}, we may pick $v\in GL_n^{\varepsilon,r,N}(\tilde{A})$ to be an $(\varepsilon,r,N)$-inverse for $u$. By Lemma \ref{MVtechlem}, there exists a control pair $(\lambda,h)$ depending only on the coercity $c$ such that, up to replacing $v$ by $\diag(v,I_k)$ for some integer $k$, there exist $(\lambda_N\varepsilon,h_{\varepsilon,N}r,\lambda_N)$-invertibles $w_i\in M_{n+m}(\widetilde{A_{\Delta_i}}\cap\tilde{A}_{h_{\varepsilon,N}r})$ such that $||\diag(u,v)-w_1w_2||<\lambda_N\varepsilon$.

By Lemma \ref{MVboundlem}, there exists a control pair $(\lambda',h')$ depending only on $c$, and a $((\lambda'\cdot\lambda)_N\varepsilon,(h'\cdot h)_{\varepsilon,N}s,(\lambda'\cdot\lambda)_N)$-idempotent $e\in M_{n+m}(\widetilde{A_{\Delta_1}}\cap \widetilde{A_{\Delta_2}})$ such that letting $w_i'$ be a $(\lambda_N\varepsilon,h_{\varepsilon,N}r,\lambda_N)$-inverse for $w_i$, we have \[\max(||e-w_1'\diag(I_n,0)w_1||,||e-w_2\diag(I_n,0)w_2'||)<(\lambda'\cdot\lambda)_N\varepsilon,\] and $\diag(\pi(e),0)$ is $((\lambda'\cdot\lambda)_N\varepsilon,(h'\cdot h)_{\varepsilon,N}r,(\lambda'\cdot\lambda)_N)$-homotopic to $\diag(I_n,0)$ in $M_{2(n+m)}(\mathbb{C})$. 

We would like to define the boundary map by $\partial([u])= [e]-[I_n]$ but in order for it to be well-defined (i.e., independent of the various choices made), we need to place $[e]-[I_n]$ in the appropriate quantitative $K_0$ group.

More precisely, we need to check that there is a control pair $(\lambda^{\mathcal{D}},h^{\mathcal{D}})$ (depending only on the coercity $c$) such that for $N\geq 1$, $0<\varepsilon<\frac{1}{20\lambda^{\mathcal{D}}_N}$ and $0<r\leq\frac{s}{h^{\mathcal{D}}_{\varepsilon,N}}$, the map $\partial:K_1^{\varepsilon,r,N}(A)\rightarrow K_0^{\lambda^{\mathcal{D}}_N\varepsilon,h^{\mathcal{D}}_{\varepsilon,N}r,\lambda^{\mathcal{D}}_N}(A_{\Delta_1}\cap A_{\Delta_2})$ given by $[u]\mapsto[e]-[I_n]$
\begin{enumerate}
\item does not depend on the choice of $e$ satisfying the conclusion of Lemma \ref{MVboundlem};
\item does not depend on the choice of $w_1,w_2$ satisfying the hypotheses of Lemma \ref{MVboundlem};
\item does not change upon replacing $u$ (resp. $v$) by $\diag(u,1)$ (resp. $\diag(v,1)$);
\item does not depend on the choice of $v\in GL_m^{\varepsilon,r,N}(A)$ such that $\diag(u,v)$ is homotopic to $I_{n+m}$ as $(2\varepsilon,2r,2(N+\varepsilon))$-invertibles;
\item only depends on the equivalence class of $u$.
\end{enumerate}

We also want the compositions
\[\minCDarrowwidth15pt \begin{CD}
K_1^{\varepsilon,r,N}(A_{\Delta_1})\oplus K_1^{\varepsilon,r,N}(A_{\Delta_2}) @>(j_{1*}-j_{2*})>> K_1^{\varepsilon,r,N}(A) @>\partial>> K_0^{\lambda^{\mathcal{D}}_N\varepsilon,h^{\mathcal{D}}_{\varepsilon,N}r,\lambda^{\mathcal{D}}_N}(A_{\Delta_1}\cap A_{\Delta_2})
\end{CD} \]
and
\[ \begin{CD}
K_1^{\varepsilon,r,N}(A) @>\partial>> K_0^{\lambda^{\mathcal{D}}_N\varepsilon,h^{\mathcal{D}}_{\varepsilon,N}r,\lambda^{\mathcal{D}}_N}(A_{\Delta_1}\cap A_{\Delta_2}) \\ @>(j_{1,2*},j_{2,1*})>> K_0^{\lambda^{\mathcal{D}}_N\varepsilon,h^{\mathcal{D}}_{\varepsilon,N}r,\lambda^{\mathcal{D}}_N}(A_{\Delta_1})\oplus K_0^{\lambda^\mathcal{D}_N\varepsilon,h^{\mathcal{D}}_{\varepsilon,N}r,\lambda^{\mathcal{D}}_N}(A_{\Delta_2})
\end{CD} \]

to be the zero maps.

Now we will address each of these points in turn. When there is a need to relax control by increasing the parameters, we will sometimes omit precise expressions of the parameters involved with the understanding that they increase in a controlled manner.

\begin{enumerate}
\item Suppose that $e_0$ and $e_1$ both satisfy the conclusion of Lemma \ref{MVboundlem}. Then $||e_0-e_1||<2\lambda_N\varepsilon$ so $e_0$ and $e_1$ are $(2\lambda_N^2\varepsilon,h_{\varepsilon,N}r,\lambda_N)$-homotopic by Lemma \ref{normestlem2}.

\item Suppose that $w_3\in GL_{n+m}^{\varepsilon,r,N}(\widetilde{A_{\Delta_1}})$ and $w_4\in GL_{n+m}^{\varepsilon,r,N}(\widetilde{A_{\Delta_2}})$ have the same properties as $w_1,w_2$ so $||\diag(u,v)-w_1w_2||<\lambda_N\varepsilon$ and $||\diag(u,v)-w_3w_4||<\lambda_N\varepsilon$. Let $w_i'$ be a $(\lambda_N\varepsilon,h_{\varepsilon,N}r,\lambda_N)$-inverse for $w_i$. Then $||w_1w_2-w_3w_4||<2\lambda_N\varepsilon$ and 
\begin{align*} &\quad\quad\quad\quad\; ||w_3'w_1-w_4w_2'|| \\ &\quad\quad\quad\leq||w_3'(w_1w_2-w_3w_4)w_2'||+||w_3'w_1(1-w_2w_2')||+||(w_3'w_3-1)w_4w_2'||\\ &\quad\quad\quad <4\lambda_N^3\varepsilon.\end{align*} 
By the complete intersection approximation property, there exists $y\in M_{n+m}(\widetilde{A_{\Delta_1,}}_{h_{\varepsilon,N}r}\cap \widetilde{A_{\Delta_2,}}_{h_{\varepsilon,N}r})$ with $||y-w_3'w_1||<4\lambda_N^3(2c+\frac{1}{2})\varepsilon$ and $||y-w_4w_2'||<4\lambda_N^3(2c+\frac{1}{2})\varepsilon$. 
Thus \[\max(||w_3y-w_1||,||yw_2-w_4||)<\biggl(4\lambda_N^4\biggl(2c+\frac{1}{2}\biggr)+\lambda_N^2\biggr)\varepsilon.\] 

Similarly, there exists $z\in M_{n+m}(\widetilde{A_{\Delta_1,}}_{h_{\varepsilon,N}r}\cap \widetilde{A_{\Delta_2,}}_{h_{\varepsilon,N}r})$ such that \[\max(||z-w_1'w_3||,||z-w_2w_4'||)<4\lambda_N^3\biggl(2c+\frac{1}{2}\biggr)\varepsilon\] so \[\max(||zw_3'-w_1'||,||zw_4-w_2||)<\biggl(4\lambda_N^4\biggl(2c+\frac{1}{2}\biggr)+\lambda_N^2\biggr)\varepsilon.\] Moreover, \begin{align*} &\quad\quad\; ||yz-1|| \\ &\quad\leq||(y-w_3'w_1)z||+||w_3'w_1(z-w_1'w_3)||+||w_3'(w_1w_1'-1)w_3||+||w_3'w_3-1||\\ &\quad<\biggl(4\lambda_N^3\biggl(2c+\frac{1}{2}\biggr)\biggl(\lambda_N^2+4\lambda_N^3\biggl(2c+\frac{1}{2}\biggr)\biggr)+\lambda_N^2\biggl(4\lambda_N^3\biggl(2c+\frac{1}{2}\biggr)\biggr)+\lambda_N^3+\lambda_N\biggr)\varepsilon
\\ &\quad=\biggl(16\lambda_N^6\biggl(2c+\frac{1}{2}\biggr)^2+8\lambda_N^5\biggl(2c+\frac{1}{2}\biggr)+\lambda_N^3+\lambda_N\biggr)\varepsilon, \end{align*} and similarly for $||zy-1||$, so $(y,z)$ is a quasi-inverse pair.

If $e$ is the quasi-idempotent element obtained from $w_1$ and $w_2$, then with respect to an appropriate control pair $(\lambda'',h'')$ depending only on the coercity $c$, $yez$ is the quasi-idempotent element obtained from $w_3$ and $w_4$, and $[yez]=[e]$ in $K_0^{\lambda''_N\varepsilon,h''_{\varepsilon,N}r,\lambda''_N}(A_{\Delta_1}\cap A_{\Delta_2})$.

\item Replacing $u$ by $\diag(u,1)$, and letting $w_1,w_2$ be such that \[||\diag(u,v)-w_1w_2||<\lambda_N\varepsilon,\] we have $||\diag(u,v,I_2)-\diag(w_1w_2,I_2)||<\lambda_N\varepsilon$. Now \[U\diag(u,v,I_2)U^{-1}=\diag(u,1,v,1)\] for some permutation matrix $U$ so \[||\diag(u,1,v,1)-U\diag(w_1w_2,I_2)U^{-1}||<\lambda_N\varepsilon.\]

If $||e-w_1'\diag(I_n,0)w_1||<(\lambda'\cdot\lambda)_N\varepsilon$, then \[\qquad||\diag(e,1,0)-\diag(w_1',I_2)\diag(I_n,0_m,1,0)\diag(w_1,I_2)||<(\lambda'\cdot\lambda)_N\varepsilon.\] But $\diag(I_n,0_m,1,0)=U^{-1}\diag(I_{n+1},0)U$ so \[\qquad||\diag(e,1,0)-\diag(w_1',I_2)U^{-1}\diag(I_{n+1},0)U\diag(w_1,I_2)||<(\lambda'\cdot\lambda)_N\varepsilon.\] Similarly, we have \[\qquad||\diag(e,1,0)-\diag(w_2,I_2)U^{-1}\diag(I_{n+1},0)U\diag(w_2',I_2)||<(\lambda'\cdot\lambda)_N\varepsilon.\] Thus $\partial([\diag(u,1)])=[\diag(e,1,0)]-[I_{n+1}]=[e]-[I_n]=\partial([u])$. Similarly, one sees that replacing $v$ by $\diag(v,1)$ does not change $\partial([u])$.

\item Let $v_0\in GL_m^{\varepsilon,r,N}(\tilde{A})$ and $v_1\in GL_k^{\varepsilon,r,N}(\tilde{A})$ be such that $\diag(u,v_0)$ is homotopic to $I_{n+m}$ as $(2\varepsilon,2r,2(N+\varepsilon))$-invertibles, and $\diag(u,v_1)$ is homotopic to $I_{n+k}$ as $(2\varepsilon,2r,2(N+\varepsilon))$-invertibles. Assume that $m\geq k$. Then $\diag(u,v_1,I_{m-k})$ is homotopic to $I_{n+m}$ as $(2\varepsilon,2r,2(N+\varepsilon))$-invertibles. Let $(U_t)_{t\in[0,1]}$ be a homotopy of $(2\varepsilon,2r,2(N+\varepsilon))$-invertibles between $\diag(u,v_0)$ and $\diag(u,v_1,I_{m-k})$. We may assume that $\pi(U_t)=I_{n+m}$ for all $t$. Then we may regard $U=(U_t)$ as a $(2\varepsilon,2r,2(N+\varepsilon))$-invertible in $C([0,1],M_{n+m}(\tilde{A}_{2r}))$ with $\pi(U)=I_{n+m}$. Moreover, $U$ is homotopic to 1 as $(2\varepsilon,2r,2(N+\varepsilon))$-invertibles. By Lemma \ref{MVtechlem}, there exist $l\in\mathbb{N}$ and invertibles $W_i \in C([0,1],M_{n+m+l}(\widetilde{A_{\Delta_i}}))$ with $||\diag(U,I_l)-W_1W_2||<\lambda_N\varepsilon$. We obtain a quasi-idempotent $E\in C([0,1],M_{n+m+l}(\tilde{A}))$ by Lemma \ref{MVboundlem}. Using $v_0$ in the definition of $\partial$ yields $[E_0]-[I_n]$ while using $v_1$ in the definition yields $[E_1]-[I_n]$, but $[E_0]-[I_n]=[E_1]-[I_n]$.

\item Suppose that $[u_0]=[u_1]$ in $K_1^{\varepsilon,r,N}(A)$. Then up to stabilization, we may assume that $u_0$ and $u_1$ are homotopic as $(4\varepsilon,2r,4N)$-invertibles in $M_n(\tilde{A})$. Let $(u_t)_{t\in[0,1]}$ be such a homotopy. We may assume that $\pi(u_t)=I_n$ for all $t$. Then we may regard $u=(u_t)$ as a $(4\varepsilon,2r,4N)$-invertible in $C([0,1],M_n(\tilde{A}_{2r}))$. Let $u'$ be a $(4\varepsilon,2r,4N)$-inverse for $u$ with $\pi(u')=I_n$. Up to stabilization, there exist quasi-invertibles $w_i$ in $C([0,1],M_n(\widetilde{A_{\Delta_i}}))$ such that $||\diag(u,u')-w_1w_2||<\lambda_{4N}\varepsilon$. Then there is a $(\lambda'''_N\varepsilon,h'''_{\varepsilon,N}r,\lambda'''_N)$-idempotent $e\in C([0,1],M_n(\widetilde{A_{\Delta_1}}_{,h'''_{\varepsilon,N} r}\cap\widetilde{A_{\Delta_2}}_{,h'''_{\varepsilon,N} r}))$ such that \[\qquad\max(||e-w_1'\diag(I_n,0)w_1||,||e-w_2\diag(I_n,0)w_2'||)<\lambda'''_N\varepsilon.\] Now $(e_t)_{t\in[0,1]}$ is a homotopy of $(\lambda'''_N\varepsilon,h'''_{\varepsilon,N}r,\lambda'''_N)$-idempotents, so $\partial([u_0])=[e_0]-[I_n]=[e_1]-[I_n]=\partial([u_1])$.
\end{enumerate}

Consider the composition
\[\minCDarrowwidth15pt \begin{CD}
K_1^{\varepsilon,r,N}(A_{\Delta_1})\oplus K_1^{\varepsilon,r,N}(A_{\Delta_2}) @>(j_{1*}-j_{2*})>> K_1^{\varepsilon,r,N}(A) @>\partial>> K_0^{\lambda^{\mathcal{D}}_N\varepsilon,h^{\mathcal{D}}_{\varepsilon,N}r,\lambda^{\mathcal{D}}_N}(A_{\Delta_1}\cap A_{\Delta_2}).
\end{CD} \]
Suppose that $u\in M_n(\widetilde{A_{\Delta_1}})$ and $v\in M_n(\widetilde{A_{\Delta_2}})$ are $(\varepsilon,r,N)$-invertibles. In the preliminary definition of $\partial([u])$, we may take $w_1=\diag(u,u')$, where $u'$ is an $(\varepsilon,r,N)$-inverse for $u$, and $w_2=I_{2n}$. Then $\partial([u])=[e_0]-[I_n]$, where $e_0$ is a $((\lambda'\cdot\lambda)_N\varepsilon,(h'\cdot h)_{\varepsilon,N}r,(\lambda'\cdot\lambda)_N)$-idempotent in $M_{2n}(\tilde{A})$ such that $||e_0-\diag(I_n,0)||<(\lambda'\cdot\lambda)_N\varepsilon$. Similarly, $\partial([v])=[e_1]-[I_n]$, where $e_1$ is an $((\lambda'\cdot\lambda)_N\varepsilon,(h'\cdot h)_{\varepsilon,N}r,(\lambda'\cdot\lambda)_N)$-idempotent in $M_{2n}(\tilde{A})$ such that $||e_1-\diag(I_n,0)||<(\lambda'\cdot\lambda)_N\varepsilon$. Now $||e_0-e_1||<2(\lambda'\cdot\lambda)_N\varepsilon$ so by Lemma \ref{normestlem2}, $e_0$ and $e_1$ are homotopic as $(2(\lambda'\cdot\lambda)_N^2\varepsilon,(h'\cdot h)_{\varepsilon,N}r,(\lambda'\cdot\lambda)_N)$-idempotents. Thus there exists a control pair $(\lambda^{\mathcal{D}},h^{\mathcal{D}})$ depending only on the coercity $c$ such that $\partial([u]-[v])=0$ in $K_0^{\lambda^{\mathcal{D}}_N\varepsilon,h^{\mathcal{D}}_{\varepsilon,N}r,\lambda^{\mathcal{D}}_N}(A_{\Delta_1}\cap A_{\Delta_2})$.

Finally, consider the composition
\[ \begin{CD}
K_1^{\varepsilon,r,N}(A) @>\partial>> K_0^{\lambda^{\mathcal{D}}_N\varepsilon,h^{\mathcal{D}}_{\varepsilon,N}r,\lambda^{\mathcal{D}}_N}(A_{\Delta_1}\cap A_{\Delta_2}) \\ @>(j_{1,2*},j_{2,1*})>> K_0^{\lambda^{\mathcal{D}}_N\varepsilon,h^{\mathcal{D}}_{\varepsilon,N}r,\lambda^{\mathcal{D}}_N}(A_{\Delta_1})\oplus K_0^{\lambda^\mathcal{D}_N\varepsilon,h^{\mathcal{D}}_{\varepsilon,N}r,\lambda^{\mathcal{D}}_N}(A_{\Delta_2}).
\end{CD} \]

The preliminary definition of the boundary map yields $j_{1,2*}(\partial([u]))=0$ in $K_0^{\lambda^{\mathcal{D}}_N\varepsilon,h^{\mathcal{D}}_{\varepsilon,N}r,\lambda^{\mathcal{D}}_N}(A_{\Delta_1})$ and $j_{2,1*}(\partial([u]))=0$ in $K_0^{\lambda^{\mathcal{D}}_N\varepsilon,h^{\mathcal{D}}_{\varepsilon,N}r,\lambda^{\mathcal{D}}_N}(A_{\Delta_2})$ when $(\lambda^{\mathcal{D}},h^{\mathcal{D}})$ is sufficiently large, so that the composition is the zero map. 

Now we give a formal definition of the boundary map in terms of a control pair $(\lambda^{\mathcal{D}},h^{\mathcal{D}})$ making all the above hold.

\begin{defn}
Let $A$ be a filtered $SQ_p$ algebra, and let $(A_{\Delta_1},A_{\Delta_2})$ be an $(s,c)$-controlled Mayer-Vietoris pair for $A$. Let $(\lambda,h)$ be the control pair from Lemma \ref{MVtechlem}, and let $(\lambda',h')$ be the control pair from Lemma \ref{MVboundlem}. 
Given $[u]\in K_1^{\varepsilon,r,N}(A)$, where $N\geq 1$, $0<\varepsilon<\frac{1}{20\lambda^{\mathcal{D}}_N}$, $0<r\leq\frac{s}{h^{\mathcal{D}}_{\varepsilon,N}}$, and $u\in GL_n^{\varepsilon,r,N}(\tilde{A})$ with $u-I_n\in M_n(A)$, 
\begin{enumerate}
\item find $v\in GL_m^{\varepsilon,r,N}(\tilde{A})$ with $v-I_n\in M_n(A)$ such that $\diag(u,v)$ is homotopic to $I_{n+m}$ as $(2\varepsilon,2r,2(N+\varepsilon))$-invertibles,
\item let $w_i$ ($i=1,2$) be an $(\lambda_N\varepsilon,h_{\varepsilon,N}r,\lambda_N)$-invertible in $M_{n+m}(\widetilde{A_{\Delta_i}})$ such that \[||\diag(u,v)-w_1w_2||<\lambda_N\varepsilon,\] 
\item let $w_i'$ be a $(\lambda_N\varepsilon,h_{\varepsilon,N}r,\lambda_N)$-inverse for $w_i$, and 
\item let $e\in M_{n+m}(\widetilde{A_{\Delta_1}}\cap\widetilde{A_{\Delta_2}})$ be such that \[\max(||e-w_1'\diag(I_n,0)w_1||,||e-w_2\diag(I_n,0)w_2'||)<(\lambda'\cdot\lambda)_N\varepsilon\] and $\diag(\pi(e),0)$ is $((\lambda'\cdot\lambda)_N\varepsilon,(h'\cdot h)_{\varepsilon,N}r,(\lambda'\cdot\lambda)_N)$-homotopic to $\diag(I_n,0)$ in $M_{2(n+m)}(\mathbb{C})$.
\end{enumerate}
Define $\partial:K_1^{\varepsilon,r,N}(A)\rightarrow K_0^{\lambda^{\mathcal{D}}_N\varepsilon,h^{\mathcal{D}}_{\varepsilon,N} r,\lambda^{\mathcal{D}}_N}(A_{\Delta_1}\cap A_{\Delta_2})$ by $[u]\mapsto[e]-[I_n]$.
\end{defn}

\begin{prop}
For every $c>0$, there exists a control pair $(\lambda,h)$ such that for any filtered $SQ_p$ algebra $A$, any $(s,c)$-controlled Mayer-Vietoris pair $(A_{\Delta_1},A_{\Delta_2})$ for $A$, and $N\geq 1$, $0<\varepsilon<\frac{1}{20\lambda_N}$, and $0<r\leq\frac{s}{h_{\varepsilon,N}}$,

\begin{enumerate}
\item if $u\in M_n(\tilde{A})$ is an $(\varepsilon,r,N)$-invertible such that $\partial([u])=[e]-[I_n]=0$ in $K_0^{\lambda^{\mathcal{D}}_N\varepsilon,h^{\mathcal{D}}_{\varepsilon,N}r,\lambda^{\mathcal{D}}_N}(A_{\Delta_1}\cap A_{\Delta_2})$, then $j_{1*}(y_1)-j_{2*}(y_2)=[u]$ in $K_0^{\lambda_N\varepsilon,h_{\varepsilon,N}r,\lambda_N}(A)$ for some $y_i\in K_1^{\lambda_N\varepsilon,h_{\varepsilon,N}r,\lambda_N}(A_{\Delta_i})$;

\item if $[e]-[I_n]\in K_0^{\lambda^{\mathcal{D}}_N\varepsilon,h^{\mathcal{D}}_{\varepsilon,N}r,\lambda^{\mathcal{D}}_N}(A_{\Delta_1}\cap A_{\Delta_2})$ is such that $[e]-[I_n]=0$ in $K_0^{\lambda^{\mathcal{D}}_N\varepsilon,h^{\mathcal{D}}_{\varepsilon,N}r,\lambda^{\mathcal{D}}_N}(A_{\Delta_i})$ for $i=1,2$, then $\partial(y)=[e]-[I_n]$ for some $y\in K_1^{\lambda_N\varepsilon,h_{\varepsilon,N}r,\lambda_N}(A_{\Delta_1}\cap A_{\Delta_2})$.
\end{enumerate}
\end{prop}

\begin{proof}
For (i), suppose that $u\in M_n(\tilde{A})$ is an $(\varepsilon,r,N)$-invertible such that $\partial([u])=[e]-[I_n]=0$ in $K_0^{\lambda^{\mathcal{D}}_N\varepsilon,h^{\mathcal{D}}_{\varepsilon,N}r,\lambda^{\mathcal{D}}_N}(A_{\Delta_1}\cap A_{\Delta_2})$. Up to stabilization, we may assume that $e$ is $(4\lambda^{\mathcal{D}}_N\varepsilon,h^{\mathcal{D}}_{\varepsilon,N}r,4\lambda^{\mathcal{D}}_N+1)$-homotopic to $\diag(I_n,0)$ in $M_{2n}(\widetilde{A_{\Delta_1}}\cap\widetilde{A_{\Delta_2}})$. By Proposition \ref{homtosim2} and Lemma \ref{idemliphom}, there exists a $(\lambda''_N\varepsilon,h''_{\varepsilon,N}r,\lambda''_N)$-inverse pair $(v,v')$ in $M_{2n}(\widetilde{A_{\Delta_1}}\cap\widetilde{A_{\Delta_2}})$ such that \[||e-v\diag(I_n,0)v'||<\lambda''_N\varepsilon.\] 

Since $\max(||e-w_1'\diag(I_n,0)w_1||,||e-w_2\diag(I_n,0)w_2'||)<\lambda^{\mathcal{D}}_N\varepsilon$, we have 
\begin{align*} ||w_1'\diag(I_n,0)w_1-v\diag(I_n,0)v'||&<(\lambda^{\mathcal{D}}_N+\lambda''_N)\varepsilon, \\
||w_2\diag(I_n,0)w_2'-v\diag(I_n,0)v'||&<(\lambda^{\mathcal{D}}_N+\lambda''_N)\varepsilon.
\end{align*} 
By successively relaxing control if necessary, we see that there exists a control pair $(\lambda''',h''')$ depending only on the coercity $c$ such that 
\begin{itemize}
\item $||\diag(I_n,0)w_1v-w_1v\diag(I_n,0)||<\lambda'''_N\varepsilon$;
\item $||v'w_2\diag(I_n,0)-\diag(I_n,0)v'w_2||<\lambda'''_N\varepsilon$;
\item $\max(||w_1v-\diag(u_1,v_1)||,||v'w_2-\diag(u_2,v_2)||)<\lambda'''_N\varepsilon$, where $u_i,v_i$ are quasi-invertible elements in $M_n(\widetilde{A_{\Delta_i}})$ for $i=1,2$;
\item $||\diag(u,u')-\diag(u_1u_2,v_1v_2)||<\lambda'''_N\varepsilon$, where $u'$ is an $(\varepsilon,r,N)$-inverse for $u$;
\item $||u-u_1u_2||<\lambda'''_N\varepsilon$; 
\item $u$ and $u_1u_2$ are homotopic as $(\lambda'''_N\varepsilon,h'''_{\varepsilon,N}s,\lambda'''_N)$-invertibles; 
\item $[u]=[u_1]+[u_2]$ in $K_1^{\lambda'''_N\varepsilon,h'''_{\varepsilon,N}r,\lambda'''_N}(A)$.
\end{itemize}

For (ii), suppose that $[e]-[I_n]\in K_0^{\lambda^{\mathcal{D}}_N\varepsilon,h^{\mathcal{D}}_{\varepsilon,N}r,\lambda^{\mathcal{D}}_N}(A_{\Delta_1}\cap A_{\Delta_2})$ is such that $[e]-[I_n]=0$ in $K_0^{\lambda^{\mathcal{D}}_N\varepsilon,h^{\mathcal{D}}_{\varepsilon,N}r,\lambda^{\mathcal{D}}_N}(A_{\Delta_i})$ for $i=1,2$. Up to stabilization, we may assume that $e$ is $(4\lambda^{\mathcal{D}}_N\varepsilon,h^{\mathcal{D}}_{\varepsilon,N}r,4\lambda^{\mathcal{D}}_N+1)$-homotopic to $\diag(I_n,0)$ in $M_{2n}(\widetilde{A_{\Delta_i}})$ for $i=1,2$. By successively relaxing control if necessary, we see that there exists a control pair $(\lambda'',h'')\geq(\lambda^{\mathcal{D}},h^{\mathcal{D}})$ depending only on the coercity $c$ such that
\begin{itemize}
\item $\max(||e-w_1'\diag(I_n,0)w_1||,||e-w_2\diag(I_n,0)w_2'||)<\lambda''_N\varepsilon$, where $(w_i,w_i')$ is some $(\lambda''_N\varepsilon,h''_{\varepsilon,N}s,\lambda''_N)$-inverse pair in $M_{2n}(\widetilde{A_{\Delta_i}})$;
\item $||w_1w_2\diag(I_n,0)-\diag(I_n,0)w_1w_2||<3(\lambda''_N)^3\varepsilon$;
\item $||w_1w_2-\diag(u,v)||<3(\lambda''_N)^3\varepsilon$, where $u,v$ are $(\lambda''_N\varepsilon,h''_{\varepsilon,N}s,\lambda''_N)$-invertibles in $M_n(A)$;
\item $\partial([u])=[e]-[I_n]$ in $K_0^{\lambda''_N\varepsilon,h''_{\varepsilon,N}r,\lambda''_N}(A_{\Delta_1}\cap A_{\Delta_2})$.
\end{itemize}
\end{proof}


We summarize the results of this section as follows:

\begin{thm}
For every $c>0$, there exists a control pair $(\lambda,h)$ such that for any filtered $SQ_p$ algebra $A$ and any $(s,c)$-controlled Mayer-Vietoris pair $(A_{\Delta_1},A_{\Delta_2})$ for $A$, we have the following $(\lambda,h)$-exact sequence of order $s$:
\[
\begindc{\commdiag}[100]		
\obj(0,5)[1a]{$\mathcal{K}_1(A_{\Delta_1}\cap A_{\Delta_2})$}		
\obj(15,5)[1b]{$\mathcal{K}_1(A_{\Delta_1})\oplus\mathcal{K}_1(A_{\Delta_2})$}
\obj(30,5)[1c]{$\mathcal{K}_1(A)$}

\obj(0,0)[2a]{$\mathcal{K}_0(A)$}	
\obj(15,0)[2b]{$\mathcal{K}_0(A_{\Delta_1})\oplus\mathcal{K}_0(A_{\Delta_2})$}
\obj(30,0)[2c]{$\mathcal{K}_0(A_{\Delta_1}\cap A_{\Delta_2})$}

\mor{1a}{1b}{$(j_{1,2*},j_{2,1*})$}	\mor{1b}{1c}{$j_{1*}-j_{2*}$}

\mor{2c}{2b}{$(j_{1,2*},j_{2,1*})$}	\mor{2b}{2a}{$j_{1*}-j_{2*}$}	\mor{1c}{2c}{$\partial$}
\enddc
\]
\end{thm}

In view of the observations after Remark \ref{ciarem}, one can also obtain the corresponding controlled exact sequence for the suspensions by considering the controlled Mayer-Vietoris pair $(SA_{\Delta_1},SA_{\Delta_2})$ for $SA$.

\section{Remarks on the $C^*$-algebra case}

Finally, we make some remarks about the case where $A$ is a filtered $C^*$-algebra. In this case, both our definition of the quantitative $K$-theory groups and the definition in \cite{OY15} are applicable, and we shall briefly explain (without detailed proofs) that the two definitions give us essentially the same information.

In \cite{OY15}, the quantitative $K$-theory groups are defined by equivalence relations similar to ours but in terms of quasi-projections and quasi-unitaries instead of quasi-idempotents and quasi-invertibles.

\begin{defn} Let $A$ be a filtered $C^*$-algebra. For $0<\varepsilon<\frac{1}{20}$ and $r>0$,
\begin{enumerate}
\item an element $p\in A$ is called an $(\varepsilon,r)$-projection if $p\in A_r$, $p^*=p$, and $||p^2-p||<\varepsilon$;
\item if $A$ is unital, an element $u\in A$ is called an $(\varepsilon,r)$-unitary if $u\in A_r$, and $\max(||uu^*-1||,||u^*u-1||)<\varepsilon$.
\end{enumerate}
\end{defn}

It is straightforward to see that every $(\varepsilon,r)$-projection is an $(\varepsilon,r,1+\varepsilon)$-idempotent, and every $(\varepsilon,r)$-unitary is an $(\varepsilon,r,1+\varepsilon)$-invertible \cite[Remark 1.4]{OY15}. Writing $K_*^{'\varepsilon,r}(A)$ for the quantitative $K$-theory groups defined in \cite{OY15}, we thus have canonical homomorphisms $\phi_*:K_*^{'\varepsilon,r}(A)\rightarrow K_*^{\varepsilon,r,N}(A)$ for $0<\varepsilon<\frac{1}{20}$, $r>0$, and $N\geq 1+\varepsilon$.

Given an $(\varepsilon,r,N)$-idempotent $e$ in $A$, by considering
\begin{align*}
p_1 &= Q((2e^*-1)(2e-1))eR((2e^*-1)(2e-1)), \\
p_2 &= \frac{p_1+p_1^*}{2},
\end{align*}
where $Q(t)$ and $R(t)$ are polynomials such that $||(t+1)^{1/2}-Q(t)||<\frac{\varepsilon}{6N^5}$ and $||(t+1)^{-1/2}-R(t)||<\frac{\varepsilon}{6N^5}$ on $[0,(2N+1)^2]$, one can show the existence of a control pair $(\lambda,h)$ such that $\diag(e,0)$ and $\diag(p_2,0)$ are homotopic as $(\lambda_N\varepsilon,h_{\varepsilon,N}r,\lambda_N)$-idempotents. Since $p_2$ is a quasi-projection, we get a homomorphism $K_0^{\varepsilon,r,N}(A)\rightarrow K_0^{'\lambda_N\varepsilon,h_{\varepsilon,N}r}(A)$, which is a controlled inverse for $\phi_0$.

In the odd case, we have the following analog of polar decomposition:
\begin{lem}\cite[Lemma 2.4]{OY}
There exists a control pair $(\lambda,h)$ such that for any unital filtered $C^*$-algebra $A$ and any $(\varepsilon,r,N)$-invertible $x\in A$, there exists a positive $(\lambda_N\varepsilon,h_{\varepsilon,N}r,\lambda_N)$-invertible $y\in A$ and a $(\lambda_N\varepsilon,h_{\varepsilon,N}r)$-unitary $u\in A$ such that $|||x|-y||<\lambda_N\varepsilon$ and $||x-uy||<\lambda_N\varepsilon$. Moreover, we can choose $u$ and $y$ such that
\begin{itemize}
\item there exists a real polynomial $Q$ with $Q(1)=1$ such that $u=xQ(x^*x)$ and $y=x^*xQ(x^*x)$;
\item $y$ has a positive $(\lambda_N\varepsilon,h_{\varepsilon,N}r,\lambda_N)$-inverse;
\item if $x$ is homotopic to 1 as $(\varepsilon,r,N)$-invertibles, then $u$ is homotopic to 1 as $(\lambda_N\varepsilon,h_{\varepsilon,N}r)$-unitaries.
\end{itemize}
\end{lem}

By applying Lemma \ref{normestlem1} and using an appropriate polynomial approximation of $t\mapsto\exp(t\log y)$, one can find a control pair $(\lambda',h')\geq(\lambda,h)$ such that $x\sim uy\sim u$ as $(\lambda_N'\varepsilon,h_{\varepsilon,N}'r,\lambda_N')$-invertibles. This gives us a controlled inverse for $\phi_1$.

In summary, we have a controlled isomorphism between $\mathcal{K}'_*(A)$ and $\mathcal{K}_*(A)$ when $A$ is a filtered $C^*$-algebra.

\bibliographystyle{plain}
\bibliography{mybib}
\end{document}